\pdfoutput=1
\documentclass[11pt,a4paper]{article}

\usepackage[utf8]{inputenc}
\usepackage[T1]{fontenc}
\usepackage{amsmath}
\usepackage{amsfonts}
\usepackage{amssymb}
\usepackage{amsthm}
\usepackage{thmtools}
\usepackage{enumerate}
\usepackage{mathtools} 
\usepackage{xfrac}
\usepackage{lmodern}
\usepackage[nottoc]{tocbibind} 
\usepackage{hyperref} 
\usepackage[capitalise]{cleveref} 
\usepackage{graphicx}
\usepackage[dvipsnames]{xcolor} 
\usepackage[english]{babel}
\usepackage{slashed}
\usepackage{microtype}
\usepackage[autostyle=true]{csquotes}
\usepackage[all,cmtip]{xy}
\usepackage{tikz-cd}

\usepackage[a4paper,
            top=2.5cm,
            bottom=2.5cm,
            inner=3.5cm,
            outer=3.5cm]{geometry}

\usepackage[citestyle=alphabetic,
            bibstyle=alphabetic,
            backend=bibtex,
            url=true,
            doi=true,
            isbn=false,
            giveninits=true]{biblatex}
\renewcommand*{\bibfont}{\small}
\bibliography{Bibliography_AnaStructInj}
\AtEveryBibitem{%
  \ifentrytype{article}{
    \clearfield{url}%
    \clearfield{urldate}%
  }{}
  \ifentrytype{book}{
    \clearfield{url}%
    \clearfield{urldate}%
  }{}
  \ifentrytype{incollection}{
    \clearfield{url}%
    \clearfield{urldate}%
  }{}
}
\usepackage{xparse} 
\usepackage{enumitem}
\setlist[enumerate]{label={\upshape(\roman*)}}
\newlist{myenumi}{enumerate}{1}
\setlist[myenumi,1]{label=\upshape(\roman*)}
\newlist{myenuma}{enumerate}{1}
\setlist[myenuma,1]{label=\upshape(\alph*)}

\declaretheorem[name=Theorem,numberwithin=section]{thm}
\declaretheorem[name=Lemma,numberlike=thm]{lem}
\declaretheorem[name=Corollary,numberlike=thm]{cor}
\declaretheorem[name=Proposition,numberlike=thm]{prop}
\declaretheorem[name=Conjecture, numberlike=thm]{conjecture}
\declaretheorem[name=Question, numberlike=thm]{question}

\declaretheorem[name=Definition,numberlike=thm,style=definition,qed=\(\blacklozenge\)]{defn}
\declaretheorem[name=Example,numberlike=thm,style=definition,qed=\(\blacklozenge\)]{example}
\declaretheorem[name=Remark,numberlike=thm,style=definition,qed=\(\blacklozenge\)]{rem}

\declaretheorem[name=Properties,numberlike=thm,style=definition,qed=\(\blacklozenge\)]{properties}

\crefdefaultlabelformat{#2\textup{#1}#3}
\crefname{thm}{Theorem}{Theorems}
\crefname{lem}{Lemma}{Lemmas}
\crefname{defn}{Definition}{Definitions}
\crefname{prop}{Proposition}{Propositions}
\crefname{cor}{Corollary}{Corollaries}
\crefname{equation}{}{}

\newcommand{\N}{\mathbb{N}}
\newcommand{\IN}{\mathbb{N}}
\newcommand{\Z}{\mathbb{Z}}

\newcommand{\Q}{\mathbb{Q}}

\newcommand{\R}{\mathbb{R}}
\newcommand{\IR}{\mathbb{R}}
\newcommand{\C}{\mathbb{C}}
\newcommand{\IC}{\mathbb{C}}

\newcommand{\propagation}{\mathrm{prop}}
\newcommand{\supp}{\mathrm{supp}}
\newcommand{\lf}{\mathrm{lf}}
\newcommand{\ch}{\operatorname{ch}}
\newcommand{\Td}{\mathrm{Td}}

\newcommand{\EM}{\mathrm{EM}}
\newcommand{\BC}{\mathrm{BC}}

\newcommand{\cE}{\mathcal{E}}
\newcommand{\cF}{\mathcal{F}}
\newcommand{\cG}{\mathcal{G}}

\newcommand{\cU}{\mathcal{U}}

\newcommand{\cK}{\mathcal{K}}
\newcommand{\cB}{\mathcal{B}}

\newcommand{\Mult}{\mathcal{M}}

\newcommand{\KX}{\mathrm{KX}}
\newcommand{\SX}{\mathrm{SX}}
\newcommand{\KK}{\mathrm{K\!\, K}}

\newcommand{\RK}{\mathrm{RK}}

\newcommand{\Bfree}{\mathrm{B}}
\newcommand{\Efree}{\mathrm{E}}
\newcommand{\Eub}{\underline{\mathrm{E}}}
\newcommand{\Ct}{\mathrm{C}}
\newcommand*{\Cz}{\Ct_0}
\newcommand*{\Cc}{\Ct_{\mathrm{c}}}
\newcommand*{\Cb}{\Ct_{\mathrm{b}}}

\newcommand*{\Lin}{\cB}
\newcommand*{\Kom}{\cK}
\newcommand*{\U}{\cU}
\newcommand*{\elltwo}{\ell^2}

\DeclareMathOperator{\Hom}{Hom}

\newcommand*{\HZ}{\mathrm{H}}
\newcommand*{\HZX}{\mathrm{HX}}
\newcommand*{\coarsify}{\operatorname{c}}
\newcommand{\K}{\mathrm{K}}
\newcommand{\Ktop}{\mathrm{K}^\mathrm{top}}
\newcommand{\Strg}{\mathrm{S}}
\newcommand{\RStrg}{\mathrm{RS}}
\newcommand{\HRPlaceholder}{\mathcal{HR}}

\newcommand*{\StolzPos}{\mathrm{Pos}^{\mathrm{spin}}}
\newcommand*{\SpinBordism}{\Omega^{\mathrm{spin}}}
\newcommand*{\StolzRel}{\mathrm{R}^{\mathrm{spin}}}
\newcommand*{\Riem}{\mathcal{R}}
\newcommand*{\RiemPos}{\Riem^{+}}
\DeclareMathOperator{\inddiff}{inddiff}

\newcommand{\rmP}{\mathsf{P}}
\newcommand{\rmD}{\mathsf{D}}

\newcommand{\frakc}{\mathfrak{c}}
\newcommand{\frakcbar}{\bar{\mathfrak{c}}}

\newcommand*{\sHigCom}{\frakcbar}
\newcommand*{\sHigCor}{\frakc}
\newcommand*{\sHigComRed}{\bar{\mathfrak{c}}^{\mathrm{red}}}
\newcommand*{\sHigCorRed}{\mathfrak{c}^{\mathrm{red}}}
\newcommand*{\HigCom}{\Ct_{\mathrm{h}}}
\newcommand*{\HigCorSpace}{\partial_{\mathrm{h}}}

\newcommand*{\Ball}{\mathrm{B}}
\newcommand{\clBall}{\overline{\mathrm{B}}}

\newcommand*{\red}{\mathrm{red}}
\NewDocumentCommand{\Cstar}{}{\ensuremath{\mathrm{C}^*}}
\NewDocumentCommand{\textCstar}{}{\ensuremath{\mathrm{C}^*\!}}
\NewDocumentCommand{\CstarRed}{}{\Cstar_{\red}}

\NewDocumentCommand{\LSym}{}{\mathrm{L}}

\DeclareMathOperator{\ForgetEquiv}{\mathsf{F}}

\NewDocumentCommand{\D}{}{\mathop{}\!\mathrm{d}}
\NewDocumentCommand{\Dirac}{}{\slashed{D}}

\NewDocumentCommand \RoeSymbol {o} {
	\mathrm{C}^{\ast}
	\IfNoValueF{#1}{_{#1}}
}

\NewDocumentCommand \VanishSymbol {o} {
	\mathrm{N}^{\ast}
	\IfNoValueF{#1}{_{#1}}
}

\NewDocumentCommand \FiproSymbol {o} {
	\mathrm{E}^{\ast}
	\IfNoValueF{#1}{_{#1}}
}

\NewDocumentCommand \RoePlaceholder {o} {
\RoeSymbol[
	\IfNoValueF{#1}{#1,}
	\mathrm{?}
]
}

\NewDocumentCommand \Roe {o} {\RoeSymbol[#1]}
\NewDocumentCommand \Fipro {o} {\FiproSymbol[#1]}

\NewDocumentCommand \varRoe {o} {
	\RoeSymbol[\sim\IfNoValueF{#1}{,#1}]
}

\NewDocumentCommand \Loc {o} {
	\RoeSymbol[
		\IfNoValueF{#1}{#1,}
		\LSym
	]
}

\NewDocumentCommand \LocVanish {o} {
	\VanishSymbol[
		\IfNoValueF{#1}{#1,}
    \LSym
	]
}

\NewDocumentCommand \FiproLoc {o} {
	\FiproSymbol[
		\IfNoValueF{#1}{#1,}
		\LSym
	]
}

\NewDocumentCommand \varLoc {o} {
	\RoeSymbol[
		\sim,
		\IfNoValueF{#1}{#1,}
		\LSym
	]
}

\NewDocumentCommand \Locz {o} {
	\RoeSymbol[
		\IfNoValueF{#1}{#1,}
		\LSym,0
	]
}

\NewDocumentCommand \FiproLocz {o} {
	\FiproSymbol[
		\IfNoValueF{#1}{#1,}
		\LSym,0
	]
}

\NewDocumentCommand \varLocz {o} {
	\RoeSymbol[
		\sim,
		\IfNoValueF{#1}{#1,}
		\LSym,0
	]
}

\newcommand*{\spinC}{spin\(^\mathrm{c}\)}

\newcommand*{\ucov}[1]{\widetilde{#1}}

\DeclareMathOperator{\im}{im}
\newcommand{\id}{\mathrm{id}}

\DeclareMathOperator{\Var}{\operatorname{Var}}

\DeclareMathOperator{\Ind}{Ind}
\DeclareMathOperator{\Ad}{Ad}
\DeclareMathOperator{\tr}{tr}

\NewDocumentCommand{\blank}{}{{-}}

\NewDocumentCommand{\parensup}{m}{\textup{(}#1\textup{)}} 

\newcommand*{\tensmax}{\mathbin{\otimes_{\max}}}
\newcommand*{\tens}{\otimes}
\newcommand*{\exttensprod}{\boxtimes}

\numberwithin{equation}{section} 

\author{Alexander Engel\thanks{Fakult{\"a}t f{\"u}r Mathematik,
Universit{\"a}t Regensburg,
93040 Regensburg,
Germany\newline
\href{mailto:alexander.engel@mathematik.uni-regensburg.de}{alexander.engel@mathematik.uni-regensburg.de}}
\and Christopher Wulff\thanks{
Mathematisches Institut, Georg--August--Universität Göttingen, Bunsenstraße 3--5, 37073 Göttingen, Germany\newline
\href{mailto:christopher.wulff@mathematik.uni-goettingen.de}{christopher.wulff@mathematik.uni-goettingen.de}}
\and Rudolf Zeidler\thanks{Mathematisches Institut, Westfälische Wilhelms--Universität Münster, Einsteinstr.~62, 48149 Münster, Germany\newline
\href{mailto:math@rzeidler.eu}{math@rzeidler.eu}}
}

\title{Slant products on the Higson--Roe exact sequence}
\date{}

\begin{document}

\maketitle

\begin{abstract}
We construct a slant product $\// \colon \mathrm{S}_p(X \times Y) \times \mathrm{K}_{1-q}(\mathfrak{c}^{\mathrm{red}}Y) \to \mathrm{S}_{p-q}(X)$ on the analytic structure group of Higson and Roe and the K-theory of the stable Higson corona of Emerson and Meyer.
The latter is the domain of the co-assembly map $\mu^\ast \colon \mathrm{K}_{1-\ast}(\mathfrak{c}^{\mathrm{red}}Y) \to \mathrm{K}^\ast(Y)$.
We obtain such products on the entire Higson--Roe sequence.
They imply injectivity results for external product maps.
Our results apply to products with aspherical manifolds whose fundamental groups admit coarse embeddings into Hilbert space.
To conceptualize the class of manifolds where this method applies, we say that a complete $\mathrm{spin}^{\mathrm{c}}$-manifold is Higson-essential if its fundamental class is detected by the co-assembly map.
We prove that coarsely hypereuclidean manifolds are Higson-essential.
We draw conclusions for positive scalar curvature metrics on product spaces, particularly on non-compact manifolds.
We also obtain equivariant versions of our constructions and discuss related problems of exactness and amenability of the stable Higson corona.
\end{abstract}

\setcounter{tocdepth}{2}
\tableofcontents

\paragraph{Acknowledgements.}

The first ideas of this paper originated in discussions at the Oberwolfach workshop ``Analysis, Geometry and Topology of Positive Scalar Curvature Metrics'' held in 2017, and the authors thank the MFO for its hospitality.

We are grateful to many people for very helpful discussions about different parts of the present paper: Alcides Buss, Siegfried Echterhoff, Heath Emerson, Bernhard Hanke, Nigel Higson, Benedikt Hunger, Ralf Meyer and Rufus Willett.

Alexander Engel and Christopher Wulff acknowledge financial support by the DFG through the Priority Programme ``Geometry at Infinity'' (SPP 2026, EN 1163/3-1 and WU 869/1-1, ``Duality and the coarse assembly map'').

Alexander Engel further acknowledges support by the SFB 1085 ``Higher Invariants'' funded by the DFG.

Rudolf Zeidler was supported by the Deutsche Forschungsgemeinschaft (DFG, German Research Foundation) under Germany's Excellence Strategy EXC 2044-390685587, Mathematics Münster: Dynamics -- Geometry -- Structure.

\section{Introduction}
The first main entity we study in this paper is the analytic exact sequence of Higson and Roe~\cite{MappingSurgeryToAnalysisIII}.
Questions around this sequence have generated substantial activity in higher index theory, for a selection of recent works see~\cite{PiazzaSchick:StolzPSC,XieYuPscLocAlg,PiazzaSchick:SignatureOperator,Zenobi:MappingSurgery,DeeleyGoffengI,PiazzaZenobi:Singular,WeinbergerXieYu:Additivity}.
Its original motivation as devised by Higson and Roe was to serve as the target of certain analytic index maps defined on the surgery sequence from geometric topology.
For this reason it also often called \enquote{analytic surgery sequence}.
Later it was also used to serve as a target for index maps on the positive scalar curvature sequence of Stolz~\cite{PiazzaSchick:StolzPSC,XieYuPscLocAlg} (see~\cref{subsec:applications}).

To describe the sequence, start with the following setup.
Let \(X\) be a proper metric space endowed with a proper action of a countable discrete group \(G\).
Then the \emph{analytic sequence of Higson and Roe} associated to \(X\) reads as follows 
\begin{equation}
  \cdots \to \K_{\ast+1}(\Roe[G] X) \xrightarrow{\partial} \Strg_\ast^G(X) \to \K_\ast^G(X) \xrightarrow{\Ind} \K_\ast(\Roe[G] X) \to \cdots,
  \label{eq:HigsonRoe}
\end{equation}
where \(\K_\ast(\Roe[G] X)\) is the topological K-theory of the \emph{equivariant Roe algebra} \(\Roe[G] X\), \(\Strg^G_\ast(X)\) is the \emph{analytic structure group} of \(X\) and \(\K_\ast^G(X)\) stands for the \emph{equivariant locally finite \(\K\)-homology} of \(X\).
The definitions of these groups are given in \cref{sec:crossProducts}.
If \(X\) is a complete \spinC-manifold of dimension \(m\), where the \spinC-structure is preserved by \(G\), then the Dirac operator defines a fundamental class \([\Dirac_X] \in \K_m^G(X)\).

If the \(G\)-action is free, then \(\K_\ast^G(X) \cong \K_\ast(G \backslash X)\).
If the \(G\)-action on \(X\) is cocompact, then \(\Roe[G] X\) is canonically Morita equivalent to the reduced group \textCstar\nobreakdash-algebra \(\CstarRed G\).
The sequence \labelcref{eq:HigsonRoe} is often applied to the case \(X = \ucov{M}\), where \(\ucov{M}\) is the universal covering of a compact smooth manifold or finite simplicial complex \(M\), and \(G = \pi_1 M\).
In this case it becomes
\begin{equation}
\label{eq_analytic_surgery_sequence}
  \cdots \to \K_{\ast+1}(\CstarRed G) \xrightarrow{\partial} \Strg_\ast^{G}(\widetilde M) \to \K_\ast(M) \xrightarrow{\Ind_G} K_\ast(\CstarRed G) \to \cdots.
\end{equation}

Our initial investigations started with the following result of Zeidler:

\begin{thm}[{\cite[Corollary~5.8]{zeidler_secondary}}]\label{thm43re}
Let $N$ be a closed \spinC-manifold\footnote{Note that, although Corollary~5.8 of \cite{zeidler_secondary} is stated for spin-manifolds, the statement makes perfectly sense for \spinC-manifolds and the proof given by Zeidler works also perfectly well in the more general case of \spinC-manifolds.} of dimension $n$ such that its universal covering $\ucov{N}$ is \parensup{rationally} stably hypereuclidean.
Then for every closed manifold $M$ the map
\[\Strg_\ast^{G}(\ucov{M}) \to \Strg_{\ast+n}^{G \times H}(\ucov{M} \times \ucov{N}), \quad x \mapsto x \times [\Dirac_N],\]
where \(G = \pi_1 M\) and \(H = \pi_1 N\), is \parensup{rationally} split-injective.
\end{thm}
Here we implicitly used an external product which mixes an element of the structure group and an element of \(\K\)-homology to produce an element of the structure group of the product.
This is a well-known product construction which we revisit in \cref{sec:crossProducts}.

Next, recall the notion of a (stably) hypereuclidean manifold.
This goes back to \citeauthor{gromov_lawson_complete}~\cite{gromov_lawson_complete}.
\begin{defn}\label{defn_stably_hypereuclidean_noncompact}
A complete oriented Riemannian manifold \(X\) of dimension \(m\) is called \emph{stably hypereuclidean} if for some $k \in \IN$, the product $X \times \IR^k$ admits a proper Lipschitz map to Euclidean space $\IR^{m+k}$ of degree~$1$.
If the latter condition is relaxed to merely non-zero degree, it is called \emph{rationally stably hypereuclidean}.
\end{defn}

The assumption of being (rationally) stably hypereuclidean is quite general.
Dranishnikov \cite{dranish_geomdedicata} proved that the universal cover of any closed aspherical manifold whose fundamental group has finite asymptotic dimension is stably hypereuclidean.
Further, there is currently no known example of an aspherical closed manifold whose universal covering is not (rationally) stably hypereuclidean.

This notion is closely related to the strong Novikov conjecture.
Indeed, if the universal covering \(\ucov{M}\) of a closed \spinC-manifold \(M\) is rationally hypereuclidean, then the higher index class \(\Ind_G([\Dirac_M]) \in \K_\ast(\CstarRed G)\) is non-zero.
In fact, since \(\R^n\) satisfies the coarse Baum--Connes conjecture, the \emph{coarse index} \(\Ind([\Dirac_{\ucov{M}}]) \in \K_\ast(\Roe \ucov{M})\) does not vanish.
The canonical forgetful map \(\K_\ast(\CstarRed G) = \K_\ast(\Roe[G] \ucov{M}) \to \K_\ast(\Roe(\ucov{M}))\) takes the higher index to the coarse index and so this also proves non-vanishing of former.
Furthermore, it turns out that the conclusion of \cref{thm43re} in itself implies non-vanishing of the higher index class:

\begin{thm}[see \cref{thm_InjImpliesNonVanish}]\label{intro_thm_InjImpliesNonVanish}
Let $z \in \K_n(N)$ be such that for every closed manifold $M$ the map $\Strg_\ast^{G}(\ucov{M}) \to \Strg_{\ast+n}^{G \times H}(\ucov{M} \times \ucov{N})$, \(x \mapsto x \times z\), where \(G = \pi_1 M\) and \(H = \pi_1 N\), is rationally injective.
Then $\Ind_H(z) \in \K_n(\CstarRed H)$ is rationally non-zero.
\end{thm}
	The analogous integral version of \cref{thm_InjImpliesNonVanish} also holds.
	In fact, it is formally weaker because injectivity implies rational injectivity and rational non-vanishing implies non-vanishing.

These considerations suggested to us that there ought to be analytic conditions related to the coarse Baum--Connes conjecture and the strong Novikov conjecture that are weaker than hypereuclidean and would imply the conclusion \cref{thm43re}.
The quest for such hypotheses lead us to to the second main player of the present paper---the \emph{coarse co-assembly map}
\[
\mu^*\colon \K_{1-*}(\sHigCorRed X)\to \K^{*}(X)
\]
of \citeauthor{EmeMey}~\cite{EmeMey}, where \(\sHigCorRed X\) denotes the stable Higson corona of \(X\) and \(\K^{*}(X)\) is the compactly supported \(\K\)-theory of \(X\).
The coarse co-assembly map is dual to the coarse index map (or \enquote{coarse assembly map})
\[
  \mu = \Ind \colon \K_\ast(X) \to \K_\ast(\Roe X)
\]
via a pairing between the \(\K\)-theories of the stable Higson corona and the Roe algebra.
In particular, a \(\K\)-homology class which pairs non-trivially with a \(\K\)-theory class in the image of \(\mu^\ast\) does not lie in the kernel of \(\Ind\).
Our first main result extends \cref{thm43re} to such classes.
To formulate this, we introduce the following new property.
\begin{defn}\label{defn_Higson_essential_intro}
  We say that a complete Riemannian \spinC-manifold \(X\) of dimension \(m\) is (rationally) \emph{Higson-essential} if there exists \(\vartheta \in \K_{1-m}(\sHigCorRed X)\) such that \(\langle [\Dirac_X], \mu^\ast(\vartheta) \rangle = 1\) (\(\neq 0\), respectively), where \([\Dirac_X] \in \K_m(X)\) denotes the \(\K\)-homological fundamental class of the \spinC-structure.
\end{defn}
This notion is reminiscent of coarse largeness properties à la \citeauthor{BrunnbauerHanke}~\cite{BrunnbauerHanke}.
In fact, we will see that \parensup{rationally} \emph{coarsely stably hypereuclidean} \spinC-manifolds are (rationally) Higson-essential (see~\cref{hypereuclideanImpliesCoassemblyLift}).
In particular,  this applies to (rationally) stably hypereuclidean manifolds.

Furthermore, if the co-assembly map \(\mu^*\colon \K_{1-\ast}(\sHigCorRed X)\to \K^{\ast}(X)\) is surjective, then a \spinC-manifold \(X\) is automatically Higson-essential.
If \(X = \ucov{M}\) is the universal covering of a closed \emph{aspherical} manifold, then this co-assembly map is in fact known to be an isomorphism for a very broad class of examples.
For instance, if $\pi_1 M$ is coarsely embeddable into a Hilbert space \cite[Theorem~9.2]{EmeMey} or more generally if \(\pi_1 M\) has a \(\gamma\)-element~\cite[Corollary~34]{EM_descent}, or if it admits an expanding and coherent combing \cite[Theorem~5.10]{engel_wulff}.

Our first main result generalizes \cref{thm43re} from stably hypereuclidean to Higson-essential.
Start with a more general version which applies to non-compact manifolds.
\begin{thm}[{see \cref{cor_application_of_higson_essential}}]\label{intro_theorem_injectivity_higson_essential}
  Let \(Y\) be an \(n\)-dimensional complete \spinC-manifold of continuously bounded geometry\footnote{see \cref{defn:boundedGeometry}.\ref{defn_bounded_geom_uniform}}.
  Suppose that \(Y\) is \parensup{rationally} Higson-essential.
  Assume furthermore that \(Y\) is endowed with a proper action of a countable discrete group \(H\) which preserves the \spinC-structure.

	Then for every proper metric space \(X\) which is endowed with a proper action of a countable discrete group \(G\), the external product maps
  \[\Strg_\ast^{G}(X) \to \Strg_{\ast+n}^{G \times H}(X \times Y), \quad x \mapsto x \times [\Dirac_Y]\]
  and
  \[
  \K_\ast(\Roe[G] X) \to \K_{\ast+n}(\Roe[G \times H]{(X \times Y)}), \quad x \mapsto x \times \Ind([\Dirac_Y])
  \]
  are \parensup{rationally} split-injective.
\end{thm}
Specializing to universal coverings of closed manifolds yields the desired generalization of \cref{thm43re}:
\begin{cor}
  Let \(N\) be a closed \spinC-manifold of dimension $n$ such that its universal covering $\ucov{N}$ is \parensup{rationally} Higson-essential.
  Then for every closed manifold $M$ the maps
  \[\Strg_\ast^{G}(\ucov{M}) \to \Strg_{\ast+n}^{G \times H}(\ucov{M} \times \ucov{N}), \quad x \mapsto x \times [\Dirac_N],\]
  and
  \[
  \K_\ast(\CstarRed(G)) \to \K_{\ast+n}(\CstarRed (G \times H)), \quad x \mapsto x \times \Ind_H([\Dirac_N])
  \]
  where \(G = \pi_1 M\) and \(H = \pi_1 N\), are \parensup{rationally} split-injective.
\end{cor}

The existence of a \(\gamma\)-element has a far stronger consequences than this.
Traditionally, this concept appeared in proofs of the strong Novikov conjecture.
Indeed, it implies injectivity of the entire equivariant index map---not just non-vanishing of the index of the fundamental class.
Using an equivariant version of the co-assembly map it also enables us to prove a stronger result in the case of aspherical complexes.

\begin{thm}[see \cref{thm_struct_injective_gamma}]\label{intro_thm_struct_injective_gamma}
Let $N$ be a finite aspherical complex, and assume that $H = \pi_1 N$ has a $\gamma$-element.

Then for every proper metric space \(X\) endowed with a proper action of a countable discrete group \(G\), the external product maps
\[
\Strg_m^{G}(X) \otimes \K_{n}(N) \to \Strg_{m+n}^{G \times H}(X \times \ucov{N})
\]
and
\[
\K_m(\Roe[G] X) \otimes \K_{n}(N) \to \K_{m+n}(\Roe[G \times H](X \times \ucov{N}))
\]
 are rationally injective for each \(m, n \in \Z\).
\end{thm}

\begin{cor}\label{intro_cor_injective_gamma}
  In the setup of \cref{intro_thm_struct_injective_gamma}, where \(X = \ucov{M}\) is the universal covering of a closed manifold \(M\) and \(G = \pi_1 M\), the external product maps
  \begin{equation}
  \label{eq_ext_prod_ana_struct}
  \Strg_m^{G}(\ucov{M}) \otimes \K_{n}(N) \to \Strg_{m+n}^{G \times H}(\ucov{M} \times \ucov{N})
  \end{equation}
  and
  \begin{equation}
    \K_m(\CstarRed G) \otimes \K_{n}(N) \to \K_{m+n}(\CstarRed (G \times H)), \label{eq_ext_prod_group_alg}
  \end{equation}
  are rationally injective for each \(m, n \in \Z\).
\end{cor}

 One might conjecture that in the case of an aspherical manifold whose universal covering is stably hypereuclidean, the rational version of \cref{thm43re} should also follow from \cref{intro_cor_injective_gamma}.
 This would be the case if the following open question had a positive answer.
\begin{question}
Let $M$ be a closed aspherical manifold. If $M$ is stably hypereuclidean, does the fundamental group $\pi_1 M$ admit a $\gamma$-element?
\end{question}

If we assume that $\pi_1 M$ does not only have a $\gamma$-element, but that $\pi_1 M$ even satisfies the Baum--Connes conjecture, then the rational injectivity of~\eqref{eq_ext_prod_group_alg} upgrades to rational bijectivity due to the Künneth formula \cite{Kuenneth_BC}. Our next result is that under the assumption of Baum--Connes we also get such an upgrade for~\eqref{eq_ext_prod_ana_struct}.
To formulate this result in its full generality, we introduce the following notation for the \emph{representable \(\K\)-homology} and its counterpart for the structure group:
\begin{align}
  \RK_\ast^G(X) &\coloneqq \varinjlim_K \K_\ast^G(K), \label{eq:representableK}\\
  \RStrg_{\ast}^G(X) &\coloneqq \varinjlim_K \Strg^G_\ast(K),\label{eq:representableStrg}
\end{align}
where the colimits run over \(G\)-cocompact subsets \(K \subseteq X\).
Note that \labelcref{eq:HigsonRoe} then induces a sequence
\begin{equation}
  \cdots \to \K_{\ast+1}(\CstarRed G) \xrightarrow{\partial} \RStrg_\ast^G(X) \to \RK_\ast^G(X) \xrightarrow{\Ind} \K_\ast(\CstarRed G) \to \cdots.
  \label{eq:HigsonRoeRepresentable}
\end{equation}

\begin{thm}[\cref{thm_kuenneth}]
\label{thm_intro_kuenneth_full}
Let $H$ be a countable discrete group.
Assume that $H$ is torsion-free and satisfies the Baum--Connes conjecture for all coefficient \textCstar-algebras with trivial $H$-action.\footnote{For example, $H$ could be a-T-menable \cite{higson_kasparov} or it could be hyperbolic \cite{lafforgue_BC,puschnigg_BC}.}

Then for any simplicial complex $M$, the external product map
\[ \RStrg_*^{G}(\ucov{M}) \otimes \RK_*(\Bfree H) \to \RStrg_*^{G \times H}(\ucov{M} \times \Efree H),\]
 where \(G = \pi_1 M\), is rationally an isomorphism. 
 If $\RK_*(\Bfree H)$ is torsion-free, then it is integrally an isomorphism.
\end{thm}

\subsection{Slant products}
The main technical innovation of our paper is the construction of \emph{slant products} between the various groups which appear in the Higson--Roe sequence and the \(\K\)-theory of stable Higson corona.
The most general of these incorporate proper actions of countable discrete groups on all involved spaces as is needed for the proof of \cref{intro_thm_struct_injective_gamma}.
However, for simplicity, start with an exposition of the non-compact setting ignoring all group actions.
The next theorem and the following properties is a summary of \cref{sec_slant_products}. 

\begin{thm}\label{main_thm_intro_slant}
Let $X$ be a proper metric space, and let $Y$ be a proper metric space of bounded geometry.

For each element $\theta \in \K_{1-q}(\sHigCorRed Y)$ we construct natural slant products~$/\theta$ such that we have a commuting diagram
\begin{equation}
\xymatrix{
\Strg_p(X \times Y) \ar[r] \ar[d]^{/ \theta} & \K_p(X \times Y) \ar[r] \ar[d]^{/ \theta} & \K_p(\Roe(X \times Y)) \ar[d]^{/ \theta} \ar[r]^-{\partial} & \Strg_{p-1}(X \times Y) \ar[d]^{/ \theta}\\
\Strg_{p-q}(X) \ar[r] & \K_{p-q}(X) \ar[r] & \K_{p-q}(\Roe X) \ar[r]^-{\partial} & \Strg_{p-1-q}(X)
}
\label{eq_slant_mainthm_intro}
\end{equation}
and such that the slant products have the properties listed below in \cref{properties_slant}.
\end{thm}

We recall the relevant definition of bounded geometry at the beginning of \cref{sec_slantconstruction}, and the stable Higson corona $\sHigCorRed Y$ and the corresponding coarse co-assembly map are recalled at the beginning of \cref{sec_slant_products}.
\begin{rem}[Notation]\label{rem_HRPlaceholder}
To state certain results more concisely, we will use the symbol \(\HRPlaceholder\) as a generic placeholder for any constituent of the Higson--Roe sequence.
That is, \(\HRPlaceholder_\ast(X)\) can stand for either \(\Strg_\ast(X)\), \(\K_\ast(X)\) or \(\K_\ast(\Roe X)\).
\end{rem}

\begin{properties}\label{properties_slant}
The slant products in \cref{main_thm_intro_slant} satisfy the following properties:
\begin{myenumi}
\item\label{item_1_mainthm_slant} If $Y$ has continuously bounded geometry, then the slant product
\[
/\theta\colon \K_p(X \times Y) \to \K_{p-q}(X)
\]
is compatible with coarse co-assembly $\mu^\ast\colon \K_{1-q}(\sHigCorRed Y) \to \K^q(Y)$, that is, for all $x \in \K_p(X \times Y)$ we have
\[x /\theta = x /\mu^\ast(\theta),\]
where the slant product on the right hand side is the usual slant product of locally finite $\K$-homology with compactly supported $\K$-theory.

\item\label{item_2_mainthm_slant} For any element $y \in \K_q(\Roe Y)$ the composition\footnote{The construction of the external product $\blank \times y$ is recalled in \cref{sec:crossProducts}.}
\begin{equation*}
\K_p(\Roe X) \xrightarrow{\times y} \K_{p+q}(\Roe(X \times Y)) \xrightarrow{/\theta} \K_p(\Roe X)
\end{equation*}
equals multiplication with $\langle y, \theta\rangle \in \Z$ on $\K_p(\RoeSymbol X)$, where this is the pairing constructed by \citeauthor{EmeMey}~\cite[Section~6]{EmeMey}, possibly up to a sign $(-1)^q$.\footnote{\citeauthor{EmeMey} did not specify their sign convention for the pairing.}

Furthermore, for $y \in \K_q(Y)$ the compositions
\begin{align*}
\K_p(X) \xrightarrow{\times y} \K_{p+q}(X \times Y) \xrightarrow{/\theta} \K_p(X)\\
\Strg_p(X) \xrightarrow{\times y} \Strg_{p+q}(X \times Y) \xrightarrow{/\theta} \K_p(X)
\end{align*}
equal multiplication by $\langle y, \mu^\ast(\theta)\rangle \in \Z$, where this is the usual pairing between locally finite $\K$-homology and compactly supported $\K$-theory.

\item\label{item_3_mainthm_slant} Denote by $\beta \in \K_{1-n}(\sHigCorRed\IR^n)$ the class corresponding to the Bott element of Euclidean space, that is, $\mu^*(\beta) \in \K^n(\R^n) \cong \Z$ is the generator.

Then the slant product $/\beta\colon \HRPlaceholder_{p+n}(X \times \IR^n) \to \HRPlaceholder_p(X)$ is an isomorphism and coincides with \(n\)-fold suspension isomorphism\footnote{Or in other words, the Mayer--Vietoris boundary map associated to the (coarsely) excisive cover \(X \times \R = X \times (-\infty, 0] \cup X \times [0, \infty) \)}.

\item\label{item_4_mainthm_slant} The slant products are compatible with the coarsification and co-coarsification maps\footnote{The co-coarsification maps are sometimes called \enquote{character maps} in the literature.}.

Recall that the coarse assembly map $\mu \colon \K_p(Z) \to \K_p(\RoeSymbol(Z))$ factors~as
\[\K_p(Z) \xrightarrow{\coarsify} \KX_p(Z) \xrightarrow{\mu} \K_p(\RoeSymbol(Z))\]
and that the coarse co-assembly map $\mu^\ast\colon \K_{1-q}(\sHigCorRed Z) \to \K^q(Z)$ as
\[\mu^\ast\colon \K_{1-q}(\sHigCorRed Z) \xrightarrow{\mu^*} \KX^q(Z) \xrightarrow{\coarsify^\ast} \K^q(Z).\]
We have slant products on the coarsification of locally finite $\K$-homology such that the diagram
\[
\xymatrix{
\K_p(X \times Y) \ar[r]^-{\coarsify} \ar[d]^{/ \theta} & \KX_p(X \times Y) \ar[d]^{/ \theta} \ar[r]^-{\mu} & \K_{p}(\RoeSymbol(X \times Y)) \ar[d]^{/ \theta}\\
\K_{p-q}(X) \ar[r]^-{\coarsify} & \KX_{p-q}(X) \ar[r]^-{\mu} & \K_{p-q}(\RoeSymbol X)
}
\]
commutes and such that for $x \in \KX_p(X \times Y)$ we have $x / \theta = x / \mu^*(\theta)$. 
The slant product on the right hand side is between the coarsifications of locally finite $\K$-homology and of compactly supported $\K$-theory (see \cref{defn_coarsified_products}).
\qedhere
\end{myenumi}
\end{properties}
Property~\labelcref{item_1_mainthm_slant} in combination with commutativity of the middle square in the diagram \labelcref{eq_slant_mainthm_intro} translates to the formula
\begin{equation*}
\mu(x/\mu^*(\theta)) = \mu(x)/\theta
\end{equation*}
for all $x \in \K_p(X \times Y)$ and all $\theta \in \K_{1-q}(\sHigCorRed Y)$.

As an example, we mention the following result which can be obtained from the existence of the slant products with the above properties:
\begin{cor}[\cref{cor:conditionsCoAssIso}]
\label{cor_conditionsCoAssIso_intro}
Let $Y$ be either
\begin{myenumi}
\item a uniformly contractible, proper metric space of continuously bounded geometry which is scaleable,
\item a uniformly contractible, proper metric space of continuously bounded geometry which admits an expanding and coherent combing, or
\item the universal cover $\Efree G$ of the classifying space $\Bfree G$ of a group $G$, if $\Bfree G$ is a finite complex and $G$ is coarsely embeddable into a Hilbert space.
\end{myenumi}

Then for every proper metric space $Y$ the external product map
\[\HRPlaceholder_m(X) \otimes \K_{n}(Y) \to \HRPlaceholder_{m+n}(X \times Y)\]
is rationally injective.
Here we used the notation from \cref{rem_HRPlaceholder}.
\end{cor}

\subsection{Equivariant slant products}
\label{sec_idea_topological}
As previously mentioned, there is an equivariant version of our slant products.
We exhibit this construction in \cref{sec_equiv_version}.
It requires an equivariant version of the coarse co-assembly map.
If \(Y\) is endowed with a proper action of a countable discrete group \(H\), we obtain an induced action on the stable Higson corona.
Using this, one obtains a co-assembly map of the form
\[\mu^*_H\colon \K_{1-\ast}(\sHigCorRed Y \rtimes_{\mu} H)\to \K_H^{\ast}(Y),\]
where,  in general, \(\mu\) can be any exact crossed product functor in the sense of~\cite[Def.~3.1]{BaumGuentWillExactCrossed}, or the reduced one if \(H\) is exact.
Note that this is a different version of the equivariant co-assembly map than the one considered by~\citeauthor{EM_descent}~\cite{EM_descent,EM_coass}.
We discuss this and related questions around exactness in \cref{label_sec_factorization_assembly_coassembly,sec_weak_containment}.
In the following, we assume that an appropriate choice for \(\mu\) has been fixed.
\begin{thm}
Let \(X\) and \(Y\) be proper metric spaces, where \(Y\) has bounded geometry, which are endowed with proper isometric actions of countable discrete groups \(G\) and \(H\), respectively.
Then for each element $\theta \in \K_{1-q}(\sHigCorRed Y \rtimes_{\mu} H)$, we construct natural slant products~$/\theta$ such that we have a commuting diagram
\begin{align*}
\xymatrix{
\Strg_p^{G\times H}(X \times Y) \ar[r] \ar[d]^{/ \theta} & \K_p^{G \times H}(X \times Y) \ar[r] \ar[d]^{/ \theta} & \K_p(\Roe[G \times H](X \times Y)) \ar[d]^{/ \theta} \ar[r]^-{\partial} & \Strg_{p-1}^{G \times H}(X \times Y) \ar[d]^{/ \theta}\\
\Strg_{p-q}^G(X) \ar[r] & \K_{p-q}^G(X) \ar[r] & \K_{p-q}(\Roe[G] X) \ar[r]^-{\partial} & \K_{p-1-q}^G(X).
}
\end{align*}
\end{thm}
The equivariant slant products satisfy formal properties analogous to \cref{properties_slant}.
We refer to \cref{subsec_EquivSlantProp} for the details.
Moreover, in \cref{subsec_compatibility_equiv_nonequiv} we prove that for free actions our equivariant slant product on K-homology is identified with the usual slant product on K-homology of the quotient space up to canonical induction isomorphisms.

The equivariant slant products are applied to prove the following injectivity result for the external products on the Higson--Roe analytic sequence \eqref{eq_analytic_surgery_sequence}.
In the following theorem we write $\ucov M$ for the universal cover of $M$, and we denote $G \coloneqq \pi_1(M)$.

\begin{thm}[{see \cref{CrossWithElementInj}}]\label{thm:topologicalInj}
Let $N$ be a finite complex, \(H = \pi_1 N\), and let $z \in \K_n(N)$.
 Assume that there is $\theta \in \K_{1-n}(\sHigCorRed \ucov N \rtimes_{\mu} H)$ with $\langle z, \mu^*_H(\theta)\rangle = 1$ \parensup{or $\langle z,\mu^*_H(\theta)\rangle \not= 0$, respectively}.

Then for every finite complex $M$, all vertical arrows in the following diagram \parensup{with \(G = \pi_1 M\)} are split-injective \parensup{rationally split-injective, respectively}.
\begin{align*}
\resizebox{\textwidth}{!}{\xymatrix{
\K_{\ast+1}(\CstarRed G) \ar[r]^-{\partial} \ar[d]^-{\times \Ind_H(z)} & \Strg_\ast^{G}(\ucov{M}) \ar[r] \ar[d]^-{\times z} & \K_\ast(M) \ar[r]^-{\Ind_G} \ar[d]^-{\times z} & \K_\ast(\CstarRed G) \ar[d]^-{\times \Ind_H(z)}\\
\K_{\ast+1+n}(\CstarRed (G \times H)) \ar[r]^-{\partial} & \Strg_{\ast+n}^{G \times H}(\ucov{M} \times \ucov{N}) \ar[r] & \K_{\ast+n}(M \times N) \ar[r]^-{\Ind_{G\times H}} & \K_{\ast+n}(\CstarRed (G \times H))
}}
\end{align*}
\end{thm}
Indeed, a right inverse for the external product maps from the theorem is given by our equivariant slant products with the class \(\theta\).
A similar result was proved by Zenobi \cite[Remark~5.20]{Zenobi:MappingSurgery}.

Using a slightly more sophisticated version of the above theorem (see \cref{thm_coass_surj}), we can deduce \cref{intro_thm_struct_injective_gamma}.
This is because the existence of a $\gamma$-element implies surjectivity of the equivariant coarse co-assembly map for aspherical complexes (see \cref{cor_equiv_coarse_coass_surj}).

\subsection{Geometric applications}\label{subsec:applications}
\subsubsection{The Stolz sequence for positive scalar curvature}
The Stolz sequence for positive scalar curvature~\cite{Stolz98Concordance} is a sequence of bordism groups incorporating Riemannian metrics of positive scalar curvature (psc).
It is in some sense analogous to the surgery sequence from geometric topology.
Given a closed smooth manifold \(M\), we will denote by \(\RiemPos(M)\) the set of Riemannian metrics of positive scalar curvature on \(M\).

The Stolz sequence has received considerable attention in higher index theory starting with work of \citeauthor{PiazzaSchick:StolzPSC}~\cite{PiazzaSchick:StolzPSC} and \citeauthor{XieYuPscLocAlg}~\cite{XieYuPscLocAlg} who established that it admits a map to the analytic sequence of Higson and Roe.
We recall the result here in the case of classifying spaces of groups.
Here it is convenient to use \labelcref{eq:HigsonRoeRepresentable}.
\begin{thm}[{\cite{PiazzaSchick:StolzPSC,XieYuPscLocAlg}}]
  Let \(G\) be a countable discrete group.
  There is a commutative diagram of exact sequences taking Stolz' positive scalar curvature sequence to the analytic sequence of Higson and Roe.
  \[
  \begin{tikzcd}[column sep=small, ampersand replacement=\&]
     \SpinBordism_m(\Bfree G) \rar \dar{\beta} \&
     \StolzRel_m(\Bfree G) \rar{\partial} \dar{\alpha} \&
     \StolzPos_{m-1}(\Bfree G) \rar \dar{\rho} \&
     \SpinBordism_{m-1}(\Bfree G) \rar \dar{\beta} \&
     \StolzRel_{m-1}(\Bfree G) \dar{\alpha}
     \\
     \RK_{m}(\Bfree G) \rar{\mu} \&
     \K_{m}(\CstarRed G) \rar{\partial}
       \& \RStrg_{m-1}^G(\Efree G) \rar
       \& \RK_{m-1}(\Bfree G) \rar{\mu}
       \& \K_{m-1}(\CstarRed G)
  \end{tikzcd}
  \]
\end{thm}
We briefly explain the constituents of the top sequence above.
First, \(\SpinBordism_\ast(\Bfree G)\) is the usual spin bordism group of \(\Bfree G\).
The group \(\StolzRel_\ast(\Bfree G)\) consists of bordism classes of pairs \((W, g)\), where \(W\) is a compact spin manifold together with a continuous map \(W \to \Bfree G\) and \(g \in \RiemPos(\partial W)\).
Finally, the \emph{positive scalar curvature bordism group} \(\StolzPos_\ast(\Bfree G)\) consists of bordism classes of pairs \((M,g)\), where \(M\) is a closed spin manifold together with a continuous map \(M \to \Bfree G\) and \(g \in \RiemPos(M)\).
The map \(\beta\) is the Atiyah--Bott--Shapiro orientation, \(\alpha\) is the \emph{higher relative index} and \(\rho\) is the \emph{higher \(\rho\)-invariant}.

The positive scalar curvature bordism group admits an external product with the spin bordism group as follows.
\begin{align*}
	\StolzPos_m(\Bfree G) \otimes \SpinBordism_n(\Bfree H) &\to \StolzPos_{m+n}(\Bfree (G \times H))\\
   [M, g] \otimes [N] &\mapsto [M \times N, g \oplus g_{N}]
\end{align*}
Here we choose any Riemannian metric \(g_N\) on \(N\) such that the product metric \(g \oplus g_{N}\) still has positive scalar curvature (this can always be achieved by rescaling).
Any two such choices yield isotopic metrics on the product (again by a rescaling argument) and in particular the same bordism class.
By the product formula for the higher \(\rho\)-invariant~(see \cite{zeidler_secondary}), it is compatible with the external product for the analytic structure group.
That is, the following diagram commutes.
\[
	\xymatrix{
		\StolzPos_m(\Bfree G) \otimes \SpinBordism_n(\Bfree H) \ar[r] \ar[d]_-{\rho \otimes \beta} & \StolzPos_{m+n}(\Bfree (G \times H)) \ar[d]_-{\rho} \\
		\RStrg_m^G(\Efree G) \otimes \RK_n(\Bfree H) \ar[r] & \RStrg_{m+n}^{G \times H}(\Efree (G \times H))
	}
\]
Similarly, we have a diagram involving the relative group.
\[
	\xymatrix{
		\StolzRel_m(\Bfree G) \otimes \SpinBordism_n(\Bfree H) \ar[r] \ar[d]_-{\alpha \otimes \beta} & \StolzRel_{m+n}(\Bfree (G \times H)) \ar[d]_-{\alpha} \\
		\K_m(\CstarRed \Gamma) \otimes \RK_n(\Bfree H) \ar[r] & \K_{m+n}(\CstarRed (G \times H))
	}
\]

We now obtain the following corollary as a consequence of \cref{intro_theorem_injectivity_higson_essential}.
In the following we say that a spin manifold is Higson-essential if it is with respect to the induced \spinC-structure.
\begin{cor}\label{cor_PSCbordismApplication}
	Let \(N\) be a closed spin manifold with \(\pi_1 N = H\) such that its universal covering \(\ucov{N}\) is Higson-essential.
	 Let \([M_i, g_i] \in \StolzPos_m(\Bfree G)\) with \(\rho(M_0, g_0) \neq \rho(M_1, g_1) \in \RStrg_{m}^G(\Efree G)\).
      Then also 
      \[\rho(M_0 \times N, g_0 \oplus g_N) \neq \rho(M_1 \times N, g_1 \oplus g_N) \in \RStrg_{m+n}^{G \times H}(\Efree (G \times H)).\]
    	In particular, \((M_0 \times N, g_0 \oplus g_N)\) and \((M_1 \times N, g_1 \oplus g_N)\) represent different bordism classes in \(\StolzPos_{m+n}(\Bfree (G \times H))\).
\end{cor}
	In \cite[Corollary~6.10]{zeidler_secondary} this corollary was formulated for \(N\) aspherical with \(H = \pi_1 N\) of finite asymptotic dimension.
	This was based on Dranishnikov's theorem that the universal covering of such a manifold is stably hypereuclidean \cite{dranish_geomdedicata}.
	Our present method strictly improves Zeidler's result, because if \(N\) is aspherical and \(H = \pi_1 N\) admits a coarse embedding into Hilbert space (which is far more general than having finite asymptotic dimension), it satisfies the hypothesis of the corollary above (compare~\cref{rem_surj_implies_Higson-essential} and \cref{cor:conditionsCoAssIso}~\labelcref{item:CoarseEmbedd}).

Furthermore, if \(M\) is a closed spin manifold with \(\pi_1 M = G\) and \(g_0, g_1 \in \RiemPos(M)\), then \((M \times [0,1], (g_0, g_1))\) represents a class in \(\StolzRel_{m+1}(\Bfree G)\), where \((g_0, g_1) \in \RiemPos(\partial (M \times [0,1])) = \RiemPos(M) \times \RiemPos(M)\).
The corresponding relative index class in \(\K_{m+1}(\CstarRed G)\) is the \emph{\parensup{higher} index difference} of \(g_0\) and \(g_1\), denoted by \(\inddiff_G(g_0,g_1) \in \K_{m+1}(\CstarRed G)\).
If the index difference is non-zero, then the two metrics are \emph{not concordant} as positive scalar curvature metrics.
In particular, they are not isotopic as psc metrics.

Again we obtain the following from \cref{intro_theorem_injectivity_higson_essential}.
\begin{cor}\label{cor_compact_concordance_obstruction}
	Let \(N\) be a closed spin manifold with \(\pi_1 N = H\) such that its universal covering \(\ucov{N}\) is Higson-essential.
	Let \(M\) be a closed spin manifold with \(\pi_1 M = G\) and \(g_0, g_1 \in \RiemPos(M)\) such that \(\inddiff_G(g_0, g_1) \neq 0 \in \K_{m+1}(\CstarRed G)\).
	Then 
    \[\inddiff_{G \times H}(g_0 \oplus g_N, g_1 \oplus g_N) \neq 0 \in \K_{m+n+1}(\CstarRed(G \times H)).\]
  In particular, \(g_0 \oplus g_N\) and \(g_1 \oplus g_N\) are not concordant as positive scalar curvature metrics on \(M \times N\).
\end{cor}
The final conclusion of the previous corollary is also true for almost-spin manifolds~\(N\), see \cref{cor_almost_spin} below.

Examples of pairs of \((M_i, g_i)\) with different \(\rho\)-invariants or pairs of metrics with non-trivial index-differences exist in abundance.
For the \(\rho\)-invariants, such examples always exist if the group \(G\) satisfies the strong Novikov conjecture and has torsion.
For the existence of non-trivial index differences, this is also true for torsion-free groups.
In fact, under the Novikov assumption, lower bounds on the ranks of the groups \(\StolzPos_\ast(\Bfree G)\) and \(\StolzRel_\ast(\Bfree G)\) can be given in terms of the homology of \(G\), see for instance~\cite{WY15FinitePart,BZ17LowDegree,ERW17pscFundamentalGroup,XYZ}.
These results are always proved by showing that secondary index maps such as \(\alpha\) and \(\rho\) have a large range.

We deduce from \cref{intro_thm_struct_injective_gamma} (applying it inside the relevant colimits) that the size of the range is preserved under certain products.

\begin{cor}
  Suppose the group \(H\) has a \(\gamma\)-element and admits a finite model for \(\Bfree H\).
  \begin{myenumi}
    \item   Let \(V \subseteq \StolzPos_m(\Bfree G)\) be a subset such that \(\rho(V) \subseteq \RStrg_m^G(\Efree G)\) generates a subgroup of rank \(\geq k\).
      Then \(V \otimes \SpinBordism_n(\Bfree H)\) generates a subgroup of rank 
      \[ 
        \geq k \cdot \operatorname{rank}(\beta \SpinBordism_n(\Bfree H))
      \]
      in \(\StolzPos_{m+n}(\Bfree (G~\times~H))\).
    \item Let \(V \subseteq \StolzRel_m(\Bfree G)\) be a subset such that \(\alpha(V) \subseteq \K_m(\CstarRed G)\) generates a subgroup of rank \(\geq k\).
      Then \(V \otimes \SpinBordism_n(\Bfree H)\) generates a subgroup of rank 
      \[ 
        \geq k \cdot \operatorname{rank}(\beta \SpinBordism_n(\Bfree H))
      \]
      in \(\StolzRel_{m+n}(\Bfree (G \times H))\).
\end{myenumi}
\end{cor}

\subsubsection{Positive scalar curvature on non-compact manifolds}
As our methods deal with non-compact spaces, there are applications to uniform positive scalar curvature on non-compact manifolds.
However, here it is necessary to restrict the large-scale type of the metrics which we consider.
Otherwise, the fact that \(\R^n\) for \(n \geq 3\) admits metrics of uniform positive scalar curvature would lead to counterexamples to the kind of results we have in mind.
We use a similar setup as in \cite[Section~1.3]{ZeidlerThesis}.

Let \(X\) be a spin \(n\)-manifold and \(g_X\) a fixed complete Riemannian metric on \(X\).
We let \(\Riem(X, g_X)\) denote the set of all those Riemannian metrics \(g\) on \(X\) such that the identity map \((X, d_g) \to (X, d_{g_X})\) is uniformly continuous.
Note that each metric in \(\Riem(X, g_X)\) is automatically complete because \(g_X\) is.
Moreover, the identity map \((X, d_g) \to (X, d_{g_X})\) is coarse because \((X,d_g)\) is a length space and so uniformly continuous maps are automatically large-scale Lipschitz. 
Moreover, we let \(\RiemPos(X, g_X)\) be the set of those metrics in \(\Riem(X, g_X)\) with uniformly positive scalar curvature.
If \(X\) is furnished with a proper isometric action of a countable discrete group \(G\) and \(g_X\) is \(G\)-invariant, we write \(\Riem(X, g_X)^G\) and \(\RiemPos(X, g_X)^G\) for the corresponding subsets of \(G\)-invariant metrics.
Note that if \(\Ind(\Dirac_X) \neq 0 \in \K_m(\Roe[G] X)\), where we define the Roe algebra with respect to the metric \(g_X\), then \(\RiemPos(X, g_X)^G =\emptyset\).
\begin{defn}
  We say that two metrics \(g_0, g_1 \in \RiemPos(X, g_X)^G\) are \emph{concordant} if there exists a metric in \(\RiemPos(X \times \R, g_X \oplus \D t^2)^G\) which restricts to \(g_0 \oplus \D t^2\) on \(X \times (-\infty,0]\) and to \(g_1 \oplus \D t^2\) on \(X \times [1, \infty)\).
\end{defn}
Given \(g_0, g_1 \in \RiemPos(X, g_X)^G\), there is the \emph{\parensup{equivariant} coarse index difference} \(\inddiff(g_0, g_1) \in \K_{n+1}(\Roe[G] X)\) which vanishes if \(g_0\) and \(g_1\) are concordant, see~\cite[Section~2.2.4]{ZeidlerThesis}, \cite[Section~4.4]{zeidler_secondary}.
Note that if \(X = \ucov{M}\) is the universal covering of a closed manifold \(M\), then the equivariant coarse index difference agrees with the index difference considered in the previous subsection via the identification \(\K_\ast(\Roe[G](\ucov{M})) = \K_\ast(\CstarRed G)\).
Moreover, given another complete spin manifold \((Y, g_Y)\)  such that \(g_i \oplus g_Y\) has uniformly positive scalar curvature for \(i =0,1\), then the coarse index difference satisfies the product formula
\[
	\inddiff(g_0 \oplus g_Y, g_1 \oplus g_Y) = \inddiff(g_0, g_1) \times \Ind([\Dirac_Y]).
\]
Using our slant products, we obtain the following results.

\begin{cor}\label{cor_noncompactObstruction}
	Let \((Y, g_Y)\) be a complete spin manifold of continuously bounded geometry which is Higson-essential.
	Let \(M\) be a closed spin manifold and \(G \coloneqq \pi_1 M\).
  \begin{myenumi}
    \item 	If \(\Ind_G[\Dirac_M] \neq 0 \in \K_{m}(\CstarRed G)\), then \(\RiemPos(M \times Y, g_M \oplus g_Y) = \emptyset\).\label{item:noncompactObstruction}
    \item Let \(g_0, g_1 \in \RiemPos(M)\) such that \(\inddiff_G(g_0,g_1) \neq 0 \in \K_{m+1}(\CstarRed G)\).
    Then \(g_0 \oplus g_Y\) and \(g_1 \oplus g_Y\) are not concordant on \(M \times Y\).
    \label{item:noncompactDifference}
  \end{myenumi}
  Analogous statements apply if \(Y\) is rationally Higson-essential and the relevant index class is rationally non-zero.
\end{cor}
\begin{proof}
  Lifting metrics induces an identification 
  \[ \Riem(M \times Y, g_M \oplus g_Y) = \Riem(\ucov{M} \times Y, g_{\ucov{M}} \oplus g_Y)^G,\]
  preserving products, uniform psc and concordances.
  Hence it suffices to prove the corresponding statements in \(\RiemPos(\ucov{M} \times Y, g_{\ucov{M}} \oplus g_Y)^G\).
  The product formula for the index and index difference together with \cref{intro_theorem_injectivity_higson_essential} imply \(\Ind([\Dirac_{\ucov{M} \times Y}]) \neq 0 \in \K_{m+n}(\Roe[G](\ucov{M} \times Y)) \) in part \labelcref{item:noncompactObstruction} and \(\inddiff(\ucov{g}_0 \oplus g_Y, \ucov{g}_1 \oplus g_Y) \neq 0 \in \K_{m+n+1}(\Roe[G](\ucov{M} \times Y)) \) in part \labelcref{item:noncompactDifference}.
\end{proof}

In particular, we deduce results for closed almost spin manifold (that is, the universal covering is spin but not necessarily the manifold itself).

\begin{cor}\label{cor_almost_spin}
  Let \(N\) be a closed manifold with \(\pi_1 N = H\) such that its universal covering \(\ucov{N}\) is spin and Higson-essential.
  Let \(M\) be a closed spin manifold with \(\pi_1 M = G\).
  \begin{myenumi}
    \item If \(\Ind_G[\Dirac_M] \neq 0 \in \K_{m}(\CstarRed G)\), then \(\RiemPos(M \times N) = \emptyset\).
    \item Let \(g_0, g_1 \in \RiemPos(M)\) such that \(\inddiff_G(g_0, g_1) \neq 0 \in \K_{m+1}(\CstarRed G)\).
    Then \(g_0 \oplus g_N\) and \(g_1 \oplus g_N\) are not concordant as positive scalar curvature metrics on \(M \times N\), where \(g_N\) is some Riemannian metric on \(N\) such that \(g_i \oplus g_N \in \RiemPos(M \times N)\)
  \end{myenumi}
\end{cor}
\begin{proof}
  If \(\RiemPos(M \times N) \neq \emptyset\), then lifting the metric shows \(\RiemPos(M \times \ucov{N}, g_M \oplus \ucov{g}_N) \neq \emptyset\) for any choice of Riemannian metrics \(g_M\) and \(g_N\) on \(M\) and \(N\), respectively.
  Hence the statement follows from \cref{cor_noncompactObstruction}~\labelcref{item:noncompactObstruction}.
  The second part is reduced to \cref{cor_noncompactObstruction} in an analogous fashion.
\end{proof}

For \(g \in \RiemPos(X, g_X)^G\), there is also a \(\rho\)-invariant \(\rho(g) \in \Strg^G_m(X)\).
If \(\rho(g_0) \neq \rho(g_1)\), then the metrics are also not concordant.
\Cref{intro_theorem_injectivity_higson_essential} then implies an analogous version of \cref{cor_noncompactObstruction} for the \(\rho\)-invariant.
However, \(\rho(g_0) \neq \rho(g_1)\) already implies \(\inddiff(g_0, g_1) \neq 0\).
In fact, this \(\rho\)-invariant is a coarse bordism invariant in a suitable sense, see~\cite[Section~2.4.2]{ZeidlerThesis}, but formulating this requires some care to ensure that the structure groups on the two different ends of a bordism remain comparable.
In the right setup, it is then also possible to establish a non-compact version of the final conclusion from \cref{cor_PSCbordismApplication}, but we refrain from expounding the details here.

\subsection{Generalizations and questions}

\subsubsection{Coefficients}

Instead of working with the ordinary Roe algebra and ordinary stable Higson corona, we could have used throughout this paper their corresponding versions with coefficients, that is, Hilbert modules $\cE$, $\cF$ and $\cG$ for the Roe algebras of the corresponding spaces and a \textCstar-algebra $C$ for the stable Higson corona.\footnote{Roe algebras with coefficients in \textCstar-algebras were considered by \citeauthor{higson_pedersen_roe}~\cite{higson_pedersen_roe}. The definition of Roe algebras with coefficients in \textCstar-algebras given in \cite{hanke_pape_schick} is quickly seen to generalize to Hilbert modules; see also \cite{WulffTwisted}.

\citeauthor{EmeMey}~\cite{EmeMey} defined the stable Higson corona from the very beginning  with coefficients in \textCstar-algebras.}
\(\K\)-homology and the analytic structure group can be similarly enriched.

For $a\in \K_m(Y;\cF)$ we then have the following diagram for the external product by $a$:
\begin{align*}
\resizebox{\textwidth}{!}{\xymatrix{
\Strg_n(X;\cE) \ar[r] \ar[d]^{\times a}
& \K_n(X;\cE) \ar[r] \ar[d]^{\times a}
& \K_n(\Roe(X;\cE)) \ar[d]^{\times \mu(a)}\ar[r]
& \Strg_{n-1}(X;\cE)) \ar[d]^{\times a}
\\\Strg_{m+n}(X \times Y;\cE\otimes \cF)) \ar[r]
& \K_{m+n}(X \times Y;\cE\otimes\cF) \ar[r]
& \K_{m+n}(\Roe(X \times Y;\cE\otimes\cF)) \ar[r]
& \Strg_{m+n-1}(X \times Y;\cE\otimes \cF)
}}
\end{align*}
and for $b\in K_{1-m}(\sHigCorRed (Y;C))$ the following version of the slant products by $b$:
\begin{align*}
\resizebox{\textwidth}{!}{\xymatrix{
\Strg_{n}(X \times Y;\cG) \ar[r] \ar[d]^{/b}
& \K_{n}(X \times Y;\cG) \ar[r] \ar[d]^{/\mu^*(b)}
& \K_{n}(\Roe(X \times Y;\cG)) \ar[r] \ar[d]^{/b}
& \Strg_{n-1}(X \times Y;\cG)\ar[d]^{/b}
\\\Strg_{n-m}(X;\cG\otimes C) \ar[r]
& \K_{n-m}(X;\cG\otimes C) \ar[r]
& \K_{n-m}(\Roe(X;\cG\otimes C)) \ar[r]
& \Strg_{n-1-m}(X;\cG\otimes C)
}}
\end{align*}

In particular, take \(X = \ast\) to be a point and consider the standard Hilbert-module \(\mathcal{A} = A \otimes \ell^2\), where \(A\) is some \textCstar-algebra.
Then the Roe algebra \(\Roe(\ast; \mathcal{A})\) is isomorphic to \(A \otimes \Kom(\ell^2)\) and so \(\K_\ast(\Roe(\ast; \mathcal{A})) = \K_\ast(A)\).
In the situation of \cref{intro_theorem_injectivity_higson_essential}, we could then deduce (rational) injectivity of the external product map
\[
  \K_\ast(A) \to \K_{\ast + n}(A \otimes \Roe[H](Y)), \quad x \mapsto x \times \Ind[\Dirac_Y].
\]

We do not discuss this any further in the paper to keep the notation lean.

\subsubsection{Real K-theory}
We have formulated the results of this paper in the framework of complex \(\K\)-theory for simplicity.
However, especially for applications to positive scalar curvature and spin geometry, it would be desirable to establish the analogous statements for \(\mathrm{KO}\)-homology and the corresponding real version of the analytic structure group.
Our construction of slant products is sufficiently abstract and does not use any idiosyncrasies of complex \(\K\)-theory (such as---for instance---using \(2\)-periodicity in an essential way).
Hence the slant products also exist in the real setup and all applications that involve external products with a single element (such as \cref{intro_theorem_injectivity_higson_essential}) go through without essential change.

However, more care has to be taken with results that rely on universal coefficient and Künneth theorems (such as \cref{intro_thm_struct_injective_gamma,thm_intro_kuenneth_full}).
They do not appear to readily generalize and instead would require a more elaborate framework such as in~\cite{Boersema:RealKuenneth}.

\subsubsection{Künneth sequence for the structure group}

In \cref{intro_cor_injective_gamma}, we could have also included the case of the external product for $\K$-homology, that is, $\K_m(M) \otimes \K_{n}(N) \to \K_{m+n}(M \times N)$.
In fact, injectivity of this map holds in full generality due to the Künneth theorem for $\K$-homology.

There is also a Künneth theorem for the $\K$-theory of the reduced group \textCstar-algebras, provided the Baum--Connes conjecture is satisfied by one of the groups.\footnote{Tu \cite{tu_UCT} proved that if $G$ is amenable, then $\CstarRed G$ lies in the bootstrap class and hence satisfies the Künneth formula. That $\CstarRed G$ satisfies the Künneth formula if $G$ satisfies the Baum--Connes conjecture was proven by \citeauthor{Kuenneth_BC}~\cite{Kuenneth_BC}.} Therefore, in this situation we get rational injectivity of \eqref{eq_ext_prod_group_alg} as in the \cref{intro_thm_struct_injective_gamma}. But the assumption of satisfying Baum--Connes is much stronger than admitting a $\gamma$-element.

It is now a natural question if there is also a version of the Künneth sequence for the structure group.
For instance, one might ask if there is a short exact sequence of the form
\begin{equation*}
0 \to (\Strg_*^{G}(\ucov{M}) \otimes \K_*(N)) \oplus (\K_*(M) \oplus \Strg_*^{H}(\ucov{N})) \to \Strg_*^{G \times H}(\ucov{M} \times \ucov{N}) \to \operatorname{?Tor?} \to 0,
\end{equation*}
where $\operatorname{?Tor?}$ is some suitable correction term analogous to the \(\operatorname{Tor}\)-term in the Künneth sequence for \(\K\)-homology.
For this to make sense, it is probably necessary to diagonally divide out \(\Strg_*^{G}(\ucov{M}) \otimes \Strg_\ast^H(\ucov{N})\) in the term on the left.
If such a sequence exists, then it should ideally imply the conclusion of \cref{thm_intro_kuenneth_full}.
However, by the result of \cref{sec_proof_SNC}, proving the existence of such a sequence will---realistically---require at least some hypotheses related to the strong Novikov conjecture.

\subsubsection{Groups with torsion}

Many of our results (such as \cref{intro_thm_struct_injective_gamma}) involve finite classifying spaces of countable discrete groups (that is, finite aspherical complexes) which restricts them to torsion-free groups.
However, in our general constructions (especially in \cref{sec_equiv_version}) we need the groups to act only properly and not necessarily freely.
Thus many of these results will have corresponding versions for groups with torsion if we work with the classifying space for proper actions $\Eub G$ instead---at least if the latter admits a \(G\)-finite model.

\subsubsection{Exactness and the stable Higson corona}

In \cref{sec_equiv_version}, when introducing the equivariant versions of the slant products, we are working with crossed products $\sHigCorRed Y \rtimes_{\mu} H$, where \(\rtimes_\mu\) is any exact crossed product functor.
This is necessary because the relevant co-assembly map might in general not exist for the reduced crossed product unless $H$ is exact.

So the natural question arises in which situations $\sHigCorRed Y$ is an amenable $H$-\textCstar-algebra (see \cref{defn_amenabilities}) and hence all different choices of crossed products for $\sHigCorRed Y \rtimes H$ coincide. This is in general always the case when $H$ is an amenable group, and we show that it furthermore holds in the case that $H$ is a Gromov hyperbolic group acting on itself:
\begin{prop}[\cref{ex_hyperbolic_amenable_action}]
Let $H$ be a Gromov hyperbolic group.
Then $\sHigCorRed H$ is an amenable $H$-\textCstar-algebra and $\sHigCorRed H \rtimes_{\max} H \cong \sHigCorRed H \rtimes_{\red} H$.
\end{prop}

One can ask about the concrete relation between amenability of $\sHigCorRed H$, the equality $\sHigCorRed H \rtimes_{\max} H \cong \sHigCorRed H \rtimes_{\red} H$ and the exactness of $H$. In \cref{sec_weak_containment} we get first partial results on this, and our general conjecture is the following:

\begin{conjecture}\label{conj_weak_containment}
Let $Y$ be a proper metric space equipped with an isometric action of a countable discrete group $H$.
Then the following conditions are equivalent:
\begin{myenuma}
\item\label{item_conj_equiv_amenable_one} The group $H$ acts amenably on the Higson compactification of $Y$.
\item\label{item_conj_equiv_amenable_two} $\sHigCorRed Y$ is an amenable $H$-\textCstar-algebra.
\item\label{item_conj_equiv_amenable_three} The group $H$ is exact and we have $\sHigCorRed Y \rtimes_{\max} H \cong \sHigCorRed Y \rtimes_{\red} H$.
\end{myenuma}
\end{conjecture}

Note that there is a related statement in the dual situation, that is, for the uniform Roe algebra: $H$ is exact if and only if the uniform Roe algebra $\mathrm{C}^\ast_{\mathrm{u}} H$ is exact \cite{brodzki_niblo_wright}.

\section{Injectivity implies non-vanishing index}
\label{sec_proof_SNC}
In this short section, we provide a proof for our motivating observation which was stated as \cref{intro_thm_InjImpliesNonVanish} in the introduction.
\begin{thm}\label{thm_InjImpliesNonVanish}
  Let \(N\) be a finite complex and $z \in \K_n(N)$ be such that for every closed manifold $M$ the map $\Strg_\ast^{G}(\ucov{M}) \to \Strg_{\ast+n}^{G \times H}(\ucov{M} \times \ucov{N})$, \(x \mapsto x \times z\), where \(G = \pi_1 M\) and \(H = \pi_1 N\), is rationally injective.
  Then $\Ind_H(z) \in \K_n(\CstarRed H)$ is rationally non-zero.
\end{thm}

\begin{proof}
	Suppose by contraposition that rationally $\Ind_H(z) = 0$.
	For every $M$, there is the following commutative diagram.
\[
\xymatrix{
  \K_{\ast+1}(\CstarRed G) \ar[r]^-{\partial_M} \ar[d]^-{\blank\times \Ind_H(z)} &
    \Strg_\ast^G(\ucov{M}) \ar[d]^-{\blank\times z} \\
  \K_{\ast+1+n}(\CstarRed (G \times H)) \ar[r]^-{\partial_{M \times N}} &
    \Strg_{\ast+n}^{G \times H}(\ucov{M} \times \ucov{N})
}
\]
A construction of this diagram is provided in \cref{subsec_crossProducts_equiv}.

	Since the left vertical arrow is rationally zero, we conclude that the image of $\partial_M \otimes \Q$ is contained in the kernel of
	\[ (\blank \times z) \otimes \Q \colon
	 \Strg_\ast^G(\ucov{M}) \otimes \Q \to \Strg_{\ast+n}^{G \times H}(\ucov{M} \times \ucov{N}) \otimes \Q.
	\]
	Thus to complete the proof it suffices to find a closed manifold $M$ such that $\partial_M \otimes \Q \neq 0$.
	Indeed, this happens for instance if $M$ is such that $G = \pi_1 M$ is a non-trivial finite group.
	This is folklore but we briefly explain it for the convenience of the reader.

  Let \(G\) be finite.
  Then
  \[
  \K_i(\CstarRed G) = \K_i(\C[G]) \cong \begin{cases} \operatorname{R}(G) & i = 0, \\
  0 & i = 1,
  \end{cases}
  \]
  where \(\operatorname{R}(G)\) denotes the Grothendieck group of finite-dimensional complex \(G\)-representations.
  Elementary character theory shows that the rank of the abelian group \(\operatorname{R}(G)\) is the number of conjugacy classes of elements in \(G\).
  Hence it is greater than one because \(G\) is non-trivial.
  Moreover, since the homology of a finite group is torsion in all positive degrees, it follows (for instance by an application of the Atiyah--Hirzebruch spectral sequence) that
  \[
    \RK_i(\Bfree G) \otimes \Q \cong \RK_i(\ast) \otimes \Q \cong \begin{cases}
      \Q & i = 0,\\ 0 & i =1,
    \end{cases}
  \]
  where the isomorphism is induced by mapping onto the point.
  Thus, the rational assembly map \(\Q \cong \K_0(\Bfree G) \to \K_0(\C[G]) \otimes \Q \cong \operatorname{R}(G) \otimes \Q\) is not surjective for dimension reasons (in fact, the image is generated by the left-regular representation of \(G\)).
  Putting these observations together, we see that the rational Higson--Roe sequence of \(G\) collapses to the following short exact sequence (with all other terms vanishing):
	\[
		0 \to \underbrace{\RK_0(\Bfree G) \otimes \Q}_{\cong\Q}
		\to \underbrace{\K_0(\CstarRed G) \otimes \Q}_{\cong \operatorname{R}(G) \otimes \Q}
		 \xrightarrow{\partial_G \otimes \Q} \RStrg_1^G(\Efree G) \otimes \Q  \to 0
	\]
  Since the assembly map is not surjective, \(\partial_G \otimes \Q \neq 0\).
	Finally, for any $M$ with $\pi_1 M = G$, the boundary map $\partial_G$ factors as
	\[
		\partial_G \colon \K_{\ast+1}(\CstarRed G) \xrightarrow{\partial_M} \Strg_\ast^G(\ucov{M}) \to \RStrg_\ast^G(\Efree G).
	\]
	Hence $\partial_G \otimes \Q \neq 0$ implies $\partial_M \otimes \Q \neq 0$.
\end{proof}

\begin{rem}
 In fact, the proof shows that injectivity of the external product map \(\Strg_\ast^G(\ucov{M}) \xrightarrow{\blank \times z} \Strg_{\ast+n}(\ucov{M} \times \ucov{N})\) for any \(M\), where \(\pi_1 M = G\) is a non-trivial finite group, suffices to obtain the conclusion of the theorem.
\end{rem}

\section{External products}\label{sec:crossProducts}

In this section we will revisit the construction of the external product maps via localization algebras.
The idea for this approach appeared first in work of Xie and Yu~\cite[Section~2.3.3]{XieYuPscLocAlg} and was fleshed out by Zeidler~\cite[Section~3.3]{zeidler_secondary}.
See also~\cite[Section~9.2]{WillettYuHigherIndexTheory}.
We start in the setup without group actions and discuss the equivariant situation thereafter.

\subsection{Non-equivariant case}\label{subsec_crossProducts_nonequiv}

Let $X$ be a proper metric space.
An \(X\)-module is a separable Hilbert space $H_X$ endowed with a non-degenerate \(\ast\)-representation $\rho_X\colon \Cz(X) \to \Lin(H_X)$.
Given an \(X\) module, the \emph{Roe algebra} $\Roe (\rho_X)$ is defined as the sub-\textCstar\nobreakdash-algebra of $\Lin(H_X)$ generated by all locally compact operators of finite propagation. These notions are defined as follows.
\begin{itemize}
\item An operator $T \in \Lin(H_X)$ is called \emph{locally compact}, if for every $f \in \Cz(X)$ the operators $\rho_X(f) T$ and $T \rho_X(f)$ are compact operators.
\item An operator $T \in \Lin(H_X)$ is said to have \emph{finite propagation} if there exists an $R > 0$ such that $\rho_X(f) T \rho_X(g) = 0$ whenever the supports of $f$ and $g$ are further apart from each other than $R$. In this case, the propagation is said to be bounded by $R$.
\end{itemize}
An \(X\)-module \((H_X, \rho_X)\) is called \emph{ample}, if no non-zero function from $\Cz(X)$ acts by a compact operator.
 From now on we shall assume that we have fixed an ample \(X\)-module \((H_X, \rho_X)\).

The crucial fact about Roe algebras is that their $\K$-theory is independent of the choice of ample \(X\)-module \cite[Corollary~6.3.13]{higson_roe} \cite[Theorem~5.1.15]{WillettYuHigherIndexTheory} up to canonical isomorphism.
Hence we usually suppress it from the notation by writing \(\Roe X\) instead of \(\Roe(\rho_X)\). Nevertheless, later on we will have to consider the Roe algebras of $X$ associated to different representations at once. In those cases, we default to the notation $\Roe(\rho_X)$.

Further, one defines the \emph{localization algebra} $\Loc X$ as the sub-\Cstar-algebra of $\Cb([1,\infty),\Roe X)$ generated by the bounded and uniformly continuous functions $L\colon [1,\infty) \to\Roe X$ such that the propagation
of $L(t)$ is finite for all $t\geq 1$ and tends to zero as $t\to\infty$.
If it is constructed using an ample \(X\)-module, its $\K$-theory is canonically isomorphic to the locally finite $\K$-homology of the space $X$, that is,
\[\K_\ast(\Loc X)\cong \KK_\ast(\Cz(X), \C) \cong \K_\ast(X).\]
This fact was originally established by Yu~\cite{yu_localization} for finite complexes and was generalized to proper metric spaces by Qiao and Roe~\cite{qiao_roe}.
See also \cite[Chapters~6--7]{WillettYuHigherIndexTheory} for a self-contained development of analytic \(\K\)-homology based on localization algebras.\footnote{Note, however, that the localization algebras considered in \cite{WillettYuHigherIndexTheory} are slightly bigger and have somewhat better functoriality properties than the original versions.
We use the original version in this paper because it is more convenient for our subsequent construction of slant products at the cost of a slightly more awkward approach to functoriality, see~\cref{rem_better_functoriality_structgroup}.}
Related results that describe more general \(\KK\)-groups using different versions of the localization algebra can be found in~\cite{DWWLocalization}.

Finally, the ideal in $\Loc X$ consisting of all such functions $L$ with $L(1)=0$ is denoted by $\Locz X$. Its $\K$-theory is called the \emph{analytic structure group of $X$} and denoted $\Strg_\ast(X)\coloneqq\K_\ast(\Locz X)$.

These three \textCstar\nobreakdash-algebras fit into a short exact sequence
\[0\to\Locz X\to\Loc X\xrightarrow{\operatorname{ev}_1}\Roe X\to 0\]
and the induced long exact sequence is the \emph{Higson--Roe sequence}
\[\dots\to \K_{*+1}(\Roe X)\to \Strg_*(X)\to \K_*(X)\xrightarrow{\Ind} \K_*(\Roe X)\to\dots\]
with the map $\Ind=(\operatorname{ev}_1)_*$ induced by evaluation at $1$ being the index map.

Given another proper metric space $Y$, we consider the above-mentioned \textCstar-algebras of $Y$ and $X\times Y$ associated to a chosen ample representation $\rho_Y\colon \Cz(Y)\to\Lin(H_Y)$ and the corresponding tensor product representation $\rho_{X\times Y}\colon \Cz(X\times Y)\to \Lin(H_{X\times Y})$ on $H_{X\times Y}\coloneqq H_X\otimes H_Y$ which is again ample.

Note that the tensor product of two locally compact operators of finite propagation is again locally compact and of finite propagation. Hence
\begin{equation}
\Roe X\otimes\Roe Y\subset \Roe(X\times Y)\label{eq_precrossroe}
\end{equation}
 and this inclusion induces the external product
\[\times\colon\K_m(\Roe X)\otimes\K_n(\Roe Y)\to \K_{m+n}(\Roe(X\times Y))\,.\]

If $L_1$ and $L_2$ are functions in the generating subset of $\Loc X$, $\Loc Y$, respectively, then one readily verifies that the function
\[L\colon [1,\infty)\to\Roe(X\times Y)\,,\quad t\mapsto L_1(t)\otimes L_2(t)\]
also satisfies the propagation condition and yields an element of $\Loc(X\times Y)$. If $L_1\in\Locz X$, then we will have $L\in\Locz(X\times Y)$. This gives rise to isometric $*$-homomorphisms
\begin{align}
\Loc X\otimes \Loc Y&\to\Loc(X\times Y)\label{eq_precrossloc}
\\\Locz X\otimes \Loc Y&\to\Locz(X\times Y)\label{eq_precrosslocz}
\end{align}

\begin{rem}
The fact that this works with the minimal tensor product on the left hand side can be seen as follows.

There are canonical faithful representations of $\Loc X$, $\Loc Y$ and $\Loc(X\times Y)$ on the Hilbert spaces $\ell^2([1,\infty))\otimes H_X$, $\ell^2([1,\infty))\otimes H_Y$ and $\ell^2([1,\infty))\otimes H_{X\times Y}$, respectively. By the definition of the minimal tensor product we can thus see $\Loc X\otimes\Loc Y$ as a sub-\textCstar\nobreakdash-algebra of $\Lin(\ell^2([1,\infty))\otimes H_X\otimes\ell^2([1,\infty))\otimes H_Y)$. Now, conjugation by the Hilbert space projection
\[\ell^2([1,\infty))\otimes H_X\otimes \ell^2([1,\infty))\otimes H_Y\to\ell^2([1,\infty))\otimes H_{X\times Y}\]
onto the diagonal of $[1,\infty) \times [1,\infty)$ is a continuous linear map
\[\Lin(\ell^2([1,\infty))\otimes H_X\otimes\ell^2([1,\infty))\otimes H_Y)\to\Lin(\ell^2([1,\infty))\otimes H_{X\times Y})\]
which is not a $*$-homomorphism. But its restriction to the minimal tensor product $\Loc X\otimes\Loc Y$ is an isometric $*$-homomorphism and has image contained in the cannonically embedded sub-\textCstar\nobreakdash-algebra $\Loc(X\times Y)$, as can be seen on generators as above.
\end{rem}

The two $*$-homomorphisms \eqref{eq_precrossloc} and \eqref{eq_precrosslocz} give rise to external product maps
\begin{align*}
\times\colon\K_m(X)\otimes\K_n(Y)&\to \K_{m+n}(X\times Y)\,,
\\\times\colon \Strg_m(X)\otimes\K_n(Y)&\to \Strg_{m+n}(X\times Y)\,.
\end{align*}

\begin{rem}\label{rem_signconvention}
The definitions of the three external products makes use of the (maximal) external tensor product functor in $\K$-theory, which we denote by
\[\exttensprod\colon \K_m(A)\otimes \K_n(B)\to \K_{m+n}(A\tensmax B)\]
in order to distinguish it from our external products.
Note that this functor is subject to a sign convention, cf.\ \cite[Remark 4.7.5]{higson_roe}\footnote{This reference treats only the minimal tensor product, but exactly the same is true for the maximal tensor product.\label{footnote_HR00_tensor}}.
We will use the sign convention which usually is used in the literature and which has the following compatibility with boundary maps, as seen in \cite[Proposition 4.7.6.(b)]{higson_roe}\footnote{see \cref{footnote_HR00_tensor}}. If \[0\to I\to A\to A/I\to 0\] is a short exact sequence of \textCstar\nobreakdash-algebras and $B$ another \textCstar\nobreakdash-algebra, then the sequence
\[0\to I\tensmax B\to A\tensmax B\to (A/I)\tensmax B\to 0\]
is also exact and the boundary maps of these two short exact sequences and the external tensor products satisfy the equation
\[\partial (x\boxtimes y)=\partial(x)\boxtimes y\]
for all $x\in \K_m(A/J)$ and $y\in \K_n(B)$. Using the graded commutativity of the external tensor product we see that the corresponding equation obtained by tensoring the short exact sequence with $B$ from the left and not from the right is only true up to a sign:
\[\partial (y\boxtimes x)=(-1)^n\cdot y\boxtimes \partial(x)\]
Note that this is the sign convention which makes the usual sign heuristics work: exchanging the order of the symbol $y$ of degree $n$ and the symbol $\partial$ of degree $-1$ in the last equation gives rise to the sign $(-1)^{n\cdot (-1)}$.
\end{rem}

The $*$-homomorphisms obtained from \eqref{eq_precrossroe}, \eqref{eq_precrossloc} and \eqref{eq_precrosslocz} by using the maximal tensor product\footnote{The maximal tensor product is needed here to make the upper row exact.} fit into a commutative diagram:
\[\xymatrix{
0\ar[r]& \Locz X\otimes_{\max}\Loc Y\ar[r]\ar[dd]&\Loc X\otimes_{\max}\Loc Y\ar[r]\ar[dd]& \Roe X\otimes_{\max}\Loc Y\ar[r]\ar[d]^{\id\otimes\operatorname{ev}_1}&0
\\&&&\Roe X\otimes \Roe Y\ar[d]&
\\0\ar[r]& \Locz (X\times Y)\ar[r]&\Loc(X\times Y)\ar[r]& \Roe(X\times Y)\ar[r]&0
}\]
Together with the fact that the external product of $\K$-theory is functorial and compatible with the connecting homomorphisms (cf.\ \cref{rem_signconvention})
this gives for any $z\in \K_n(Y)$ rise to an external product morphism between long exact sequences:
\begin{align*}
\xymatrix{
\K_{\ast+1}(\Roe X) \ar[r]^-{\partial} \ar[d]^-{-\times \Ind(z)} & \Strg_\ast(X) \ar[r] \ar[d]^-{-\times z} & \K_\ast(X) \ar[r]^-{\Ind} \ar[d]^-{-\times z} & \K_\ast(\Roe X) \ar[d]^-{-\times \Ind(z)}\\
\K_{\ast+1+n}(\Roe(X \times Y)) \ar[r]^-{\partial} & \Strg_{\ast+n}(X \times Y) \ar[r] & \K_{\ast+n}(X \times Y) \ar[r]^-{\Ind} & \K_{\ast+n}(\Roe(X \times Y))
}
\end{align*}

\subsection{Equivariant case}\label{subsec_crossProducts_equiv}

Let \((H_X, \rho_X)\) be an \(X\)-module as in \cref{subsec_crossProducts_nonequiv}.
We suppose in addition that \(X\) is furnished with a proper isometric action of a countable discrete group \(G\).
Furthermore, we assume that we have a unitary representation \(u_G \colon G \to \U(H_X)\) which turns \((\rho_X, u_G)\) into a covariant pair\footnote{The group \(G\) acts on functions on \(X\) by \((g \cdot f)(x) = f(g^{-1} x)\).
Being a covariant pair means that $u_G(g) \rho_X(f) u_G(g)^\ast = \rho_X(g \cdot f)$ for all $g\in G$ and $f\in \Cz(X)$.}.
Given this data, we say that \((H_X, \rho_X, u_G)\) is an \(X\)-\(G\)-module.
An \(X\)-\(G\)-module is called \emph{locally free} if for each finite subgroup \(F \subseteq G\) and any \(F\)-invariant Borel subset \(E \subseteq X\), there is a Hilbert space \(H_E\) such that \(1_E H_X\) and \(\elltwo(F) \otimes H_E\) are isomorphic as \(F\)-Hilbert spaces, where  \(\elltwo(F)\) is endowed with the left-regular representation and \(H_E\) is endowed with the trivial representation.
An \(X\)-\(G\)-module is \emph{ample} if it is ample as an \(X\)-module and locally free.
Ample \(X\)-\(G\)-modules always exist~\cite[Lemma~4.5.5]{WillettYuHigherIndexTheory}.

In the following, we let \((H_X, \rho_X, u_G)\) be a fixed ample \(X\)-\(G\)-module.
We also fix an ample \(Y\)-\(H\)-module \((H_Y, \rho_Y, u_H)\), where \(Y\) is another proper metric space furnished with a proper isometric action of some countable discrete group \(H\).

We get \textCstar\nobreakdash-algebras $\Roe[G] X$, $\Loc[G] X$ and $\Locz[G] X$ by considering \emph{equivariant} locally compact operators of finite propagation, respectively suitable families $L$ of them.
Similarly as before, their \(\K\)-theory groups do not depend on the choice of ample \(X\)-\(G\)-module up to canonical isomorphism~\cite[Theorems~5.2.6, 6.5.7, Proposition~6.6.2]{WillettYuHigherIndexTheory}.
They fit into the short exact sequence
\[0\to\Locz[G] X\to\Loc[G] X\xrightarrow{\operatorname{ev}_1}\Roe[G] X\to 0\]
and the induced long exact sequence is then denoted
\[\dots\to \K_{*+1}(\Roe[G] X)\to \Strg_*^G(X)\to \K_*^G(X)\xrightarrow{\Ind} \K_*(\Roe[G] X)\to\dots\]
Writing here $\K_*^G(X)$ for $\K_\ast(\Loc[G] X)$ is justified by the fact that $\K_\ast(\Loc[G] X)$ is naturally isomorphic to the equivariant $\K$-homology of $X$, see~\cite[Proposition~6.6.2]{WillettYuHigherIndexTheory}.
If $G$ acts freely on $X$, then we have
\[\K_\ast(\Loc[G] X) \cong \K_*^G(X) \cong \K_\ast(G \backslash X),\]
compare~\cref{subsec_compatibility_equiv_nonequiv}.

As in the non-equivariant case we can construct now the external products and get a commutative diagram
\begin{align*}
\resizebox{\textwidth}{!}{
\xymatrix{
0\ar[r]& \Locz[G] X\otimes_{\max}\Loc[H] Y\ar[r]\ar[dd]&\Loc[G] X\otimes_{\max}\Loc[H] Y\ar[r]\ar[dd]& \Roe[G] X\otimes_{\max}\Loc[H] Y\ar[r]\ar[d]^{\id\otimes\operatorname{ev}_1}&0
\\&&&\Roe[G] X\otimes \Roe[H] Y\ar[d]&
\\0\ar[r]& \Locz[G\times H] (X\times Y)\ar[r]&\Loc[G\times H](X\times Y)\ar[r]& \Roe[G\times H](X\times Y)\ar[r]&0
}
}
\end{align*}
which induces for any $z\in \K_n^H(Y)$ the external product morphisms between the corresponding long exact sequences:
\begin{align*}
\resizebox{\textwidth}{!}{
\xymatrix{
\K_{\ast+1}(\Roe[G] X) \ar[r]^-{\partial} \ar[d]^-{-\times \Ind(z)} & \Strg^G_\ast(X) \ar[r] \ar[d]^-{-\times z} & \K^G_\ast(X) \ar[r]^-{\Ind} \ar[d]^-{-\times z} & \K_\ast(\Roe[G] X) \ar[d]^-{-\times \Ind(z)}\\
\K_{\ast+1+n}(\Roe[G\times H](X \times Y)) \ar[r]^-{\partial} & \Strg^{G\times H}_{\ast+n}(X \times Y) \ar[r] & \K^{G\times H}_{\ast+n}(X \times Y) \ar[r]^-{\Ind} & \K_{\ast+n}(\Roe[G\times H](X \times Y))
}
}
\end{align*}

If \(G\) acts cocompactly on \(X\), then \(\Roe[G] X\) is Morita equivalent to \(\CstarRed G\), see~\cite[Proposition~5.3.4]{WillettYuHigherIndexTheory}.

On the universal covers of finite complexes, the action is proper, free and cocompact.
In this way we get the diagram considered in \cref{thm:topologicalInj}.

\section{Slant products}
\label{sec_slant_products}

The goal of this section is to construct the various slant products and prove \cref{main_thm_intro_slant}.
We start by giving the definition of the stable Higson compactification and corona of a proper metric space $Y$ (which is usually non-compact) and of the coarse co-assembly map from \cite{EmeMey}.

If $\vartheta\colon Y\to Z$ is a map into another metric space $Z$, then one defines for each $r>0$ the \emph{$r$-variation} of $\vartheta$ as the function
\[\Var_{r}\vartheta\colon Y\to [0,\infty)\,,\quad x \mapsto \sup \{ d(\vartheta(x),\vartheta(y)) \mid y\in Y\text{ with }d(x,y) \leq r\}\,.\]
The funtion $\vartheta$ is said to have \emph{vanishing variation} if $\Var_{r}\vartheta$ converges to zero at infinity for all $r>0$.

Let $\Kom \coloneqq \Kom(\elltwo)$ denote the compact operators on the standard Hilbert space $\elltwo \coloneqq \elltwo(\N)$.
The \emph{stable Higson compactification} of $Y$, denoted by $\sHigCom Y$, is the \textCstar\nobreakdash-algebra of all bounded, continuous functions of vanishing variation $Y \to \Kom$.
The \emph{stable Higson corona} of $Y$ is the quotient $\Cstar$-algebra $\sHigCor Y \coloneqq \sHigCom Y / \Cz(Y, \Kom)$.

If $Y$ is unbounded, the stable Higson compactification and corona $\sHigCom Y$ and $\sHigCor Y$ contain an isometrically embedded copy of $\Kom$ as the sub-\textCstar\nobreakdash-algebra of constant functions. Their reduced $\K$-theory is defined as
\begin{align*}
\tilde\K_*(\sHigCom Y)&\coloneqq \K_*(\sHigCom Y)/\im\left(\Z\cong\K_*(\Kom)\to\K_*(\sHigCom Y)\right)
\\\tilde\K_*(\sHigCor Y)&\coloneqq \K_*(\sHigCor Y)/\im\left(\Z\cong\K_*(\Kom)\to\K_*(\sHigCor Y)\right)
\end{align*}
though beware that this definition is only reasonable for unbounded metric spaces \cite[Remark~3.9]{EmeMey}.

There are reduced versions of the above two $\Cstar$-algebras whose $\K$-theory behaves better for bounded spaces:
Let $\sHigComRed Y$ be the $\Cstar$-algebra of all bounded, continuous functions of vanishing variation $Y \to \Lin(\elltwo)$ with $f(x) - f(x^\prime) \in \Kom$ for all $x, x^\prime \in Y$.
Let $\sHigCorRed Y \coloneqq \sHigComRed Y / \Cz(Y, \Kom)$. We have the isomorphisms
\[\K_*(\sHigComRed Y)\cong\tilde\K_*(\sHigCom Y)\, , \quad \K_*(\sHigCorRed Y)\cong\tilde\K_*(\sHigCor Y)\]
if $Y$ is unbounded \cite[Proposition~5.5]{EmeMey} and for bounded $Y$ we simply define the reduced $\K$-theory by these isomorphisms.

The coarse co-assembly map is the connecting homomorphism
\[\mu^*\colon \K_*(\sHigCorRed Y)\to \K^{1-*}(Y)\]
associated to the short exact sequence
\[0\to \Cz(Y, \Kom)\to \sHigComRed Y\to \sHigCorRed Y\to 0\,.\]
In the rest of the paper, we will only use the reduced versions of the stable Higson compactification and corona.

\subsection{Construction of the slant products}\label{sec_slantconstruction}
From now on we assume that $X$ and $Y$ are proper metric spaces, and that $Y$ has bounded geometry (see \cref{defn:boundedGeometry}.\ref{defn_bounded_geom}).
The representations $\rho_X$, $\rho_Y$ and $\rho_{X\times Y}$ on the Hilbert spaces $H_X$, $H_Y$ and $H_{X\times Y}\coloneqq H_X\otimes H_Y$ shall be exactly as in \cref{sec:crossProducts}.

\begin{defn}[Bounded geometry]\label{defn:boundedGeometry}
The term \emph{bounded geometry} has in the literature different meanings depending to which kind of object it refers to:
\begin{myenuma}
\item\label{defn_bounded_geom} A metric space $Y$ is said to have \emph{bounded geometry} if there exists \(r>0\) and a subset \(\hat{Y}\subset Y$ such that $Y=\bigcup_{\hat y\in \hat{Y}} \Ball_r(\hat y)\) and such that for each \(R>0\) there exists a constant \(K_{R}\) such that for every \(y \in Y\) the number of elements $\#(\hat{Y}\cap \clBall_R(y))$ is bounded by \(K_{R}\).

Here $\Ball_r(-)$ denotes the open ball and $\clBall_r(-)$ denotes the closed ball of radius $r$.
\item\label{defn_bounded_geom_uniform} A metric space $Y$ is said to have \emph{continuously bounded geometry} if for every $r>0$ and \(R > 0\) there exists a constant $K_{r,R}>0$ such that the following conditions hold.
\begin{itemize}
  \item For every \(r > 0\) there is a subset $\hat{Y}_r\subset Y$ such that $Y=\bigcup_{\hat y\in \hat{Y}_r}\Ball_r(\hat y)$ and such that for all $r,R>0$ and \(y \in Y\) the number $\#(\hat{Y}_r\cap \clBall_R(y))$ is bounded by $K_{r,R}$.
  \item For all \(\alpha>0\), we have \(K_\alpha\coloneqq\limsup_{r\to 0}K_{r,\alpha r}<\infty\).
\end{itemize}

\item\label{defn_bounded_geom_manifolds} A complete Riemannian manifold has \emph{bounded geometry} if it has uniformly positive injectivity radius and the curvature tensor and all its covariant derivatives are uniformly bounded.\footnote{Note that completeness is actually a redundant requirement here since it is implied by having a uniformly positive injectivity radius.}

\item\label{defn_bounded_geom_simplicial} A simplicial complex is said to have \emph{bounded geometry} if there is a uniform bound on the number of simplices in the link of each vertex.

Note that this is equivalent to the simplicial complex being uniformly locally finite and finite-dimensional.
\qedhere
\end{myenuma}
\end{defn}

Continuously bounded geometry implies bounded geometry. Bounded geometry of Riemannian manifolds is an even stronger property: It implies continuously bounded geometry of the underlying metric space \cite[Lemma~5.2]{WulffTwisted}. Further, if a simplicial complex has bounded geometry, then the underlying metric space will have continuously bounded geometry.

If a metric space $Y$ is uniformly discrete (for instance, a countable discrete group equipped with a proper, left-invariant metric), then it has bounded geometry if and only if it has continuously bounded geometry (and in this case we can take $\hat Y_r \coloneqq Y$ for each $r > 0$).

\subsubsection{Slant product for the Roe algebra}
\label{subsubsec_slant_Roe}

We begin by constructing the slant product
\begin{equation*}
/\colon \K_p(\Roe(X \times Y)) \otimes \K_{1-q}(\sHigCorRed Y) \to \K_{p-q}(\Roe X)
\end{equation*}
for $X$, $Y$ proper metric spaces, and $Y$ of bounded geometry (\cref{defn:boundedGeometry}.\ref{defn_bounded_geom}).

To construct this slant product we will take the Roe algebra $\Roe X$ on the right hand side not to be the one associated to the representation $\rho_X$, that is $\Roe(\rho_X)$, but the one associated to the representation
\begin{equation}
\tilde\rho_X\coloneqq \rho_X\otimes\id_{H_Y\otimes\elltwo}\colon \Cz(X)\to \Lin(\tilde H_X)\label{eq_varRepOfCzX}
\end{equation}
on $\tilde H_X\coloneqq H_X\otimes H_Y\otimes \elltwo$,
that is $\Roe(\tilde\rho_X)$. Note that the representation $\tilde\rho_X$ is ample and non-degenerate, since $\rho_X$ is.

Let $\Fipro(\tilde\rho_X) \subset \cB(\tilde H_X)$ denote the sub-\textCstar\nobreakdash-algebra generated by all finite propagation operators. Our notation comes from the fact that this \textCstar\nobreakdash-algebra is slightly bigger than the \textCstar\nobreakdash-algebra generated by all pseudolocal operators of finite propagation, which is usually denoted by $\mathrm{D}^*(X)$.
\begin{lem}\label{lem_RoeidealinFipro}
The Roe algebra $\Roe(\tilde\rho_X)$ is an ideal in $\Fipro(\tilde\rho_X)$.
\end{lem}
\begin{proof}
Let $S\in\Roe(\tilde\rho_X)$ and $T\in\Fipro(\tilde\rho_X)$ have finite propagation. Then $S\circ T$ and $T\circ S$ both have finite propagation, too.

To show local compactness, let $g\in \Cz(X)$. We may assume that the support of $g$ is compact and then choose a function $g'\in\Cz(X)$ which is constantly equal to $1$ on the $R$-neighborhood of the support of $g$, where $R>0$ is the propagation of $T$.
Then the four operators
\begin{align*}
\tilde\rho_X(g)\circ (S\circ T)&=(\tilde\rho_X(g)\circ S)\circ T
\\(T\circ S)\circ \tilde\rho_X(g)&=T\circ (S\circ \tilde\rho_X(g))
\\(S\circ T)\circ \tilde\rho_X(g)&=(S\circ \tilde\rho_X(g'))\circ T\circ \tilde\rho_X(g)
\\\tilde\rho_X(g)\circ (T\circ S)&=\tilde\rho_X(g)\circ T\circ\tilde\rho_X(g')\circ S
\end{align*}
are clearly compact, because $S$ is locally compact. Density arguments finish the proof.
\end{proof}

Observe that the tensor product of the representation $\rho_Y$ and the canonical representation of $\Kom$ on $\elltwo$ is a non-degenerate representation of $\Cz(Y, \Kom)$ on $H_Y \otimes \elltwo$.
Thus it extends uniquely to a strictly continuous representation
\begin{equation}\label{eq_rep_of_Cb}
\bar\rho_Y\colon \Mult(\Cz(Y, \Kom))\to \cB(H_Y \otimes \ell^2)
\end{equation}
of the multiplier algebra and subsequent tensoring with the identity on $H_X$ gives us a representation
\begin{equation*}
\tilde\rho_Y\coloneqq \id_{H_X}\otimes \bar\rho_Y\colon\Mult(\Cz(Y, \Kom)) \to \cB(\tilde H_X)
\end{equation*}
whose image clearly commutes with the image of the representation $\tilde\rho_X$.
Since $\Cz(Y, \Kom)$ is an essential ideal in $\Cb(Y, \Kom)$, the latter embedds canonically into $\Mult(\Cz(Y, \Kom))$ and we denote the restrictions of $\bar\rho_Y$ and $\tilde\rho_Y$ to  $\Cb(Y, \Kom)$ and  $\Cz(Y, \Kom)$ by the same letters.

\begin{lem}\label{lem_slant_Roe_cycles_idealizer}
The images of the two representations
\begin{align*}
\tau\colon\Roe(\rho_{X\times Y})&\to \Lin(\tilde H_X) \text{ given by } S\mapsto S\otimes\id_{\elltwo}\,,
\\\tilde\rho_Y\colon\Mult(\Cz(Y,\Kom)) &\to \cB(\tilde H_X) \text{ defined above}\,,
\end{align*}
are contained in $\Fipro(\tilde\rho_X)$. 
\end{lem}
\begin{proof}
The propagation of each operator in the image of $\tilde\rho_Y$ is clearly zero. For $\tau$, the claim follows from the following important lemma.
\end{proof}

\begin{lem}\label{lem_propagtioncomparision}
If $S\in\Lin(H_{X\times Y})$ has finite propagation bounded by $R>0$ with respect to $\rho_{X\times Y}$, then the operator $S\otimes\id_{\elltwo}\in\Lin(\tilde H_X)$ has finite propagation bounded by $R$ with respect to $\tilde\rho_X$.
\end{lem}
\begin{proof}
Let $g,h\in\Cz(X)$ be functions whose supports are further apart than~$R$. Then $\tilde\rho_X(g)\circ (S\otimes\id_{\elltwo})\circ\tilde\rho_X(h)$ can be written as a strong limit of operators of the form $(\rho_{X\times Y}(g\otimes \varphi)\circ S\circ \rho_{X\times Y}(h\otimes \varphi))\otimes\id_{\elltwo}$ with $\varphi\in \Cz(Y)$, but the latter are all zero, because the supports of $g\otimes \varphi,h\otimes \varphi\in\Cz(X\times Y)$ are further apart than $R$.
\end{proof}

\begin{lem}\label{lem_slant_Roe_cycles_commute}
The image of the representation $\tau$ commutes up to $\Roe(\tilde\rho_X)$ with the image of $\sHigComRed Y$ under the representation $\tilde\rho_Y$.
\end{lem}

Hence, by the universal property of the maximal tensor product, we get an induced $*$-homomorphism
\begin{equation}
\label{eq_starhom_Phi}
\Phi\colon\Roe(\rho_{X\times Y})\otimes_{\max} \sHigComRed Y\to \Fipro(\tilde\rho_X)/\Roe(\tilde\rho_X)
\end{equation}
given by $S\otimes f\mapsto [\tau(S)\circ\tilde\rho_Y(f)]$.

\begin{proof}
Let $S\in\Roe(\rho_{X\times Y})$ and $f\in\sHigComRed Y$. We may assume that $S$ has finite propagation with respect to $\rho_{X\times Y}$ and hence $\tau(S)$ also has finite propagation with respect to $\tilde\rho_X$ by the \cref{lem_propagtioncomparision}. As $\tilde\rho_Y(f)$ also has finite propagation (namely zero) with respect to $\tilde\rho_X$, so does the commutator $[\tau(S),\tilde\rho_Y(f)]$.

Therefore it remains to show that the commutator is also locally compact with respect to $\tilde\rho_X$. Here we need the bounded geometry of $Y$.

Choose a subset $\hat Y\subset Y$ as in \cref{defn:boundedGeometry}.\ref{defn_bounded_geom}, i.\,e.\ with $Y=\bigcup_{y\in \hat{Y}}\Ball_r(\hat y)$ and such that for each $R'>0$ the number $\#(\hat{Y}\cap \clBall_{R'}(y))$ is uniformly bounded in $y\in Y$ by some constant $K_{R'}>0$. In this proof, the relevant value for $R'$ will be $R'=R+2r$, where $R$ is the propagation of the operator $S$.

By thinning out the open cover $\{\Ball_{r}(\hat{y})\}_{\hat{y} \in \hat{Y}}$, we obtain a decomposition of $Y$ into a family  $\{Z_{\hat{y}}\}_{\hat{y} \in \hat{Y}}$ of pairwise disjoint Borel subsets such that $Z_{\hat{y}} \subseteq \Ball_r(\hat{y})$ for all $\hat{y} \in \hat{Y}$.
The representation \(\rho_Y \colon \Cz(Y) \to \cB(H_Y)\) extends uniquely to the bounded Borel functions on \(Y\) subject to the condition that pointwise converging uniformly bounded sequences of functions are taken to strongly converging sequences of operators.
For a Borel subset \(Z \subseteq Y\), let \(1_Z \in \cB(H_Y)\) denote the projection corresponding to the characteristic function of \(Z\).
Now consider the strongly convergent series
\[\hat{f} \coloneqq \sum_{\hat{y} \in \hat{Y}} 1_{Z_{\hat{y}}} \otimes f(\hat{y}) \in \cB(H_Y \otimes \elltwo)\,.\]

The proof will be completed by showing first that the operators $(\id_{H_X}\otimes\hat f-\tilde\rho(f))\circ \tau(S)$ and $\tau(S)\circ (\id_{H_X}\otimes\hat f-\tilde\rho(f))$ are locally compact and second that the commutator $[\tau(S),\id_{H_X}\otimes\hat f]$ is locally compact.

Let $h_1\leq h_2\leq\dots$ be a sequence of compactly supported functions $Y\to[0,1]$ such that the compact subsets $h_n^{-1}\{1\}$ exhaust $Y$ as $n\to\infty$. Furthermore, let $P_n\in\Kom$ denote the projection onto the span of the first $n$ basis vectors $\delta_1,\dots,\delta_n\in\elltwo$.
It is clear from the construction of $\hat f$ together with the vanishing variation of $f$ that
\[(\id_{H_X}\otimes\hat f-\tilde\rho_Y(f))=\lim_{n\to\infty}(\id_{H_X}\otimes\hat f-\tilde\rho_Y(f))\circ \tilde\rho_Y(h_n\otimes P_n)\]
with convergence in norm.
Using this, we find for every $g\in\Cz(X)$ the equation
\begin{align*}
\tilde\rho_X(g)\circ &(\id_{H_X}\otimes\hat f-\tilde\rho_Y(f))\circ \tau(S)=
\\&=\lim_{n\to\infty}(\id_{H_X}\otimes\hat f-\tilde\rho_Y(f))\circ ((\rho_{X\times Y}(g\otimes h_n)\circ S)\otimes P_n)
\end{align*}
where the right hand side is a norm limit of compact operators, hence itself compact.

If $g\in \Cz(X)$ has compact support, then as in the proof of \cref{lem_RoeidealinFipro} we can choose $g'\colon X\to [0,1]$ of compact support which is equal to $1$ on the $R$-neighborhood of the support of $g$ and one similarly obtains compactness the operator
\[(\id_{H_X}\otimes\hat f-\tilde\rho_Y(f))\circ\tau(S)\circ\tilde\rho_X(g)=(\id_{H_X}\otimes\hat f-\tilde\rho_Y(f))\circ\tilde\rho_X(g')\circ\tau(S)\circ\tilde\rho_X(g)\,.
\]
We have thus shown that $(\id_{H_X}\otimes\hat f-\tilde\rho_Y(f))\circ\tau(S)$ is locally compact and analogously we obtain local compactness of $\tau(S)\circ(\id_{H_X}\otimes\hat f-\tilde\rho_Y(f))$, too.

It remains to show local compactness of the commutator
\begin{align*}
[\tau(S),\hat f]&=\tau(S)\circ\sum_{\hat z\in\hat Y}\id_{H_X}\otimes 1_{Z_{\hat{z}}} \otimes f(\hat{z})-\sum_{\hat y\in\hat Y}(\id_{H_X}\otimes 1_{Z_{\hat{y}}} \otimes f(\hat{y}))\circ\tau(S)
\\&=\sum_{\hat y,\hat z\in\hat Y} \left( (\id_{H_X}\otimes 1_{Z_{\hat{z}}})\circ S\circ (\id_{H_X}\otimes 1_{Z_{\hat{y}}}) \right)\otimes \left(f(\hat z)-f(\hat y)\right)
\\&=\sum_{\begin{smallmatrix}\hat y,\hat z\in\hat Y\\\mathclap{d(\hat y,\hat z)\leq R+2r}\end{smallmatrix}} \left( (\id_{H_X}\otimes 1_{Z_{\hat{z}}})\circ S\circ (\id_{H_X}\otimes 1_{Z_{\hat{y}}}) \right)\otimes \left(f(\hat z)-f(\hat y)\right)
\end{align*}
where the sums converge a priori in the strong operator topology. Let us first show that the last sum converges even in norm. Because of the vanishing variation of $f$ there is a finite subset $L\subset\hat Y$ such that $\|f(\hat z)-f(\hat y)\|<\varepsilon$ whenever $\hat y,\hat z\in\hat Y\setminus L$ satisfy $d(\hat y,\hat z)\leq R+2r$. For arbitrary $v\in \tilde H_X$ the vectors $v_{\hat y}\coloneqq (\id_{H_X}\otimes 1_{Z_{\hat{y}}}\otimes\id_{\elltwo})v$ for $\hat y\in\hat Y$ are pairwise orthogonal and $v=\sum_{\hat y\in\hat Y}v_{\hat y}$, hence $\|v\|^2=\sum_{\hat y\in\hat Y}\|v_{\hat y}\|^2$.
Then the calculation
\begin{align*}
&\bigg\|\sum_{\begin{smallmatrix}\hat y,\hat z\in\hat Y\setminus L\\d(\hat y,\hat z)\leq R+2r\end{smallmatrix}} \left( (\id_{H_X}\otimes 1_{Z_{\hat{z}}})\circ S\circ (\id_{H_X}\otimes 1_{Z_{\hat{y}}}) \right)\otimes \left(f(\hat z)-f(\hat y)\right)v\bigg\|^2=
\\&\ =\sum_{\hat z\in\hat Y\setminus L}\bigg\|\sum_{\begin{smallmatrix}\hat y\in\hat Y\setminus L\\d(\hat y,\hat z)\leq R+2r\end{smallmatrix}} \left( (\id_{H_X}\otimes 1_{Z_{\hat{z}}})\circ S\circ (\id_{H_X}\otimes 1_{Z_{\hat{y}}}) \right)\otimes \left(f(\hat z)-f(\hat y)\right)v\bigg\|^2
\\&\ \leq \sum_{\hat y\in\hat Y\setminus L}\left(K_{R+2r}\cdot\|S\|\cdot \varepsilon\cdot \|v_{\hat y}\|\right)^2
\\&\ \leq K_{R+2r}^2\cdot\|S\|^2\cdot \varepsilon^2\cdot\|v\|^2
\end{align*}
shows the claimed norm convergence.

Now, given $g\in\Cz(X)$ and choosing for each $\hat z\in\hat Y$ a compactly supported function $h_{\hat z}$ which is constantly $1$ on $\Ball_r(\hat z)$, we find that
\begin{align*}
&\tilde\rho_X(g)\circ[\tau(S),\hat f]=
\\&=\sum_{\begin{smallmatrix}\hat y,\hat z\in\hat Y\\\mathclap{d(\hat y,\hat z)\leq R+2r}\end{smallmatrix}} \!\bigg( (\id_{H_X}\otimes 1_{Z_{\hat{z}}})\circ \underbrace{\rho_{X\times Y}(g\otimes h_{\hat z})\circ S}_{\in \Kom(H_{X\times Y})}\circ (\id_{H_X}\otimes 1_{Z_{\hat{y}}}) \bigg)\otimes \bigg(\underbrace{f(\hat z)-f(\hat y)}_{\in\Kom}\bigg)
\end{align*}
is a norm convergent sum of compact operators, hence itself compact. Analogously, $[\tau(S),\hat f]\circ\tilde\rho_X(g)$ is compact, and hence the commutator $[\tau(S),\hat f]$ is locally compact.
\end{proof}

\begin{lem}\label{lem_slant_Roe_cycles_factors}
The $*$-homomorphism $\Phi$ from \eqref{eq_starhom_Phi} factors through the tensor product $\Roe(\rho_{X \times Y}) \otimes_{\max} \sHigCorRed Y$.
In other words, it defines a $*$-homomorphism
\begin{equation}
\label{eq_slant_Roe_cycles_quotient}
\Psi\colon\Roe(\rho_{X \times Y}) \otimes_{\max} \sHigCorRed Y \to  \Fipro(\tilde\rho_X)/\Roe(\tilde\rho_X).
\end{equation}
\end{lem}
\begin{proof}
Due to exactness of the maximal tensor product, the claim is equivalent to $\Phi$ vanishing on $\Roe(\rho_{X \times Y}) \otimes \Cz(Y,\Kom)=\Roe(\rho_{X \times Y}) \otimes \Cz(Y)\otimes \Kom$.
Therefore, given operators $S\in \Roe(\rho_{X \times Y})$ and $f\otimes T\in \Cz(Y)\otimes \Kom$ we have to show that $\tau(S)\circ\tilde\rho_Y(f\otimes T)\in\Roe(\tilde\rho_X)$.

The finite propagation part of this statement is proven exactly as the one in the preceding lemma. For local compactness we use the formula
\[\tau(S)\circ\tilde\rho_Y(f\otimes T)\circ\tilde\rho_X(g)=(S\circ\rho_{X\times Y}(g\otimes f))\otimes T\in\Kom(H_{X\times Y})\otimes\Kom(\elltwo)\]
and a similar one for $\tilde\rho_X(g)\circ\tau(S)\circ\tilde\rho_Y(f\otimes T)$, involving the same function $g'$ as in the previous lemma.
\end{proof}

\begin{defn}\label{defn_coarseslant}
The slant product between $\K$-theory of the Roe algebra and the $\K$-theory of the reduced stable Higson corona is now defined as $(-1)^p$ times the composition
\begin{align}
\K_p(\Roe(X \times Y)) \otimes \K_{1-q}(\sHigCorRed Y) \ &= \K_p(\Roe(\rho_{X \times Y})) \otimes \K_{1-q}(\sHigCorRed Y) \notag\\
& \to \ \K_{p+1-q}(\Roe(\rho_{X \times Y}) \otimes_{\max} \sHigCorRed Y)\notag\\
& \xrightarrow{\Psi_*} \ \K_{p+1-q}(\Fipro(\tilde\rho_X)/\Roe(\tilde\rho_X))\notag\\
& \xrightarrow{\partial} \ \K_{p-q}(\Roe(\tilde\rho_X))\notag\\
& = \K_{p-q}(\Roe X),\label{eq_slant_Roe_unreduced}
\end{align}
where the first arrow is the external product on $\K$-theory, and the third arrow the boundary operator in the corresponding long exact sequence.
\end{defn}

\subsubsection{Slant products for the localization algebras}\label{subsubsec_slant_localization}

To construct the analogous slant products for the localization algebras
we use the same approach as for the construction of the slant product on the Roe algebra in the previous section.

We define $\FiproLoc(\tilde{\rho}_X)$ as the $\Cstar$-subalgebra of $\Cb([1,\infty),\Fipro(\tilde{\rho}_X))$ generated by the bounded and uniformly continuous functions $S\colon [1,\infty) \to \Fipro(\tilde{\rho}_X)$ such that the propagation of $S(t)$ is finite for all $t \ge 1$ and tends to zero as $t \to \infty$.
Similarly we define $\FiproLocz(\tilde{\rho}_X)$ as the ideal in $\FiproLoc(\tilde{\rho}_X)$ consisting of all maps that vanish at~$1$.
Note that $\Loc(\tilde{\rho}_X)$ is an ideal in $\FiproLoc(\tilde{\rho}_X)$ and $\Locz(\tilde{\rho}_X)$ is even an ideal in all of the three $\FiproLoc(\tilde{\rho}_X)$, $\FiproLocz(\tilde{\rho}_X)$ and of course $\Loc(\tilde{\rho}_X)$.

\begin{lem}\label{lem_slant_Loc_technical}
The following analogues of \cref{lem_slant_Roe_cycles_idealizer,lem_slant_Roe_cycles_commute,lem_slant_Roe_cycles_factors} hold true:
\begin{enumerate}
\item
The images of the two isometric $*$-homomorphisms
\[\tau_\LSym\colon\Loc(\rho_{X\times Y})\to \Cb([1,\infty),\Lin(\tilde H_X))\,,\]
which is obained by applying the functor $\Cb([1,\infty),-)$ to $\tau$, and
\[\tilde\rho_{Y,\LSym}\colon\Mult(\Cz(Y,\Kom)) \xrightarrow{\tilde\rho_Y}\Lin(\tilde H_X)\xrightarrow[\text{as constant functions}]{\text{inclusion}} \Cb([1,\infty),\Lin(\tilde H_X))\,,\]
are contained in $\FiproLoc(\tilde\rho_X)$.
\item The image of $\tau_\LSym$ commutes up to $\Loc(\tilde\rho_X)$ with the image of $\sHigComRed Y$ under $\tilde\rho_{Y,L}$
and the image of $\Locz(\rho_{X\times Y})$ under $\tau_\LSym$ commutes up to $\Locz(\tilde\rho_X)$ with the image of $\sHigComRed Y$ under $\tilde\rho_{Y,L}$.
Hence they induce $*$-homomorphisms
\begin{align*}
\Phi_\LSym\colon\Loc(\rho_{X\times Y})\otimes_{\max} \sHigComRed Y &\to \FiproLoc(\tilde\rho_X)/\Loc(\tilde\rho_X)
\\\Phi_{\LSym,0}\colon\Locz(\rho_{X\times Y})\otimes_{\max} \sHigComRed Y &\to \FiproLoc(\tilde\rho_X)/\Locz(\tilde\rho_X)
\end{align*}
given by $S\otimes f\mapsto [\tau_\LSym(S)\circ\tilde\rho_{Y,\LSym}(f)]$ and the image of $\Phi_{\LSym,0}$ is even contained in $ \FiproLocz(\tilde\rho_X)/\Locz(\tilde\rho_X)$.
\item The $*$-homomorphisms $\Phi_\LSym$ and $\Phi_{\LSym,0}$ factor through $\Loc(\rho_{X\times Y})\otimes_{\max}\sHigCorRed Y$ and $\Locz(\rho_{X\times Y})\otimes_{\max}\sHigCorRed Y$, respectively.
That is, they define $*$-homomorphisms
\begin{align*}
\Psi_\LSym\colon\Loc(\rho_{X \times Y}) \otimes_{\max} \sHigCorRed Y &\to  \FiproLoc(\tilde\rho_X)/\Loc(\tilde\rho_X)\,,
\\\Psi_{\LSym,0}\colon\Locz(\rho_{X \times Y}) \otimes_{\max} \sHigCorRed Y &\to  \FiproLocz(\tilde\rho_X)/\Locz(\tilde\rho_X)\,.
\end{align*}
\end{enumerate}
\end{lem}

\begin{proof}
The estimates on the propagation in these lemmas rely on \cref{lem_propagtioncomparision} and on the fact that the propagation of the composition of operators is at most the sum of the propagations of the summands, and so we are still fine in our situation here. And due to our definition of $\Loc(-)$ we have to check local compactness in the proofs of the analogous versions of \cref{lem_slant_Roe_cycles_commute,lem_slant_Roe_cycles_factors} only point-wise in time, i.e.\ for fixed $t \in [1,\infty)$, and hence we can directly use the corresponding arguments from the proofs of the \cref{lem_slant_Roe_cycles_commute,lem_slant_Roe_cycles_factors}.
\end{proof}

\begin{defn}\label{defn_localizationslant}
The slant products
\begin{align*}
\K_p(X \times Y) \otimes \K_{1-q}(\sHigCorRed Y) & \to \K_{p-q}(X)\\
\Strg_p(X \times Y) \otimes \K_{1-q}(\sHigCorRed Y) & \to \Strg_{p-q}(X)
\end{align*}
are defined as $(-1)^p$ times compositions analogous to that of \cref{defn_coarseslant} but using the maps $\Psi_{\LSym}$ and $\Psi_{\LSym,0}$, respectively, instead of $\Psi$.
\end{defn}

\subsection{Compatibility with the Higson--Roe sequence}\label{sec_commutativity_diagram}
In this section we prove the following compatibility of the slant products with the Higson--Roe sequence.
\begin{thm}\label{thm_commutativity_diagram}
The diagram
\begin{align*}
\mathclap{
\xymatrix{
\Strg_p(X \times Y) \ar[r] \ar[d]^{/ \theta} & \K_p(X \times Y) \ar[r] \ar[d]^{/ \theta} & \K_p(\Roe(X \times Y)) \ar[d]^{/ \theta} \ar[r]^-{\partial} & \Strg_{p-1}(X \times Y) \ar[d]^{/ \theta}\\
\Strg_{p-q}(X) \ar[r] & \K_{p-q}(X) \ar[r] & \K_{p-q}(\Roe X) \ar[r]^-{\partial} & \Strg_{p-1-q}(X)
}
}
\end{align*}
commutes for every $\theta\in \K_{1-q}(\sHigCorRed Y)$.
\end{thm}

\begin{proof}
Consider the diagram
\begin{align*}
\resizebox{\textwidth}{!}{\xymatrix{
\K_p(\Locz(X \times Y)) \ar[r] \ar[d]^-{\exttensprod \theta} & \K_p(\Loc(X \times Y)) \ar[r] \ar[d]^-{\exttensprod \theta} & \K_p(\Roe(X \times Y)) \ar[d]^-{\exttensprod  \theta} \ar[r]^-{\partial} & \K_{p-1}(\Locz(X \times Y)) \ar[d]^-{\exttensprod \theta}
\\{\begin{matrix}\K_{p+1-q}(\Locz(X \times Y)\\\qquad\otimes_{\max}\sHigCorRed Y) \end{matrix}}\ar[r] \ar[d]^{(\Psi_{\LSym,0})_*}
& {\begin{matrix}\K_{p+1-q}(\Loc(X \times Y)\\\qquad\otimes_{\max}\sHigCorRed Y) \end{matrix}} \ar[r] \ar[d]^{(\Psi_{\LSym})_*}
& {\begin{matrix}\K_{p+1-q}(\Roe(X \times Y)\\\qquad\otimes_{\max}\sHigCorRed Y) \end{matrix}} \ar[d]^{\Psi_*} \ar[r]^-{\partial} &{\begin{matrix}\K_{p-q}(\Locz(X \times Y)\\\qquad\otimes_{\max}\sHigCorRed Y) \end{matrix}} \ar[d]^{(\Psi_{\LSym,0})_*}
\\\K_{p+1-q}\left(\frac{\FiproLocz X}{\Locz X}\right) \ar[r]\ar[d]^{\partial} & \K_{p+1-q}\left(\frac{\FiproLoc X}{\Loc X}\right) \ar[r]\ar[d]^{\partial} & \K_{p+1-q}\left(\frac{\Fipro X}{\Roe X}\right) \ar[r]^-{\partial}\ar[d]^{\partial} & \K_{p-q}\left(\frac{\FiproLocz X}{\Locz X}\right)\ar[d]^{\partial}
\\\K_{p-q}(\Locz X) \ar[r] & \K_{p-q}(\Loc X) \ar[r] & \K_{p-q}(\Roe X) \ar[r]^-{\partial} & \K_{p-1-q}(\Locz X)
}}
\end{align*}
whose rows are the long exact sequences in $\K$-theory induced by the obvious short exact sequences of \textCstar\nobreakdash-algebras and whose vertical compositions are the slant products defined in the previous section up to the signs $(-1)^p$ for the left three columns and $(-1)^{p-1}$ for the right column.

Commutativity of the upper three squares is a well known property of the external tensor product in $\K$-theory, i.\,e.\ functoriality for the left two squares and the sign convention of \cref{rem_signconvention} for the square to the right.

Commutativity of the three squares in the middle row and the left two squares  in the bottom row are due to naturality of the long exact sequence in $\K$-theory, considering that the three $*$-homomorphisms $\Psi_{\LSym,0}$, $\Psi_{\LSym}$ and $\Psi$ together comprise a morphism of short exact sequences and considering that we have the following commutative diagram with exact rows and collumns:
\[\xymatrix{
& 0 \ar[d] & 0 \ar[d] & 0 \ar[d] & \\
0 \ar[r] & \Locz (X) \ar[r] \ar[d] & \Loc(X) \ar[r] \ar[d] & \Roe(X) \ar[r] \ar[d] & 0\\
0 \ar[r] & \FiproLocz(X) \ar[r] \ar[d] & \FiproLoc(X) \ar[r] \ar[d] & \Fipro(X) \ar[r] \ar[d] & 0\\
0 \ar[r] & \FiproLocz(X) / \Locz(X) \ar[r] \ar[d] & \FiproLoc(X) / \Loc(X) \ar[r] \ar[d] & \Fipro(X) / \Roe(X) \ar[r] \ar[d] & 0\\
& 0 & 0 & 0 &
}\]
It is an abstract fact that the outer boundary maps associated to such a grid of exact sequences commute up to multiplication with $-1$.
This proves commutativity of the bottom right square up to \(-1\), and this extra sign is exactly the one needed to match the difference of the signs implemented in the slant products.
\end{proof}

\subsection{The slant product on \texorpdfstring{$\K$}{K}-homology}\label{sec_recovering_usual_slant_prod}
In this section we are going to compare the slant product for the localization algebra $\Loc$ with the usual slant product between $\K$-homology and $\K$-theory.
There are several ways to define the latter, so let us specify that the definition which we want to work with is the one obained from $E$-theory.

Recall that $E$-theory is a bivariant $\K$-theory for \textCstar\nobreakdash-algebras, i.\,e.\ it has properties analogous to those of $\KK$-theory, and it even agrees with $\KK$-theory on nuclear and separable \textCstar\nobreakdash-algebras. One recovers the $\K$-homology and $\K$-theory groups of a locally compact Hausdorff space $X$ as the special cases $\K_*(X)\cong E_*(\Cz(X),\C)$ and $\K^*(X)\cong E_{-*}(\C,\Cz(X))$.
The reason why we prefer $E$-theory over $\KK$-theory is that the localization algebras are closely related to asymptotic morphisms as in the definition of $E$-theory. Even more, the isomorphism between the $\K$-theory of the localization algebra and $\K$-homology which is given in \cite[Corollary 4.2 and Proposition 4.3]{qiao_roe} (see \cref{prop_KofLocisEhom} below) is in fact given as an isomorphism 
\[ \Delta \colon \K_*(\Loc X)\xrightarrow{\cong} E_*(\Cz(X),\C).\]

Our standard references for $E$-theory are the papers \cite{guentnerhigsontrout,guentnerhigson}. Note that they actually only consider the $E$-theory groups $E_0(-,-)$ in degree zero, but the higher $E$-theory groups are obtained from them in the usual way, i.\,e.\ by tensoring suitably with $\Cz(\R^n)$ (or with the $\Z_2$-graded Clifford algebras $\C\ell_n$), just as it is done in $\KK$-theory. Therefore, also the same sign heuristics as in $\KK$-theory apply.

There is just one subtle difference between $E$-theory and $\KK$-theory which one has to get right to make the sign heuristics work: In $\KK$-theory, the Kasparov product of $x\in\KK_m(A,B)$ and $y\in \KK_n(B,C)$ is usually written as $x\otimes_By\in\KK_{m+n}(A,C)$, whereas in $E$-theory, the composition product of $x\in E_m(A,B)$ and $y\in E_n(B,C)$ is usually written like a composition of functions as $y\circ x\in E_{m+n}(A,B)$. Comparing these two notations one realizes that the order of $x$ and $y$ is exchanged and hence it should be expected that they agree only up to a sign $(-1)^{mn}$. This is indeed the case, and it is due to the fact that in taking the products one has to choose identifications $\Cz(\R^m)\otimes\Cz(\R^n)\cong\Cz(\R^{m+n})$, and these identifications have to be chosen differently for $E$-theory than for $\KK$-theory to make the sign heurisics work.

Maybe the best way to visualize this is the commutative diagram
\[\xymatrix{
\KK_m(A,B)\otimes\KK_n(B,C)\ar[r]^-{\otimes_B}\ar[d]^{(-1)^{mn}\cdot\text{flip}}&\KK_{m+n}(A,C)\ar[d]
\\E_n(B,C)\otimes E_m(A,B)\ar[r]^-{\circ}&E_{m+n}(A,C)
}\]
where the left vertical arrow not only maps $\KK$-theory to $E$-theory and exchanges the two factors, but also multiplies by $(-1)^{mn}$.

Note that this difference sticks out in the special case $m+n=0$, $A=C=\C$ and $B=\Cz(X)$: For $x\in \K_n(X)$ and $y\in \K^n(X)$ the $E$-theoretic pairing is $\langle x,y\rangle\coloneqq x\circ y\in E_0(\C,\C)\cong\Z$ and the $\KK$-theoretic pairing is $\langle y,x\rangle\coloneqq y\otimes_{\Cz(X)}x\in \KK_0(\C,\C)\cong\Z$ and these two pairings differ by $(-1)^{-n^2}=(-1)^{n}$,
\[\langle y,x\rangle=(-1)^n \langle x,y\rangle\,,\]
and thanks to the choice of the order of $x,y$ in these notations of the pairing, the formula is again compatible with the sign heuristics. We will always use the $E$-theoretic version of the pairing in this paper.

\begin{defn}\label{defn_Eslant}
The slant product
\begin{align*}
\K_p(X\times Y)\otimes \K^q(Y) & \to \K_{p-q}(X)\\
x\otimes\theta & \mapsto x/\theta
\end{align*}
is defined as
\begin{alignat*}{2}
E_p(\Cz(X)\otimes \Cz(Y),\C)&\otimes E_{-q}(\C,\Cz(Y)) & \ \to \ & E_{p-q}(\Cz(X),\C)
\\x & \otimes \theta & \ \mapsto \ & x\circ (\id_{\Cz(X)}\boxtimes\theta)\,.
\end{alignat*}
Note that for $Y$ a one-point space and $p=q$ one recovers the pairing as a special case of the slant product: $\langle x,\theta\rangle=x/\theta$.
\end{defn}

The two properties of the following lemma are in accordance with the sign heuristics and their proofs are straightforward.

\begin{lem}\label{lem_crossslantcompatibilityinEtheory}
For $x\in \K_m(X)$, $y\in \K_p(Y\times Z)$ and $\theta\in \K^q(Z)$ we have
\[(x\times y)/\theta=x\times(y/\theta)\]
and for $x\in \K_p(X\times Y\times Z)$, $\eta\in \K^q(Y)$ and $\theta\in \K^r(Z)$ we have
\[x/(\eta\times\theta)=(-1)^{qr}\cdot (x/\theta)/\eta\,.\]
\end{lem}

The remainder of this section is devoted to proving the following comparision theorem between the $E$-theoretic slant product and the slant product which we had defined for the localization algebra.

\begin{thm}\label{thm_KHomSlantCompatible}
Let $Y$ have continuously bounded geometry.
Then the slant product on the localization algebra from \cref{defn_localizationslant} and the $E$-theoretic slant product are related via the co-assembly map $\mu^\ast\colon \K_{1-q}(\sHigCorRed Y) \to \K^q(Y)$ and the isomorphisms $\Delta\colon\K_*(\Loc(-))\xrightarrow{\cong}\K_*(-)$ by the commutative diagram
\[\xymatrix{
	\K_p(\Loc(X \times Y)) \otimes \K_{1-q}(\sHigCorRed Y)
		\ar[r]^-{/}
		\ar[d]_-{\Delta \otimes \mu^\ast}
	& \K_{p-q}(\Loc X ) \ar[d]_-{\Delta}^-{\cong}
	\\	\K_p(X \times Y) \otimes \K^q(Y)
	\ar[r]^-{/}
	&\K_{p-q}(X)
}\]
\end{thm}

Let us begin by recalling the isomorphisms $\Delta$ between the $\K$-theory of the localization algebra and the $\K$-homology groups.

\begin{prop}[compare {\cite[Corollary 4.2 and Proposition\ 4.3]{qiao_roe}}]\label{prop_KofLocisEhom}
Let $X$ be a proper metric space and $(H_X,\rho_X)$ an $X$-module.
Then
\begin{align*}
\delta(\rho_X)\colon\Loc(\rho_X)\otimes \Cz(X)&\to\Cb([1,\infty),\Kom(H_X))/\Cz([1,\infty),\Kom(H_X))
\\(T_t)_{t\in[1,\infty)}\otimes f&\mapsto [(T_t\circ \rho_X(f))_{t\in[1,\infty)}]
\end{align*}
defines an asymptotic morphism and hence a canonical element of $E_0(\Loc(\rho_X)\otimes \Cz(X),\C)$.
The isomorphism between the $\K$-theory of the localization algebra constructed with an ample $X$-module $(H_X,\rho_X)$ and $\K$-homology is given by the $E$-theoretic composition product
\begin{align*}
\Delta(\rho_X)\colon\K_m(\Loc(\rho_X))&\xrightarrow{\cong} E_m(\Cz(X),\C)\cong\K_m(X)
\\ x&\mapsto \delta(\rho_X)\circ (x\boxtimes \id_{\Cz(X)})\,.
\end{align*}
\end{prop}
For later purposes we note that $\delta(\rho_X)$ obviously factorizes through the quotient $(\Loc(\rho_X)/\Cz([1,\infty),\Roe(\rho_X)))\otimes \Cz(X)$.

The first part of this proposition is an easy corollary of the following well-known technical lemma, which we will also need to exploit several further times below.

\begin{lem}[{compare \cite[Proposition~4.1.]{qiao_roe}, \cite[Lemma~6.1.2]{WillettYuHigherIndexTheory}}]\label{lem:CommutatorVsPropagation}
  There exists a constant \(C > 0\) such that the following holds.
	Let \((H_X, \rho_X)\) be an \(X\)-module and \(\vartheta \colon X \to \C\) a bounded Borel function.
  Then for every operator $T \in \cB(K)$ the estimate
	\[ \| [T, \rho_X(\vartheta)] \| \leq C\cdot \|T\|\cdot\|\Var_{\propagation(T)}\vartheta\| \]
	 holds.
\end{lem}

Note that \(\vartheta\) is uniformly continuous iff \(\|\Var_{r}\vartheta\| \to 0\) as \(r \to 0\).

\begin{lem}\label{lem:asymptoticCommute}
	Let $(T_t)_{t \in [1,\infty)} \in \Loc(\rho_{X \times Y})$ and $f \in \Cz(Y, \Kom)$.
	Then
	\begin{equation} \label{eq:asymptoticCommutatorCz}
		\lim_{t \to \infty}\left[ \tau(T_t), \tilde\rho_Y(f) \right] = 0.
	\end{equation}
\end{lem}
\begin{proof}
Since $\Cz(Y, \Kom) = \Cz(Y) \otimes \Kom$, we assume by an approximation argument without loss of generality that $f = \tilde{f} \otimes K$, where $\tilde{f} \in \Cz(Y)$ and $K \in \Kom$.
Moreover, we can assume that $\lim_{t \to \infty} \propagation(T_t) = 0$.
Then \cref{lem:CommutatorVsPropagation} implies that there exists $C \geq 0$ such that
\[ \| [T_t, \id_{H_X} \otimes \rho_Y(\tilde{f})] \| \leq C \|T\|\,\|\Var_{\propagation(T_t)} 1 \otimes \tilde{f} \|.\]
We have \(\|\Var_r 1 \otimes \tilde{f}\| \to 0\) as \(r \to 0\) because \(1 \otimes \tilde{f} \colon X \times Y \to \C\) is uniformly continuous.
Hence $[T_t, \id_{H_X} \otimes \rho_Y(\tilde{f})] \to 0$ as $t \to \infty$.
We conclude that
\begin{align*}
 \lim_{t \to \infty} \left[ \tau(T_t), \tilde\rho_Y(f) \right]
 &= \lim_{t \to \infty} \left[ T_t \otimes \id_{\elltwo}, \id_{H_X} \otimes \rho_Y(\tilde{f}) \otimes K \right] \\
 &= \lim_{t \to \infty} \left[ T_t, \id_{H_X} \otimes \rho_Y(\tilde{f}) \right] \otimes K = 0.
\qedhere
\end{align*}
\end{proof}

  \Cref{lem:CommutatorVsPropagation} only applies to scalar-valued functions.
  Therefore the tensor product decomposition \(\Cz(Y, \Kom) = \Cz(Y) \otimes \Kom\) was a crucial aspect in the argument above.
  In particular, this argument does not prove \labelcref{eq:asymptoticCommutatorCz} for functions in the stable Higson compactification, although it would work for the usual (unstable) Higson compactification.
  However, under the assumption of continuously bounded geometry, a more difficult argument yields \labelcref{eq:asymptoticCommutatorCz} also for functions in the stable Higson compactification, which is the content of the following proposition.

\begin{prop}\label{prop:asymptoticCommuteBndGmtry}
\begin{sloppypar}
Suppose that $Y$ has continuously bounded geometry.
Let $(T_t)_{t \in [1,\infty)} \in \Loc(\rho_{X\times Y})$ and $f \colon Y \to \Lin(\elltwo)$ bounded and uniformly continuous.
Then
\begin{equation} \label{eq:asymptoticCommutator}
	\lim_{t \to \infty}\left[ \tau(T_t), \tilde\rho_Y(f) \right] = 0.
\end{equation}
\end{sloppypar}
\end{prop}
\begin{proof}
	The argument is similar as in the proof of \cite[Lemma~5.6]{WulffTwisted} and \cref{lem_slant_Roe_cycles_commute} above.
	First, we assume as usual that $\lim_{t \to \infty} \propagation(T_t) = 0$.
	Now let $\varepsilon > 0$.
	Choose $\delta > 0$ such that $\|f(y) - f(y^\prime) \| < \varepsilon$ if $d(y, y^\prime) < \delta$.
	Let $r \coloneqq \delta / 4$.
 	Let $\hat{Y}_r \subseteq Y$ be a subset witnessing our conditions for continuously bounded geometry.
	After perhaps making $\delta  = 4r$ smaller, we can ensure that $K_{r, 4 r} \leq K_4 + 1 < \infty$.
	By thinning out the cover $(\Ball_{r}(\hat{y}))_{\hat{y} \in \hat{Y}_r}$, we obtain a pairwise disjoint cover $(Z_{\hat{y}})_{\hat{y} \in \hat{Y}_r}$ consisting of Borel sets such that $Z_{\hat{y}} \subseteq \Ball_r(\hat{y})$ for all $\hat{y} \in \hat{Y}_r$.
  The representation \(\rho_Y \colon \Cz(Y) \to \cB(H_Y)\) extends uniquely to the bounded Borel functions on \(Y\) subject to the condition that pointwise converging uniformly bounded sequences of functions are taken to strongly converging sequences of operators.
  For a Borel subset \(Z \subseteq Y\), let \(1_Z \in \cB(H_Y)\) denote the operator corresponding to the characteristic function of \(Z\).
	Now consider the strongly convergent series $\hat{f} \coloneqq \sum_{\hat{y} \in \hat{Y}_r} 1_{Z_{\hat{y}}} \otimes f(\hat{y}) \in \cB(H_Y \otimes \elltwo)$.
	Then $\| \bar\rho(f) - \hat{f} \| \leq \varepsilon$.
	We can write the commutator of $\tau(T_t) = T_t \otimes \id_{\elltwo}$ and $\id_{H_X} \otimes \hat{f}$ as follows.
	\begin{gather}
		[T_t \otimes \id_{\elltwo}, \id_{H_X} \otimes \hat{f}] =  \sum_{\hat{y}, \hat{z} \in \hat{Y}_r}
			  \left( (\id_{H_X} \otimes 1_{Z_{\hat{z}}}) T_t (\id_{H_X} \otimes 1_{Z_{\hat{y}}}) \right) \otimes (f(\hat{y}) - f(\hat{z})) \nonumber
\\
	= \sum_{\substack{\hat{y}, \hat{z} \in \hat{Y}_r,\\ d(\hat{y}, \hat{z}) \leq \propagation(T_t) + 2 r}}
		  \left( (\id_{H_X} \otimes 1_{Z_{\hat{z}}}) T_t (\id_{H_X} \otimes 1_{Z_{\hat{y}}}) \right) \otimes (f(\hat{y}) - f(\hat{z}))
      \label{eq:rewriteCommutator}
\end{gather}
Let \(v \in H_X \otimes H_Y \otimes \elltwo\) be an arbitrary vector.
In the following we will estimate the norm of the vector \([T_t \otimes \id_{\elltwo}, \id_{H_X} \otimes \hat{f}] v\) for large \(t\).
For each \(\hat{y} \in \hat{Y}_r\), let \(v_{\hat{y}} \coloneqq (\id_{H_X} \otimes 1_{Z_{\hat{y}}} \otimes \id_{\elltwo}) v\).
Then the family \((v_{\hat y})_{\hat y}\) consists of pairwise orthogonal vectors as \((\id_{H_X} \otimes 1_{Z_{\hat y}} \otimes \id_{\elltwo})_{\hat y}\) is a family of pairwise orthogonal projections.
The sum \(\sum_{\hat y} v_{\hat y}\) converges to \(v\) because \(\sum_{\hat y} \id_{H_X} \otimes 1_{Z_{\hat y}}\otimes \id_{\elltwo} \) strongly converges to the identity.
Let $t_0 \geq 1$ such that $\propagation(T_t) < \delta / 2$ for all $t \geq t_0$.
In this case, $\propagation(T_t) + 2 r < \delta / 2 + \delta / 2 = \delta$.
Recall that \(\delta\) is chosen to ensure \(\|f(\hat{y}) - f(\hat{z})\| < \varepsilon\) whenever \(d(\hat{y}, \hat{z}) < \delta\).
We now obtain the following estimate for all \(t \geq t_0\).
\begin{align*}
  \| &[T_t \otimes \id_{\elltwo}, \id_{H_X} \otimes \hat{f}] v \|^2  =\\
  &\underset{\labelcref{eq:rewriteCommutator}}{=}
  \left\| \sum_{\substack{\hat{y}, \hat{z} \in \hat{Y}_r,\\ d(\hat{y}, \hat{z}) \leq \propagation(T_t) + 2 r}}
		  \left( (\id_{H_X} \otimes 1_{Z_{\hat{z}}}) T_t  \otimes (f(\hat{y}) - f(\hat{z}))\right) \ v_{\hat y} \right\|^2 \\
  &= \sum_{\hat z \in \hat{Y}_r} \left\| \sum_{\substack{\hat{y} \in \hat{Y}_r,\\ d(\hat{y}, \hat{z}) \leq \propagation(T_t) + 2 r}}
		  \left( (\id_{H_X} \otimes 1_{Z_{\hat{z}}}) T_t  \otimes (f(\hat{y}) - f(\hat{z}))\right) \ v_{\hat y} \right\|^2\\
 &\leq \sum_{\hat z \in \hat{Y}_r} K_{r, \delta} \sum_{\substack{\hat{y} \in \hat{Y}_r,\\ d(\hat{y}, \hat{z}) < \delta}}
     \left\| \left( (\id_{H_X} \otimes 1_{Z_{\hat{z}}}) T_t  \otimes (f(\hat{y}) - f(\hat{z}))\right) \ v_{\hat y} \right\|^2 \\
&\leq \sum_{\hat z \in \hat{Y}_r} K_{r, \delta} \sum_{\substack{\hat{y} \in \hat{Y}_r,\\ d(\hat{y}, \hat{z}) < \delta}}
    \|T_t\|^2 \underbrace{\|f(\hat{y}) - f(\hat{z})\|^2}_{< \varepsilon^2} \|v_{\hat y}\|^2 \\
&\leq \varepsilon^2 K_{r, \delta} \|T_t\|^2 \underbrace{\sum_{\substack{\hat{z}, \hat{y} \in \hat{Y}_r,\\ d(\hat{z}, \hat{y}) < \delta}}
    \|v_{\hat y}\|^2}_{\leq K_{r, \delta} \|v\|^2}
    \leq \varepsilon^2 K_{r, \delta}^2 \|T\|^2 \|v\|^2.
\end{align*}
In the first and the last inequality above we use that the number of elements in the set \(\hat{Y}_r \cap \Ball_{\delta}(\hat{z})\) is bounded by \(K_{r, \delta}\) by the definition of continuously bounded geometry.
We obtain
\[
	\left\| [ T_t \otimes \id_{\elltwo}, \id_{H_X} \otimes \hat{f} ] \right\|
		\leq \varepsilon K_{r, \delta} \|T\| = \varepsilon K_{r, 4r} \|T\| \leq \varepsilon (K_4+1) \|T\|.
\]
Finally, we conclude that for all \(t \geq t_0\) the estimate
\[
	\left\| [ \tau(T_t), \tilde\rho_Y(f) ] \right\|
		\leq \varepsilon (K_4+1) \|T\| + 2 \|T\| \underbrace{\|\bar\rho(f) - \hat{f} \|}_{\leq \varepsilon} \leq \varepsilon \|T\| (K_4 + 3)
\]
holds.
As $\varepsilon > 0$ was arbitrary, this proves the claim.
\end{proof}

\begin{cor} \label{cor:asymptoticCommutatorHigCom}
Suppose that $Y$ has continuously bounded geometry.
Let $T \in \Loc(\rho_{X \times Y})$ and $f \in \sHigComRed Y$.
Then
\begin{equation}
	t \mapsto \left[ \tau(T_t), \tilde\rho_Y(f) \right] \in \Cz([1,\infty), \Roe(\tilde\rho_X))\,.
\end{equation}
\end{cor}
\begin{proof}
Combine \cref{prop:asymptoticCommuteBndGmtry} and \cref{lem_slant_Roe_cycles_commute}.
\end{proof}

The first step towards the proof of \cref{thm_KHomSlantCompatible} is to reformulate the $E$-theoretic slant product in terms of the localization algebra as follows.
Consider the map
\begin{align}
\Upsilon_{\LSym} \colon \Loc(\rho_{X\times Y}) \otimes \Cz(Y, \Kom) & \to \Loc(\tilde\rho_X) / \Cz([1,\infty),\Roe(\tilde\rho_X)),\label{eq_defn_Ypsilon}\\
(T_t)_{t \in [1,\infty)} \otimes f & \mapsto \left[ (\tau(T_t) \circ \tilde\rho_Y(f))_{t \in [1,\infty)} \right].\notag
\end{align}

One verifies readily that the expression $(\tau(T_t) \circ \tilde\rho_Y(f))_{t \in [1,\infty)}$ defines an element of the localization algebra $\Loc(\tilde\rho_X)$.
\Cref{lem:asymptoticCommute} implies by the by now familiar argument that $\Upsilon_{\LSym}$ is a well-defined $\ast$-homomorphism.

\begin{lem}\label{lem_alternativeKhomslant}
Under the isomorphism $\Delta(\rho_{X\times Y})$ and $\Delta(\tilde\rho_X)$ from \cref{prop_KofLocisEhom}, the $E$-theoretic slant product of \cref{defn_Eslant} agrees with the composition
\begin{align*}
	\K_p(\Loc(X \times Y)) \otimes \K^{q}(Y)
    &\cong \K_p(\Loc(\rho_{X \times Y})) \otimes \K_{-q}(\Cz(Y, \Kom)) \\
  &\xrightarrow{\boxtimes} \K_{p-q}(\Loc(\rho_{X\times Y}) \otimes \Cz(Y, \Kom)) \\
	&\xrightarrow{({\Upsilon_{\LSym}})_\ast} \K_{p-q}( \Loc(\tilde\rho_X) / \Cz([1, \infty), \Roe(\tilde\rho_X))) \\
	&\xrightarrow{\cong} \K_{p-q}(\Loc(\tilde\rho_X)) \cong \K_{p-q}(\Loc(X)),
\end{align*}
where the third map is the inverse of the isomorphism induced on $\K$-theory by the canonical projection $\Loc(\tilde\rho_X) \to \Loc(\tilde\rho_X) / \Cz([1,\infty), \Roe(\tilde\rho_X))$.
\end{lem}

\begin{proof}
First of all we notice that the composition in the statement obviously gives the same map as its non-stabilized counterpart, i.\,e.\ the analogous composition defined using the $*$-homomorphism
\begin{align*}
\Upsilon'_{\LSym} \colon \Loc(\rho_{X\times Y}) \otimes \Cz(Y) & \to \Loc(\rho'_X) / \Cz([1,\infty),\Roe(\rho'_X)),\\
(T_t)_{t \in [1,\infty)} \otimes f & \mapsto \left[ (T_t \circ (\rho'_Y(f)))_{t \in [1,\infty)} \right]
\end{align*}
where $\rho'_X\coloneqq \rho_X\otimes\id_{H_Y}$ and $\rho'_Y\coloneqq\id_{H_X}\otimes\rho_Y(f)$ denote the canonical representations of $\Cz(X)$ and $\Cz(Y)$ on $H_{X\times Y}$.

Now the claim is equivalent to commutativity of the diagram
\[\xymatrix{
\K_p(\Loc(\rho_{X\times Y}))\otimes \K^q(Y)\ar[r]^{\boxtimes}\ar[d]^{\Delta(\rho_{X\times Y})\otimes\id}
&\K_{p-q}(\Loc(\rho_{X\times Y})\otimes\Cz(Y))\ar[d]^{(\Upsilon'_{\LSym})_*}
\\E_p(\Cz(X\times Y),\C)\otimes E_{-q}(\C,\Cz(Y))\ar[d]^{/}
&\K_{p-q}(\Loc(\rho'_X) / \Cz([1,\infty),\Roe(\rho'_X))) \ar[dl]_{\Delta(\rho'_X)}^{\cong}
\\E_{p-q}(\Cz(X),\C)
&\K_{p-q}(\Loc(\rho'_X))\ar[l]_{\Delta(\rho'_X)}^{\cong}\ar[u]^{\cong}
}\]
in which the lower right triangle commutes tautologically.

The vertical left arrows take $x\otimes\theta$ to
\[\delta(\rho_{X\times Y})\circ (x\boxtimes\id_{\Cz(X\times Y)})\circ(\id_{\Cz(X)}\boxtimes\theta)=\delta(\rho_{X\times Y})\circ (x\boxtimes\id_{\Cz(X)}\boxtimes\theta)\]
whereas the composition of the right three arrows of the pentagon map it to
\[\delta(\rho'_X)\circ(((\Upsilon'_{\LSym})_*\circ(x\boxtimes\theta))\boxtimes\id_{\Cz(X)})=\delta(\rho'_X)\circ((\Upsilon'_{\LSym})_*\boxtimes\id_{\Cz(X)})\circ(x\boxtimes\theta\boxtimes\id_{\Cz(X)})\,.\]
Hence it suffices to show that the two asymptotic morphisms $\delta(\rho_{X\times Y})$ and $\delta(\rho'_X)\circ(\Upsilon'_{\LSym}\otimes\id_{\Cz(X)})$ agree up to precomposing with the homomorphism exchanging the two tensor factors $\Cz(X)$ and $\Cz(Y)$ in the domain.
But this holds due to the equation
\[T_t\circ\rho_{X\times Y}(f\otimes g)=T_t\circ \rho'_Y(g)\circ\rho'_X(f)\,.\qedhere\]
\end{proof}

\begin{proof}[Proof of \cref{thm_KHomSlantCompatible}]
By combining \cref{lem_slant_Loc_technical} and \cref{cor:asymptoticCommutatorHigCom} we get a $\ast$-homomorphism
\begin{align*}
\bar{\Upsilon}_{\LSym} \colon \Loc(\rho_{X \times Y}) \tensmax \sHigComRed Y & \to \FiproLoc(\tilde\rho_X) / \Cz([1,\infty),\Roe(\tilde\rho_X)),\\
(T_t)_{t \in [1,\infty)} \otimes f & \mapsto \left[ (\tau(T_t) \circ \tilde\rho_Y(f))_{t \in [1,\infty)} \right].
\end{align*}
Its restriction to \(\Loc(\rho_{X \times Y}) \tensmax \Cz(Y, \Kom)\) is \(\Upsilon_{\LSym}\), see \cref{eq_defn_Ypsilon}.
Therefore we have the commutative diagram
\[
	\resizebox{\textwidth}{!}{
	\xymatrix{
 	0 \ar[r] & \Loc(\rho_{X \times Y}) \tens \Cz(Y, \Kom) \ar[r] \ar[d]^-{\Upsilon_{\LSym}} & \Loc(\rho_{X \times Y}) \tensmax \sHigComRed(Y) \ar[d]^-{\bar{\Upsilon}_\LSym} \ar[r] & \Loc(\rho_{X \times Y}) \tensmax \sHigCorRed(Y) \ar[d]^-{\Psi_{\LSym}} \ar[r] & 0 \\
	0 \ar[r] & \sfrac{\Loc(\tilde\rho_X)}{\Cz([1, \infty), \Roe(\tilde\rho_X))} \ar[r] & \sfrac{\FiproLoc(\tilde\rho_X)}{\Cz([1,\infty), \Roe(\tilde\rho_X))} \ar[r] & \sfrac{\FiproLoc(\tilde\rho_X)}{\Loc(\tilde\rho_X)} \ar[r] & 0 \\
	0 \ar[r] & \Loc(\tilde\rho_X) \ar[r] \ar@{->>}[u] & \FiproLoc(\tilde\rho_X) \ar[r] \ar@{->>}[u] & \sfrac{\FiproLoc(\tilde\rho_X)}{\Loc(\tilde\rho_X)} \ar@{=}[u] \ar[r] & 0,
	}}
	\]
	where the arrows directed upwards are the canonical projections.
	These arrows induce isomorphisms on $\K$-theory because $\Cz([1, \infty), \Roe(\tilde\rho_X))$ is contractible.
	It induces a diagram in $\K$-theory
\[\xymatrix{
\K_p(\Loc(\rho_{X \times Y})) \otimes \K_{1-q}( \sHigCorRed(Y))\ar[r]^-{\id\otimes\mu^*}\ar[d]^{\boxtimes}
&\K_p(\Loc(\rho_{X \times Y})) \otimes \K_{-q}( \Cz(Y, \Kom))\ar[d]^{\boxtimes}
\\\K_{p+1-q}(\Loc(\rho_{X \times Y}) \tensmax \sHigCorRed(Y))\ar[r]^-{\partial}\ar[d]^{(\Psi_{\LSym})_*}
&\K_{p-q}(\Loc(\rho_{X \times Y}) \tens \Cz(Y, \Kom))\ar[d]^{(\Upsilon_{\LSym})_*}
\\\K_{p+1-q}( \sfrac{\FiproLoc(\tilde\rho_X)}{\Loc(\tilde\rho_X)})\ar[r]^-{\partial}
&\K_{p-q}(\Loc(\tilde\rho_X))
}\]
in which the right hand vertical composition is the $E$-theoretic slant product up to the isomorphisms $\Delta$ by \cref{lem_alternativeKhomslant} and the composition of the left vertical arrows with the bottom horizontal arrow is $(-1)^p$ times the slant product from \cref{defn_localizationslant}. The lower square of the diagram commutes and the upper square commutes only up to a sign $(-1)^p$ (see \cref{rem_signconvention}).
\end{proof}

\subsection{Composing slant with external products}\label{sec_composition_cross_prods}

In analogy to the first part of \cref{lem_crossslantcompatibilityinEtheory} we shall now prove the following theorem, which says that the external and slant products which we have constructed are compatible in the sense that $(x\times z)/\theta=x\times(z/\theta)$.
It would also be nice to have an analogue of the second part of \cref{lem_crossslantcompatibilityinEtheory}, but this would require the construction of a secondary\footnote{The external product has to be a secondary one, i.\,e.\ one with a degree shift, to make the degrees work out: if $x$, $\eta$, $\theta$ have degrees $p$, $1-q$, $1-r$, respectively, then $(x/\theta)/\eta$ has degree $p-q-r$, and in order to let $x/(\theta\times\eta)$ have the same degree, the external product $\theta\times\eta$ must have degree $1-q-r$ and not $(1-q)+(1-r)$.} external product of the form
\[\K_{1-q}(\sHigCorRed Y)\otimes \K_{1-r}(\sHigCorRed Z)\to \K_{1-q-r}(\sHigCorRed (Y\times Z))\]
and it is completely unclear how such an external product could be constructed.

\begin{thm}\label{thm_slantcrosscomp}
The compositions
\[
\xymatrix@R-2pc{
\K_m(X)\ar[r]^-{\times z}&\K_{m+p}(X\times Y\times Z)\ar[r]^-{/\theta}&\K_{m+p-q}(X\times Y)
\\\Strg_m(X)\ar[r]^-{\times z}&\Strg_{m+p}(X\times Y\times Z)\ar[r]^-{/\theta}&\Strg_{m+p-q}(X\times Y)
\\\K_m(\Roe X)\ar[r]^-{\times z}&\K_{m+p}(\Roe(X\times Y\times Z))\ar[r]^-{/\theta}&\K_{m+p-q}(\Roe (X\times Y))
}\]
are equal to the external product with the apropriate slant product $z/\theta$ for all $m,p,q\in \Z$ and all $z,\theta$ as follows:
\begin{itemize}
\item In the first two compositions $z\in \K_p(Y\times Z)$ and in the third one
$z\in \K_p(\Roe (Y\times Z))$. 
\item In the first composition either $\theta\in \K^q(Z)$ or $\theta\in \K_{1-q}(\sHigCorRed Z)$ and in the other two $\theta\in \K_{1-q}(\sHigCorRed Z)$.
\end{itemize}
\end{thm}

\begin{proof}
For the first of these three compositions and $\theta\in\K^q(Z)$ this is exactly the first part of \Cref{lem_crossslantcompatibilityinEtheory}. So it remains to show the cases where $\theta\in \K_{1-q}(\sHigCorRed Z)$ for the slant products we have constructed in \cref{sec_slantconstruction}. We shall only consider the second composition, because everything else goes through completely analogously.

Let $\rho_X$, $\rho_Y$, $\rho_Z$ be representations of $\Cz(X)$, $\Cz(Y)$, $\Cz(Z)$ on $H_X$, $H_Y$, $H_Z$, respectively. Denote by $\rho_{Y\times Z}$, $\rho_{X\times Y\times Z}$ the tensor product representations of $\Cz(Y\times Z)$, $\Cz(X\times Y\times Z)$ on $H_Y\otimes H_Z$, $H_X\otimes H_Y\otimes H_Z$, respectively. Furthermore, we define representations $\tilde{\rho}_{X\times Y}$ of $\Cz(X\times Y)$ and $\tilde{\rho}_Z$ of $\Mult(\Cz(Z,\Kom))$ on $H_X\otimes H_Y\otimes H_Z\otimes\ell^2$ in complete analogy to the definition of $\tilde\rho_X$ and $\tilde\rho_Y$ previously in this section and similarily we define representations $\hat{\rho}_{Y}$ of $\Cz(Y)$ and $\hat{\rho}_Z$ of $\Mult(\Cz(Z,\Kom))$ on $H_Y\otimes H_Z\otimes\ell^2$.

Note that taking tensor products of operators pointwise in time gives rise to $*$-homomorphisms
\begin{align*}
\Locz(\rho_X)\tensmax\FiproLoc(\hat\rho_Y)&\to \FiproLocz(\tilde\rho_{X\times Y})
\\\Locz(\rho_X)\tensmax\Loc(\hat\rho_Y)&\to \Locz(\tilde\rho_{X\times Y})
\end{align*}
with the second one being the restriction of the first one, and hence also a quotient $*$-homomorphism
\[\Locz(\rho_X)\tensmax\frac{\FiproLoc(\hat\rho_Y)}{\Loc(\hat\rho_Y)}\to \frac{\FiproLocz(\tilde\rho_{X\times Y})}{\Locz(\tilde\rho_{X\times Y})}\]
and these fit into the following diagram.
\begin{align*}
\resizebox{\textwidth}{!}{
\xymatrix{
{\begin{matrix}\K_m(\Locz(\rho_X))\otimes\\ \K_p(\Loc(\rho_{Y\times Z}))\\\otimes \K_{1-q}(\sHigCorRed Z)\end{matrix}}
\ar[r]^-{\exttensprod\tens\id}\ar[d]^-{\id\tens\exttensprod}
&{\begin{matrix}\K_{m+p}{\begin{pmatrix}\Locz(\rho_X)\\\tensmax \Loc(\rho_{Y\times Z})\end{pmatrix}}\\\otimes \K_{1-q}(\sHigCorRed Z)\end{matrix}}
\ar[r]\ar[d]^-{\exttensprod}
&{\begin{matrix}\K_{m+p}(\Locz(\rho_{X\times Y\times Z}))\\\otimes \K_{1-q}(\sHigCorRed Z)\end{matrix}}
\ar[d]^-{\exttensprod}
\\{\begin{matrix}\K_m(\Locz(\rho_X))\otimes \\\K_{p+1-q}{\begin{pmatrix}\Loc(\rho_{Y\times Z})\\\tensmax \sHigCorRed Z\end{pmatrix}}\end{matrix}}
\ar[r]^-{\exttensprod}\ar[d]^-{\id\otimes(\Psi_{\LSym})_*}
&\K_{m+p+1-q}{\begin{pmatrix}\Locz(\rho_X)\\\tensmax \Loc(\rho_{Y\times Z})\\\tensmax \sHigCorRed Z\end{pmatrix}}
\ar[r]\ar[d]^-{(\id\otimes\Psi_{\LSym})_*}
&\K_{m+p+1-q}{\begin{pmatrix}\Locz(\rho_{X\times Y\times Z})\\\tensmax \sHigCorRed Z\end{pmatrix}}
\ar[d]^{(\Psi_{\LSym,0})_*}
\\{\begin{matrix}\K_m(\Locz(\rho_X))\otimes \\\K_{p+1-q}\left(\frac{\FiproLoc(\hat\rho_Y)}{\Loc(\hat\rho_Y)}\right)\end{matrix}}\ar[r]^-{\exttensprod}
\ar[d]^{\id\otimes\partial}
&\K_{m+p+1-q}{\begin{pmatrix}\Locz(\rho_X)\\\tensmax \frac{\FiproLoc(\hat\rho_Y)}{\Loc(\hat\rho_Y)}\end{pmatrix}}
\ar[r]\ar[d]^{\partial}
&\K_{m+p+1-q}\left(\frac{\FiproLocz(\tilde\rho_{X\times Y})}{\Locz(\tilde\rho_{X\times Y})}\right)\ar[d]^-{\partial}
\\{\begin{matrix}\K_m(\Locz(\rho_X))\otimes \\\K_{p-q}(\Loc(\hat\rho_Y))\end{matrix}}
\ar[r]^-{\exttensprod}
&\K_{m+p-q}{\begin{pmatrix}\Locz(\rho_X)\\\otimes\Loc(\hat\rho_Y)\end{pmatrix}}
\ar[r]
&\K_{m+p-q}(\Locz(\tilde\rho_{X\times Y}))
}
}
\end{align*}
All squares in this diagram commute, except the lower left one, which only commutes up to a sign $(-1)^m$ due to our sign convention in \cref{rem_signconvention}.

The top and bottom rows are the external products for the localization algebras.
The left column is $(-1)^p$ times the slant product while the right column is $(-1)^{m+p}$ times the slant product. The signs match up perfectly, thus proving the claim.
\end{proof}

Spezializing to the case where $Y$ is a one-point space will give us Property~\ref{properties_slant}~\ref{item_2_mainthm_slant} as a corollary, but first we have to introduce the pairings.
\begin{defn}
The pairing
\[\langle-,-\rangle\colon \K_p(\Roe Y) \otimes \K_{1-q}(\sHigCorRed Y)\to \K_{p-q}(\Roe\{*\})\cong \begin{cases}\Z&p-q\text{ even}\\0&p-q\text{ odd}\end{cases}\]
is defined as the special case of the slant products for the space $X,Y$ where $X=\{*\}$ is a single point.
The same construction applied to the localization algebra instead of the Roe algebra also yields a pairing
\[\langle-,-\rangle\colon \K_p(Y) \otimes \K_{1-q}(\sHigCorRed Y)\to \K_{p-q}(\{*\})\cong \begin{cases}\Z&p-q\text{ even}\\0&p-q\text{ odd}\,.\end{cases}\qedhere\]
\end{defn}

\begin{rem}\
  
\begin{enumerate}
\item It is easy to see that the first of these pairings agrees with the one defined in \cite[Section 6]{EmeMey}, but possibly only up to a sign $(-1)^p$. This is because Emerson and Meyer have not specified in detail which sign conventions they use (cf.\ the remarks about signs at the beginning of \cref{sec_recovering_usual_slant_prod}).
\item Compatibility of the slant products with the assembly maps $\mu_Y$ and $\mu_{\{*\}}$ of $Y$ and $\{*\}$ (see \cref{thm_commutativity_diagram}) shows that the second pairing is a special case of the first one: We have $\langle z,\theta\rangle=\langle \mu_Y(z),\theta\rangle$ for all $z\in \K_p(Y)$ and $\theta\in \K_{1-q}(\sHigCorRed Y)$.
\item Recall from the beginning of \cref{sec_recovering_usual_slant_prod} that the pairing of $\K$-homology with $\K$-theory is also only a special case of the slant product between them. Therefore, if $Y$ has continuously bounded geometry, then \cref{thm_KHomSlantCompatible} also implies $\langle z,\theta\rangle=\langle z,\mu_Y^*(\theta)\rangle$ for all $z\in \K_p(Y)$ and $\theta\in \K_{1-q}(\sHigCorRed Y)$.
\item Note that in the case $X=\{*\}$, the $\K$-theory of $\Roe(\rho_X)$ is isomorphic to the $\K$-theory of $\C$ for any non-degenerate faithful representation, because $\Roe(\rho_X)\cong\Kom(H_X)$ with $H_X\not=0$.
Hence we don't have to choose an ample representation for the construction of the pairings, but we may take the non-ample representation of $\C$ on $H_X=\C$ by multiplication instead, simplifying the formulas significantly.\qedhere
\end{enumerate}
\end{rem}

\begin{cor}\label{cor_slantcrosscomp}
The compositions
\[
\xymatrix@R-2pc{
\K_m(X)\ar[r]^-{\times z}&\K_{m+p}(X\times Y)\ar[r]^-{/\theta}&\K_{m+p-q}(X)
\\\Strg_m(X)\ar[r]^-{\times z}&\Strg_{m+p}(X\times Y)\ar[r]^-{/\theta}&\Strg_{m+p-q}(X)
\\\K_m(\Roe X)\ar[r]^-{\times z}&\K_{m+p}(\Roe(X\times Y))\ar[r]^-{/\theta}&\K_{m+p-q}(\Roe X)
}\]
are equal to the multiplication with $\langle z,\theta\rangle$, which is either an integer, if $p-q$ is even, or zero by construction, if $p-q$ odd, for all $m,p,q\in \Z$ and all $z,\theta$ as follows:
\begin{itemize}
\item In the first two compositions $z\in \K_p(Y)$ and in the third one 
$z\in \K_p(\Roe Y)$. 
\item In the first composition either $\theta\in \K^q(Y)$ or $\theta\in \K_{1-q}(\sHigCorRed Y)$ and in the other two $\theta\in \K_{1-q}(\sHigCorRed Y)$.
\end{itemize}
\end{cor}

Now we also get Property~\ref{properties_slant}~\ref{item_3_mainthm_slant} as a corollary.
\begin{cor}
Denote by $\beta\in\K_{1-n}(\sHigCorRed\R^n)$ the Bott element of the Euclidean space, i.\,e.\ $\mu^*(\beta)\in\K^n(\R^n)\cong\Z$ is the generator which pairs to one with  the fundamental class $[\Dirac_{\R^n}] \in \K_n(\IR^n)$ of Euclidean space: $\langle [\Dirac_{\R^n}],\beta\rangle=1$. Then the slant products
\begin{align*}
\K_{p+n}(X\times\R^n)&\xrightarrow{\blank/\beta=\blank/\mu^*(\beta)}  \K_p(X)
\\\Strg_{p+n}(X\times\R^n)&\xrightarrow[\hphantom{\blank/\beta=\blank/\mu^*(\beta)}]{\blank/\beta}\Strg_p(X)
\\\K_{p+n}(\Roe(X\times\R^n))&\xrightarrow[\hphantom{\blank/\beta=\blank/\mu^*(\beta)}]{\blank/\beta}\K_p(\Roe X)
\end{align*}
coincide with the $n$-fold Mayer--Vietories boundary maps.
\end{cor}
\begin{proof}
It was shown 
 in \cite[Theorem~5.5]{zeidler_secondary} that the external products
\begin{align*}
\K_p(X)&\xrightarrow{\times [\Dirac_{\R^n}]} \K_{p+n}(X \times \R^n)&\Strg_p(X)\,,&\xrightarrow{\times [\Dirac_{\R^n}]} \Strg_{p+n}(X \times \R^n)\,,
\end{align*}
which are also called suspension maps, are isomorphisms and that their inverses are given by the $n$-fold Mayer--Vietoris boundary. Essentially the same proof also shows that the external product
\[\K_p(\Roe X) \xrightarrow{\times \Ind [\Dirac_{\R^n}]} \K_{p+n}(\Roe(X \times \R^n))\]
is also an isomorphism whose inverse is given by the $n$-fold coarse Mayer--Vietoris boundary map.
Hence, it suffices to show that the slant products with $\beta$
are also inverses to these suspension maps.

By the preceding corollary, the composition of the external product followed by the slant product is in each of these three cases equal to
 multiplication by $\langle [\Dirac_{\R^n}],\beta\rangle=1$. We conclude that the slant product by $\beta$ is the left inverse to the suspension map. Because the suspension map is an isomorphism, this suffices to conclude the claim.
\end{proof}

\subsection{Naturality of the slant products}\label{sec_functoriality}

Naturality of our slant products is of course a property that has to be expected.
We prove it in this section and use it directly in the next to coarsify all of our previous statements.

Throughout this section let $X,X',Y,Y'$ denote proper metric spaces with $Y$ and $Y'$ having bounded geometry. We consider the Hilbert spaces $H_X$, $H_Y$, $H_{X\times Y}$, $\tilde H_X$ and the ample representations $\rho_X$, $\rho_Y$, $\rho_{X\times Y}$, $\tilde\rho_X$, $\bar\rho_Y$, $\tilde\rho_Y$ as before and let the Hilbert spaces $H_{X'}$, $H_{Y'}$, $H_{X'\times Y'}$, $\tilde H_{X'}$ and the ample representations $\rho_{X'}$, $\rho_{Y'}$, $\rho_{X'\times Y'}$, $\tilde\rho_{X'}$, $\bar\rho_{Y'}$, $\tilde\rho_{Y'}$ be chosen and constructed in complete analogy.

\begin{defn}[{\cite[Definition 6.3.9]{higson_roe}}]\label{defn:coarseCoveringIsometry}
Let $\alpha\colon X\to X'$ be a coarse map. An isometry $V\colon H_X\to H_{X'}$ \emph{covers} a coarse map $\alpha\colon X\to X'$, if
\[\{(x',\alpha(x))\in X'\times X'\mid (x',x)\in\supp(V)\}\]
is an entourage of $X'$.
\end{defn}

\begin{prop}[{\cite[Section 6.3]{higson_roe}}]
Any coarse map $\alpha\colon X\to X'$ is covered by an isometry $V\colon H_X\to H_{X'}$. Conjugation by $V$ yields a $*$-homomorphism
\[\Ad_V\colon \Roe(\rho_X)\to \Roe(\rho_{X'})\,,\quad T\mapsto VTV^*\]
and the induced map in $\K$-theory $(\Ad_V)_*\colon \K_*(\Roe X)\to \K_*(\Roe X')$ is independent of the choice of covering isometry $V$ and depends functorial on the coarse map $\alpha$. For this reason we will denote it by $\alpha_*$. \qed
\end{prop}

\begin{prop}[{see \cite{EmeMey}}]\label{prop_stablehigsonfunctoriality}
The reduced stable Higson corona---and hence also its $\K$-theory---is contravariantly functorial under coarse maps. If $\beta\colon Y\to Y'$ is a coarse map and an element of $[f]\in\sHigCorRed Y'$ is represented by a function $f\in\sHigComRed Y'$, then the class $\beta^*[f]\in \sHigCorRed Y$ is represented by any function $g\in\sHigComRed Y$ for which the (possibly non-continuous) function $g-f\circ\beta$ converges to zero at infinity.\qed
\end{prop}

\begin{thm}\label{thm_coarseslantfunctorial}
The slant product for the Roe algebra is natural under pairs of coarse maps $\alpha\colon X\to X'$ and $\beta\colon Y\to Y'$ in the sense that
\[\alpha_*(x/\beta^*(\theta))=(\alpha\times \beta)_*(x)/\theta\]
for all $x\in \K_*(\Roe(X\times Y))$ and $\theta\in \K_*(\sHigCorRed Y')$.
\end{thm}

\begin{proof}
Assume that $V\colon H_X\to H_{X'}$ is an isometry covering $\alpha$ and $W \colon H_Y\to H_{Y'}$ is an isometry covering $\beta$. Then $V\otimes W\colon H_{X\times Y}\to H_{X'\times Y'}$ covers $\alpha\times \beta$ and $V\otimes W\otimes\id_{\elltwo}\colon \tilde H_X\to \tilde H_{X'}$ covers $\alpha$. The latter is in complete analogy to \cref{lem_propagtioncomparision}.

We have to show commutativity of the diagram
\[\xymatrix@C=-6em{
&{\K_p(\Roe(\rho_{X\times Y}))\otimes \K_{1-q}(\sHigCorRed Y')}\ar[dl]_{\id\otimes \beta^*}\ar[dr]^-{(\Ad_{V\otimes W})_*\otimes\id}\ar[dd]&
\\\K_p(\Roe(\rho_{X\times Y}))\otimes \K_{1-q}(\sHigCorRed Y)\ar[dd]&&\K_p(\Roe(\rho_{X'\times Y'}))\otimes \K_{1-q}(\sHigCorRed Y')\ar[dd]
\\&{\K_{p+1-q}(\Roe(\rho_{X\times Y})\otimes\sHigCorRed Y')}\ar[dl]_{(\id\otimes \beta^*)_*}\ar[dr]^-{(\Ad_{V\otimes W}\otimes\id)_*}&
\\\K_{p+1-q}(\Roe(\rho_{X\times Y})\otimes \sHigCorRed Y)\ar[d]^{\Psi}&&\K_{p+1-q}(\Roe(\rho_{X'\times Y'})\otimes \sHigCorRed Y')\ar[d]^{\Psi'}
\\\K_{p+1-q}(\Fipro(\tilde\rho_X)/\Roe(\tilde\rho_X))\ar[rr]^{(\Ad_{V\otimes W\otimes\id})_*}\ar[d]&&\K_{p+1-q}(\Fipro(\tilde\rho_{X'})/\Roe(\tilde\rho_{X'}))\ar[d]
\\\K_{p-q}(\Roe(\tilde\rho_X))\ar[rr]^{(\Ad_{V\otimes W\otimes\id})_*}&&\K_{p-q}(\Roe(\tilde\rho_{X'}))
}\]
where $\Psi'$ is defined in complete analogy to $\Psi$ and we have used that the conjugation by $V\otimes W\otimes\id$ clearly also maps $\Fipro(\tilde\rho_X)$ to $\Fipro(\tilde\rho_{X'})$.

The three quadrilaterals clearly commute. It remains to investigate the pentagon and here it is sufficient to show commutativity of the underlying pentagon of $*$-homomorphisms:
\begin{equation}\label{eq_coarseslantfunctorialpentagon}\xymatrix@C=-1em{
&\Roe(\rho_{X\times Y})\otimes\sHigCorRed Y'\ar[dl]_{\id\otimes \beta^*}\ar[dr]^-{\Ad_{V\otimes W}\otimes\id}&
\\\Roe(\rho_{X\times Y})\otimes \sHigCorRed Y\ar[d]^{\Psi_{\LSym}}&&\Roe(\rho_{X'\times Y'})\otimes \sHigCorRed Y'\ar[d]^{\Psi'_{\LSym}}
\\\Fipro(\tilde\rho_X)/\Roe(\tilde\rho_X)\ar[rr]^{\Ad_{V\otimes W\otimes\id}}&&\Fipro(\tilde\rho_{X'})/\Roe(\tilde\rho_{X'})
}\end{equation}
Given $S\in \Roe(\rho_{X\times Y})$ and $f\in\sHigComRed Y'$, the left path along the pentagon maps $S\otimes [f]\in \Roe(\rho_{X\times Y})\otimes\sHigCorRed Y'$ to the element of $\Fipro(\tilde\rho_{X'})/\Roe(\tilde\rho_{X'})$ represented by the operator
\[T_1\coloneqq (V\otimes W\otimes\id)\circ (S\otimes\id)\circ (V^*\otimes (\bar\rho_Y(g)\circ( W^*\otimes\id)))\]
with $g\in\sHigComRed Y$ as in \cref{prop_stablehigsonfunctoriality}.
The right path, on the other hand, maps it to the element represented by
\[T_2\coloneqq (V\otimes W\otimes\id)\circ (S\otimes\id)\circ (V^*\otimes ((W^*\otimes\id)\circ\bar\rho_Y(f)))\,.\]
It has to be shown that the difference of these two operators lies in $\Roe(\tilde\rho_{X'})$.

We already know that choosing different representatives $f$ and $g$ only changes these $T_1$ and $T_2$ by elements of $\Roe(\tilde\rho_{X'})$. Hence it suffices to show that the difference $T_1-T_2$ can be made arbitrarily small by choosing suitable representatives $f$ and $g$. This amounts to showing that the difference
\[(W\otimes\id)\circ\bar\rho_Y(g)-\bar\rho_{Y'}(f)\circ(W\otimes\id)\]
can be made arbitrarily small, which will be done in the remaining part of the proof.

Exploiting the bounded geometry of $Y$ we choose $\hat Y\subset Y$, $r>0$ and $K_R>0$ for each $R>0$ exactly as in \cref{defn:boundedGeometry}.\ref{defn_bounded_geom}, that is, such that $Y=\bigcup_{\hat y\in \hat Y}\Ball_r(\hat y)$ and for each $R>0$ the number of elements $\#(\hat Y\cap\clBall_R(y))$ is bounded by $K_R$ uniformly in $y$. Furthermore, we choose $\hat Y'\subset Y'$ and $r'>0$ such that the family of balls $\{\Ball_{r'}(\hat y')\}_{\hat y'\in\hat Y'}$ is a locally finite open cover of $Y'$. This can be done due to the bounded geometry of $Y'$, although we don't need the full strength of bounded geometry of $Y'$ at this point.

Similar to what we have done in the proof of \cref{lem_slant_Roe_cycles_commute} we choose decompositions of $Y$ and $Y'$ into families  $\{Z_{\hat{y}}\}_{\hat{y} \in \hat{Y}}$ and  $\{Z'_{\hat{y}'}\}_{\hat{y}' \in \hat{Y}'}$ of pairwise disjoint Borel subsets which are subordinate to the open covers  $\{\Ball_r(\hat y)\}_{\hat{y} \in \hat{Y}}$ and  $\{\Ball_{r'}(\hat y')\}_{\hat{y}' \in \hat{Y}'}$. Given Borel subsets $Z\subset Y$ or $Z'\subset Y'$, we again denote by $1_Z \in \Lin(H_Y)$ and $1_{Z'} \in \Lin(H_{Y'})$ the projections corresponding to the characteristic functions of $Z$ and $Z'$ under the canonical extensions of $\rho_Y$ and $\rho_{Y'}$ to Borel functions, respectively.

Let $s_1>0$ be such that $d(\beta(y),\beta(z))\leq s_1$ for all $y,z\in Y$ with $d(y,z)\leq r$ and define the following number, which is finite, because $W$ covers $\beta$:
\[s_2\coloneqq\sup\{d(y',\beta(y))\mid (y',y)\in\supp(W)\}\,.\]
If $\hat y\in\hat Y$ and $\hat y'\in\hat Y'$ are such that $d(\hat y',\beta(\hat y))$ is bigger than
\[s\coloneqq s_1+s_2+r'\,,\]
then the the distance between $Z'_{\hat y'}$ and $\beta(Z_{\hat y})$ is bigger than $s_2$ and therefore
$1_{Z'_{\hat y'}}\circ W\circ 1_{Z_{\hat y}}=0$.

Finally, let $\varepsilon>0$. We may assume that our element of $\sHigCorRed Y'$ is represented by a function $f\in\sHigComRed Y'$ whose $s$-variation $\Var_s(f)$ is bounded by $\varepsilon$ everywhere. A priori this is only the case outside of a compact subset, but multiplying our function with a slowly varying function $X\to[0,1]$ which is $0$ on a large compact set and $1$ outside of an even larger compact yields a new representative with the demanded property.
We can furthermore use a partition of unity $\{\varphi_{\hat y}\}_{\hat{y} \in \hat{Y}}$ subordinate to the open cover $\{\Ball_r(\hat y)\}_{\hat{y} \in \hat{Y}}$ to define the function
\[g\coloneqq\sum_{\hat y\in\hat Y}f(\beta(\hat y))\varphi_{\hat y}\in\sHigComRed Y \,.\]
Clearly, $g-f\circ \beta$ converges to zero at infinity and $g$ is therefore a valid representative of $\beta^*[f]$.

Now define the strongly convergent series
\begin{align*}
\hat f&\coloneqq\sum_{\hat y'\in\hat Y'}1_{Z'_{\hat y'}}\otimes f(\hat y')\in\Lin(H_{Y'}\otimes\elltwo)\,,
\\\hat g&\coloneqq\sum_{\hat y\in\hat Y}1_{Z_{\hat y}}\otimes f(\beta(\hat y))\in\Lin(H_{Y}\otimes\elltwo)\,.
\end{align*}
It is clear from the the bound on the $s$-variation of $f$ that $\|\hat f-\bar\rho_{Y'}(f)\|\leq\varepsilon$ and $\|\hat g-\bar\rho_{Y}(g)\|\leq\varepsilon$.
Our goal is to estimate the norm of the operator
\begin{align*}
\hat T&\coloneqq(W\otimes\id)\circ\hat g-\hat f\circ(W\otimes\id)
\\&=\sum_{\substack{\hat y\in\hat Y\\\hat y'\in\hat Y'}} \left(1_{Z'_{\hat y'}}\circ W\circ 1_{Z_{\hat y}}\right) \otimes(f(\beta(\hat y))-f(\hat y'))
\\&=\sum_{\substack{\hat y\in\hat Y\,,\,\hat y'\in\hat Y'\\d(\hat y',\beta(\hat y))\leq s}} \left(1_{Z'_{\hat y'}}\circ W\circ 1_{Z_{\hat y}}\right) \otimes(f(\beta(\hat y))-f(\hat y'))\,.
\end{align*}
Again as in the proof of \cref{lem_slant_Roe_cycles_commute} we decompose an arbitrary $v\in\tilde H_X$ into the vectors $v_{\hat y}\coloneqq(\id_{H_X}\otimes 1_{Z_{\hat y}}\otimes\id_{\elltwo})v$ with $\hat y\in \hat Y$ and calculate
\begin{align*}
\|\hat Tv\|^2&=\sum_{\hat y'\in\hat Y'}\left\|\sum_{\substack{\hat y\in\hat Y\\d(\hat y',\beta(\hat y))\leq s}} \left(1_{Z'_{\hat y'}}\circ W\circ 1_{Z_{\hat y}}\right) \otimes(f(\beta(\hat y))-f(\hat y'))v\right\|^2
\\&\leq \sum_{\hat y'\in\hat Y'}\left(K_s\cdot \varepsilon\cdot \|v_{\hat y}\|\right)^2=K_s^2\cdot \varepsilon^2\cdot \|v\|^2\,.
\end{align*}
Hence we have the norm estimate
\[\left\|(W\otimes\id)\circ\bar\rho_Y(g)-\bar\rho_{Y'}(f)\circ(W\otimes\id)\right\|\leq 2\varepsilon+\|\hat T\|\leq (2+K_s)\varepsilon\,,\]
proving the claim.
\end{proof}

The proof of \cref{thm_coarseslantfunctorial} can easily be adapted to yield analogous statements for the localization algebras. In order to do so, we first have to recall the functoriality of their $\K$-theories.

\begin{defn}\label{def:equivcoveringiso}
Suppose that $\alpha\colon X\to X'$ is a uniformly continuous coarse map between proper metric spaces and $\rho_X\colon \Cz(X)\to\Lin(H_X)$, $\rho_{X'}\colon \Cz(X')\to \Lin(H_{X'})$ representations on Hilbert spaces. A uniformly continuous family of isometries $V\colon [1,\infty)\to\Lin(H_X,H_{X'}), t\mapsto V_t$ is said to \emph{cover} $\alpha$ if the number
  \[\omega_{\alpha, V}(t) \coloneqq \sup\{d(y, \alpha(x))\mid (y,x) \in\supp(V_t) \}\]
  is finite for all \(t \geq 1\) and satisfies \(\omega_{\alpha, V}(t) \to 0\) as \(t \to \infty\).
\end{defn}

\begin{prop}[{compare \cite[Proposition~3.2]{qiao_roe}, \cite[Theorem 6.6.3]{WillettYuHigherIndexTheory}}]\label{prop_LocFunctoriality}
If  $\rho_X\colon \Cz(X)\to\Lin(H_X)$ is a representation and  $\rho_{X'}\colon \Cz(X')\to \Lin(H_{X'})$ an ample one, then any uniformly continuous coarse map $\alpha\colon X\to X'$ is covered by a uniformly continuous family of isometries $V$. Conjugation by $V$ yields a $*$-homomorphism
\[\Ad_{V}\colon \Loc(\rho_{X})\to\Loc(\rho_{X'})\,,\quad L\mapsto [t\mapsto V_t\circ L_t\circ V_t^*]\]
which gives rise to a commutative diagram
\[\xymatrix{
0\ar[r] & \Locz(\rho_{X})\ar[r]\ar[d]^{\Ad_V} & \Loc(\rho_{X})\ar[r]\ar[d]^{\Ad_V}  & \Roe(\rho_{X})\ar[r]\ar[d]^{\Ad_{V_1}} &0
\\0\ar[r] & \Locz(\rho_{X'})\ar[r] & \Loc(\rho_{X'})\ar[r] & \Roe(\rho_{X'})\ar[r] &0  \,.
}\]
The vertical maps in the induced commutative diagram in $\K$-theory
\[\mathclap{
\xymatrix{
\K_{\ast+1}(\Roe X) \ar[r]^-{\partial} \ar[d]^-{(\Ad_{V_1})_*} & \Strg_\ast(X) \ar[r] \ar[d]^-{(\Ad_V)_*} & \K_\ast(X) \ar[r]^-{\Ind} \ar[d]^-{(\Ad_V)_*} & \K_\ast(\Roe X) \ar[d]^-{(\Ad_{V_1})_*}\\
\K_{\ast+1}(\Roe X') \ar[r]^-{\partial} & \Strg_{\ast}(X') \ar[r] & \K_{\ast}(X') \ar[r]^-{\Ind} & \K_{\ast}(\Roe X')
}}\]
are independent of the choice of the covering isometry and depend functorial on $\alpha$. For this reason they will all be denoted by $\alpha_*$.
\end{prop}

Note that this proposition implies the independence of the $\K$-theory of $\Loc X$, $\Loc X$, $\Roe X$ from the chosen ample module which we have mentioned earlier. Furthermore, it includes special cases of the usual functoriality of the $\K$-homology of spaces under proper continuous maps and of the $\K$-theory of the Roe algebra under coarse maps.

\begin{thm}\label{thm_LocSlantNatural}
The slant products for the localization algebras are natural under pairs of uniformly continuous coarse maps $\alpha\colon X\to X'$ and $\beta\colon Y\to Y'$ in the sense that
\begin{equation}\label{eq_LocSlantNatural}
\alpha_*(x/\beta^*(\theta))=(\alpha\times \beta)_*(x)/\theta
\end{equation}
for all $x\in \K_*(X\times Y)$ or $x\in \Strg_*(X\times Y)$ and $\theta\in \K_*(\sHigCorRed Y')$.
\end{thm}

\begin{proof}
The proof is completely analogous to the one of \cref{thm_coarseslantfunctorial}. The main step is to show commutativity of the analogue versions of  \eqref{eq_coarseslantfunctorialpentagon} for $\Loc$ and $\FiproLoc$ as well as for $\Locz$ and $\FiproLocz$.
To this end, one has to show that the difference of two representatives in $\FiproLoc(\tilde\rho_{X'})$ or $\FiproLocz(\tilde\rho_{X'})$ differ by an element of $\Loc(\tilde\rho_{X'})$ or $\Locz(\tilde\rho_{X'})$, respectively. But this is exactly the commutativity of \eqref{eq_coarseslantfunctorialpentagon} applied pointwise to the family of operators at each time in $[1,\infty)$.
\end{proof}

\begin{rem}\label{rem_better_functoriality_structgroup}
We restricted ourselves in \cref{prop_LocFunctoriality} and \cref{thm_LocSlantNatural} to functoriality under uniformly continuous coarse maps only, because this is the functoriality which can be described easily in terms of the localization algebras.
But both statements of \cref{thm_LocSlantNatural} can be generalized further:
\begin{itemize}
\item As we have mentioned above, the functoriality of \cref{prop_LocFunctoriality} is a special case of the functoriality of the $\K$-homology groups under proper continuous maps under the isomorphisms $\K_*(-)\cong\K_*(\Loc(-))$. Hence the results of \cref{sec_recovering_usual_slant_prod} together with the naturality of the coassembly map under continuous coarse maps and the well-known naturality of the topological slant product between $\K$-homology and $\K$-theory immediately implies that Formula \eqref{eq_LocSlantNatural} even holds for all $x\in \K_*(X\times Y)$, $\theta\in \K_*(\sHigCorRed Y')$, proper continuous maps $\alpha\colon X\to X'$ and all continuous coarse maps $\beta\colon Y\to Y'$, if $Y$ and $Y'$ have continuously bounded geometry.

\item By using a different picture of the structure group $\Strg_*(-)=\K_*(\Locz(-))$ one can show that the functoriality extends to continuous coarse maps \cite[Section 12.4]{higson_roe}, \cite[Chapter~6]{WillettYuHigherIndexTheory}. In this case we can show that Formula \eqref{eq_LocSlantNatural} even holds for all $x\in \Strg_*(X\times Y)$, $\theta\in \K_*(\sHigCorRed Y')$ and continuous coarse maps $\alpha\colon X\to X'$ and $\beta\colon Y\to Y'$ by applying the following trick:

Given a continuous coarse map $\alpha\colon X\to X'$ we define a new proper metric space $X''$ as the set $X$ equipped with the metric $d''\coloneqq\max\{d,\alpha^*d'\}$, that is,
\[d''(x,y)=\max\{d(x,y),d'(\alpha(x),\alpha(y))\}\,,\]
where $d$ and $d'$ are the metrics on $X$ and $X'$, respectively.
Consider the commutative diagram
\begin{equation*}
\xymatrix{X\ar[rr]^\alpha &&X'\\&X''\ar[ul]^{\alpha'}\ar[ur]_{\alpha''}&}
\end{equation*}
where $\alpha'$ is the identity map and $\alpha''$ is the same as $\alpha$ on the underlying sets. It follows from $d''\geq d$ and $d''\geq (\alpha'')^*d'$ that both $\alpha'$ and $\alpha''$ are uniformly continuous coarse maps. Furthermore, the inverse $(\alpha')^{-1}$ is continuous, because $\alpha$ is continuous, and it is a coarse map, because $\alpha$ is a coarse map. Hence, $\alpha'$ is both a homeomorphism and a coarse equivalence and therefore induces an isomorphism between the structure groups.
Similarily, the continuous coarse map $\beta$ can be decomposed into $\beta=\beta''\circ(\beta')^{-1}$ with $\beta'$ and $\beta''$ being uniformly continuous coarse maps such that $(\beta')^{-1}$ is a continuous coarse map. The naturality of the slant product under the pair $(\alpha,\beta)$ now follows from the naturality of \cref{thm_LocSlantNatural} under the pairs $(\alpha',\beta')$ and $(\alpha'',\beta'')$.\qedhere
\end{itemize}
\end{rem}

\subsection{Coarsified versions of the external and slant products}
\label{sec_compatibility_Rips}
\begin{sloppypar}
In this section we first recall the notions of coarse $\K$-homology $\KX_*$ and the coarse $\K$-theory $\KX^*$ and define the coarse structure group $\SX_*$. Then we use the results of the previous sections to construct coarsified versions of the external and slant products
\begin{alignat*}{3}
\times\colon&& \KX_m(X)&\otimes \KX_n(Y)&&\to \KX_{m+n}(X\times Y)
\\\times\colon&& \SX_m(X)&\otimes \KX_n(Y)&&\to \SX_{m+n}(X\times Y)
\\/\colon&& \KX_p(X\times Y)&\otimes \KX^q(Y)&&\to \KX_{p-q}(X)
\\/\colon&& \KX_p(X\times Y)&\otimes \K_{1-q}(\sHigCorRed Y)&&\to \KX_{p-q}(X)
\\/\colon&& \SX_p(X\times Y)&\otimes \K_{1-q}(\sHigCorRed Y)&&\to \SX_{p-q}(X)
\end{alignat*}
and show their compatibility with the maps in a coarsified version of the Higson--Roe analytic sequence
\begin{equation}\label{eq_coarseHigsonRoesequence}
\dots\to\K_{*+1}(\Roe X)\to \SX_*(X)\to \KX_*(X)\xrightarrow{\mu} \K_*(\Roe X)\to\dots\,,
\end{equation}
the coarsification maps $\K_*(X)\to \KX_*(X)$, $\Strg_*(X)\to \SX_*(X)$ and the co-coarsification map $\KX^*(Y)\to \K^*(Y)$, as well as with the co-assembly map $\K_{1-q}(\sHigCorRed(Y))\to \KX^q(Y)$.
We also generalize our results obtained in \cref{sec_composition_cross_prods} to the coarsifications.
\end{sloppypar}

\subsubsection{Definition of the coarse theories}
The coarse $\K$-homology $\KX_*$ and the coarse structure groups $\SX_*$ are defined using the Rips complex construction. We first consider the case of discrete proper metric spaces.
\begin{defn}
Let $X$ be a  discrete proper metric space and $R\geq 0$. The Rips complex of $X$ at scale $R$ is (the geometric realization of) the simplicial complex $P_RX$ whose vertex set is $X$ and whose simplices are those spanned by the finite sets of vertices of diameter at most $R$.
\end{defn}

The Rips complexes can be metrized by proper metrics such that all the inclusions $X\subset P_RX\subset P_SX$ for $0\leq R<S$ are
isometric coarse equivalences. These inclusions turn the Rips complexes into a directed system indexed over $\R_{\geq 0}$ and we can define
\begin{align*}
\KX_*(X)&\coloneqq \varinjlim_{R\geq 0}\K_*(P_RX)&\text{and}&
&\SX_*(X)&\coloneqq \varinjlim_{R\geq 0}\Strg_*(P_RX)\,.
\end{align*}
If $\alpha\colon X\to Y$ is a coarse map between discrete proper metric spaces then for each $R\geq 0$ there exists $S\geq 0$ such that $\alpha$ extends linearily to a continuous coarse map $P_RX\to P_SY$ which we denote by the same letter $\alpha$.
If $\beta\colon X\to Y$ is another coarse map which is close to $\alpha$ and $\beta\colon P_RX\to P_TY$ is its linear extension to the Rips complex, then the extensions of $\alpha$ and $\beta$ are homotopic after postcomposing them with the inclusion into $P_UX$ for some $U\gg\max\{S,T\}$. The homotopy is constructed by linear interpolation between $\alpha$ and $\beta$ and consists of uniformly continuous coarse maps which all belong to the same closeness class.
This property together with the homotopy invariance of $\K$-homology and the structure groups\footnote{By \cref{prop_LocFunctoriality} we conclude that the structure group is homotopy invariant for uniformly continuous coarse maps. By \cref{rem_better_functoriality_structgroup} we can extend this to invariance for continuous coarse map, and using \cref{eq_coarseHigsonRoesequence} together with the invariance of $\K_*(\Roe X)$ under coarse homotopies \cite{higson_roe_homotopy} we can push this even further to invariance under maps which are  simultaneously proper continuous and coarse homotopies. But since we do not need this generality here, we will not provide the details of these.} immediately implies functoriality of $\KX_*$ and $\SX_*$ under closeness classes of coarse maps.

\begin{defn}
For any proper metric space $X$ we define
\begin{align*}
\KX_*(X)&\coloneqq \varinjlim_{R\geq 0}\K_*(P_RX')
&\text{and}&
&\SX_*(X)&\coloneqq \varinjlim_{R\geq 0}\Strg_*(P_RX')
\end{align*}
where $X'\subset X$ is any discrete coarsely equivalent subspace. These groups are independent of the choice of discretization $X'\subset X$ up to canonical isomorphism and functorial under closeness classes of coarse maps.
\end{defn}

Note that this coarsening procedure does not yield anything new for the $\K$-theory of the Roe algebra and the stable Higson corona, since the directed systems of Rips complexes consist only of coarse equivalences and therefore $\K_*(\Roe X)\cong\varinjlim_{R\geq 0}\K_*(\Roe(P_RX'))$ and $\K_*(\sHigCorRed X)\cong\varprojlim_{R\geq 0}\K_*(\sHigCorRed(P_RX'))$.
The definition of the coarse $\K$-theory $\KX^*$ is a bit more complicated (cf.\ \cite[Definition 4.3, Note 4.4]{EmeMey}) but we note the following.
\begin{lem}[{cf.\ \cite[Remark 4.5]{EmeMey}}]
For any proper metric space $X$ with discrete coarsely equivalent subspace $X'\subset X$ there is a Milnor-$\varprojlim^1$-sequence
\[0\to{\varprojlim_{R\geq 0}}^1\K^{*+1}(P_RX')\to \KX^*(X)\to \varprojlim_{R\geq 0}\K^*(P_RX')\to 0\]
which is natural for coarse maps in the obvious way, i.\,e.\ the following holds:

If $f\colon X\to Y$ is a coarse map and $X'\subset X$, $Y'\subset Y$ are coarsely equivalent discrete subsets, then every coarse map $f'\colon X'\to Y'$ which is close to the restriction of $f$ to $X'$ induces maps between the left and right terms of the short exact sequences. These induced maps are up to canonical isomorphism independent of the choices of $X',Y'$ and $f'$, and together with $f^*$ on the middle term they constitute a map between short exact sequences.\qed
\end{lem}
Our slant products will factor through the limit on the right hand side and hence we can ignore the $\varprojlim^1$-term for our purposes.

\subsubsection{Maps between the coarse theories}
Now that we have introduced all the relevant groups, let us describe the maps relating them. We start with the coarsified Higson--Roe exact sequence.
\begin{defn}
The coarsified Higson--Roe sequence \eqref{eq_coarseHigsonRoesequence} is obtained as the limit of the Higson--Roe sequences of the Rips complexes $P_RX'$.
\end{defn}
It is again exact and it is clearly natural under coarse maps. Note by the way that the coarse Baum--Connes conjecture for a space $X$ is equivalent to the vanishing of $\SX_*(X)$.

Second, there is the coarsified version of the co-assembly map, which we do not define here in order to avoid having to define $\KX^*$ precisely.
However, we have the following description.
\begin{rem}[{\cite[Definition 4.6]{EmeMey}}]
For any proper metric space $X$ there is a coarsified co-assembly map
\[\mu^*\colon\K_{1-*}(\sHigCorRed X)\to \KX^*(X)\]
whose composition with the homomorphism $\KX^*(X)\to \varprojlim_{R\geq 0}\K^*(P_RX')$ is equal to the limit of the coassembly maps
\[\K_{1-*}(\sHigCorRed X)\cong \K_{1-*}(\sHigCorRed(P_RX'))\xrightarrow{\mu^*} \K^*(P_RX')\,.\]
These maps are all natural under coarse maps.
\end{rem}

Finally, there are also the coarsification and co-coarsification maps, which are defined as follows. Given a proper metric space $X$, we choose a discrete $R$-dense subset $X'\subset X$ for some $R>0$ and a partition of unity $\{\varphi_{x'}\}_{x'\in X'}$ subordinate to the cover $\{\Ball_R(x')\}_{x'\in X'}$ of $X$. Then the map
\begin{equation}
\label{eq_into_Rips}
X\to P_{2R}X'\,,\quad x\mapsto \sum_{x'\in X'}\varphi_{x'}(x)\cdot x'
\end{equation}
is a continuous coarse equivalence. Furthermore, if a second map $X\to P_{2S}X''$ is defined in exactly the same way using another $S$-dense discrete subset $X''$ and another partition of unity, then the two maps become homotopic via a homotopy of continuous coarse equivalences after postcomposing them with the inclusion into $P_T(X'\cup X'')$ for some $T\gg\max\{2R,2S\}$. Therefore, the following maps are independent of all choices (using \cref{rem_better_functoriality_structgroup} to get homotopy invariance of the structure group for continuous coarse maps).\footnote{Under mild assumptions on the proper metric space $X$ we can even arrange \cref{eq_into_Rips} to be a uniformly continuous coarse equivalence \cite[Section~7]{bunke_engel_coarse_assembly} and hence we would not need the better homotopy invariance of the structure group discussed in \cref{rem_better_functoriality_structgroup}.}
\begin{defn}
The coarsification maps
\begin{align*}
\coarsify\colon\K_*(X)&\to  \varinjlim_{R\geq 0}\K_*(P_RX')=\KX_*(X)
\\\coarsify\colon\Strg_*(X)&\to\varinjlim_{R\geq 0}\Strg_*(P_RX')=\SX_*(X)
\end{align*}
and the co-coarsification map
\[ \coarsify^*\colon\KX^*(X)\to\varprojlim_{R\geq 0}\K^*(P_RX')\to\K^*(X)\]
are defined as the maps induced by any map $X\to P_{2R}X'$ as above.
\end{defn}

\begin{lem}
The coarsification and co-coarsification maps are natural under continuous coarse maps.
\end{lem}

\begin{proof}
Assume that a diagram
\[\xymatrix@C=0ex{
X\ar[d]_\alpha &\supseteq&X'\ar[d]^{\alpha'}
\\Y&\supseteq&Y'
}\]
is given, where $\alpha$ is a continuous coarse map between proper metric spaces and $\alpha'$ is a coarse map between discrete coarsely equivalent subsets such that the diagram commutes up to closeness.
Then for each sufficiently large $R\geq 0$ there is $S\geq 0$ such that the diagram
\[\xymatrix{
X\ar[d]_\alpha\ar[r] &P_{2R}X'\ar[d]^{\alpha'}
\\Y\ar[r]&P_{2S}Y'
}\]
of continuous coarse maps commutes up to closeness. If $S$ is chosen large enough, then the diagram even commutes up to a homotopy via continuous coarse maps which are close to $\alpha$, the homotopy being defined by a linear interpolation. The claim follows.
\end{proof}

The following important property follows immediately from the definition.
\begin{lem}
The coarse assembly and co-assembly maps decompose into the compositions of their coarsified counterparts and the \parensup{co-}coarsification maps, i.\,e.\ the diagrams
\[\xymatrix@C=0ex{
\K_*(X)\ar[rr]^-{\mu}\ar[dr]_-{\coarsify}&&\K_*(\Roe X)&\qquad&\K_{1-*}(\sHigCorRed X)\ar[rr]^-{\mu^*}\ar[dr]_-{\mu^*}&&\K^*(X)
\\&\KX_*(X)\ar[ur]_-{\mu}&&&&\KX^*(X)\ar[ur]_-{\coarsify^*}&
}\]
commute.\qedhere
\end{lem}

Let us recall for future reference the well-known conditions under which the (co-)coarsification maps are isomorphisms. To this end we first recall the following definition:
\begin{defn}
Let $X$ be a metric space. We call $X$ \emph{uniformly contractible}, if for every $R > 0$ there exists an $S \ge R$ such that for every point $x \in X$ the inclusion $B_R(x) \hookrightarrow B_S(x)$ is nullhomotopic.
\end{defn}

The proof of the following result can be found in several places in the literature like \cite[Chapter~2]{roe_index_coarse}, \cite[Proposition~6.105]{buen}, \cite[Theorem~7.6.2]{nowak_yu} or \cite[Theorem~4.8]{EmeMey}.

\begin{prop}\label{prop_uniformly_contractible}
Let $X$ be a proper metric space of bounded geometry. If $X$ is uniformly contractible, then the \parensup{co-}coarsification maps are isomorphisms.
\end{prop}

\subsubsection{The coarsified external and slant products}\label{subsec_equivcoarseifiedcrossslant}

The construction of the coarsified external and slant products relies on the following easy lemma, which essentially says that the products of Rips complexes $P_RX'\times P_RY'$ can be seen as deformation retracts of the Rips complex $P_R(X'\times Y')$ of the products, up to enlarging the scale of the Rips complexes.

\begin{lem}
Let $X'$ and $Y'$ be discrete proper metric spaces. For every $R\geq 0$ we define the continuous coarse equivalences
\begin{align*}
p_R\colon P_R(X'\times Y')&\to P_RX'\times P_RY'
\\i_R\colon P_RX'\times P_RY'&\to P_{2R}(X'\times Y')
\end{align*}
by the formulas
\begin{align*}
p_R\bigg(\sum_{(x',y')\in X'\times Y'}\!\!\nu_{(x',y')}\cdot (x',y')\bigg)&=\bigg(\sum_{x'\in X'}\sum_{y'\in Y'}\nu_{(x',y')}\cdot x', \sum_{y'\in Y'}\sum_{x'\in X'}\nu_{(x',y')}\cdot y'\bigg)
\\i_R\bigg(\sum_{x'\in X'}\kappa_{x'}\cdot x',\sum_{y'\in Y'}\lambda_{y'}\cdot y'\bigg)&= \sum_{(x',y')\in X'\times Y'}\!\!\kappa_{x'}\lambda_{y'}\cdot (x',y')\,.
\end{align*}
Then $p_{2R}\circ i_R$ is equal to the inclusion $P_RX'\times P_RY'\to P_{2R}X'\times P_{2R}Y'$
and the composition $i_R\circ p_R$ is homotopic to the inclusion $P_R(X'\times Y')\to P_{2R}(X'\times Y')$ via a homotopy of continuous coarse maps which are all close to the inclusion map.\qed
\end{lem}

\begin{cor}
Given proper metric spaces $X$ and $Y$ with discrete coarsely equivalent subsets $X'\subset X$ and $Y'\subset Y$, the maps $p_R$ and $i_R$ give rise to natural isomorphisms
\begin{align*}
\KX_*(X\times Y)&=\varinjlim_{R\geq 0}\K_*(P_R(X'\times Y'))\cong\varinjlim_{R\geq 0}\K_*(P_RX'\times P_RY')\,,
\\\SX_*(X\times Y)&=\varinjlim_{R\geq 0}\Strg_*(P_R(X'\times Y'))\cong\varinjlim_{R\geq 0}\Strg_*(P_RX'\times P_RY')\,.
\end{align*}
\end{cor}

The existence of the coarsified external and slant products are now a direct consequence of this corollary and the naturality of the external and slant products under the pairs of inclusions $P_RX'\subset P_SX'$ and $P_RY'\subset P_SY'$ for $R\leq S$.

\begin{defn}\label{defn_coarsified_products}
The corsified external products
\begin{alignat*}{3}
\times\colon&& \KX_m(X)&\otimes \KX_n(Y)&&\to \KX_{m+n}(X\times Y)
\\\times\colon&& \SX_m(X)&\otimes \KX_n(Y)&&\to \SX_{m+n}(X\times Y)
\end{alignat*}
are obtained by taking the direct limit over the external products
\begin{alignat*}{3}
\times\colon&& \K_m(P_RX')&\otimes \K_n(P_RY')&&\to \K_{m+n}(P_RX'\times P_RY')
\\\times\colon&& \Strg_m(P_RX')&\otimes \K_n(P_RY')&&\to \Strg_{m+n}(P_RX'\times P_RY')
\end{alignat*}
and if $Y$ has bounded geometry then the coarsified slant products
\begin{alignat*}{3}
/\colon&& \KX_p(X\times Y)&\otimes \KX^q(Y)&&\to \KX_{p-q}(X)
\\/\colon&& \KX_p(X\times Y)&\otimes \K_{1-q}(\sHigCorRed Y)&&\to \KX_{p-q}(X)
\\/\colon&& \SX_p(X\times Y)&\otimes \K_{1-q}(\sHigCorRed Y)&&\to \SX_{p-q}(X)
\end{alignat*}
are obtained by taking the direct limit over the slant products
\begin{alignat*}{4}
/\colon&& \K_p(P_RX'\times P_RY')&\otimes \KX^q(Y)&\to& \K_p(P_RX'\times P_RY')\otimes \K^q(P_RY')
\\&&&&\to& \K_{p-q}(P_RX')
\\/\colon&& \K_p(P_RX'\times P_RY')&\otimes \K_{1-q}(\sHigCorRed Y)&\cong& \K_p(P_RX'\times P_RY')\otimes \K_{1-q}(\sHigCorRed (P_RY'))
\\&&&&\to& \K_{p-q}(P_RX')
\\/\colon&& \Strg_p(P_RX'\times P_RY')&\otimes \K_{1-q}(\sHigCorRed Y)&\cong& \Strg_p(P_RX'\times P_RY')\otimes \K_{1-q}(\sHigCorRed  (P_RY'))
\\&&&&\to& \Strg_{p-q}(P_RX')\,.
\end{alignat*}
In the special case that $X$ is a single point we obtain the coarsified versions of the pairings
\begin{alignat*}{3}
\langle-,-\rangle\colon&&\KX_p(Y)&\otimes\KX^q(Y)&\to&\K_{p-q}(\C)
\\\langle-,-\rangle\colon&&\KX_p(Y)&\otimes\K_{1-q}(\sHigCorRed Y)&\to&\K_{p-q}(\C)
\end{alignat*}
whose values lie in $\Z$ if $p-q$ is even and who vanish if $p-q$ is odd.
\end{defn}

The following properties are obvious by applying the properties of the uncoarsified external and slant products proven in Sections \ref{sec_commutativity_diagram}, \ref{sec_recovering_usual_slant_prod}, \ref{sec_composition_cross_prods}, \ref{sec_functoriality} to the Rips complexes and taking limits.

\begin{thm}
The coarsified external and slant products are natural for pairs of coarse maps and compatible with the maps in the coarsified Higson--Roe sequence \eqref{eq_coarseHigsonRoesequence}, the coarsification maps $\K_*(X)\to \KX_*(X)$, $\Strg_*(X)\to \SX_*(X)$ and co-coarsification map $\KX^*(Y)\to \K^*(Y)$, as well as the co-assembly map $\K_{1-q}(\sHigCorRed Y)\to \KX^q(Y)$.\footnote{For compatibility with the co-assembly map we have to note that if $Y$ has bounded geometry, then the Rips complex of any uniformly discrete, coarsely equivalent subset of it has continuously bounded geometry since it will be a simplicial complex of bounded geometry; see \cref{defn:boundedGeometry}.}

Further, taking first the external product with an element $z\in \KX_m(Y\times Z)$ and then the slant product with an element $\theta\in K_{1-n}(\sHigCorRed Z)$ or $\theta\in \KX^n(Z)$ is equal to the external product with $z/\theta\in \KX_{m-n}(Y)$, and in particular if $Y=\{*\}$ is a one-point space, then this composition is equal to multiplication with $\langle z,\theta\rangle = \langle z,\mu^*(\theta)\rangle=\langle \mu(z),\theta\rangle$.\qed
\end{thm}

\section{Equivariant slant products}
\label{sec_equiv_version}

We now generalize the results from the previous section to an equivariant setup.
Throughout this section let $X,Y$ be proper metric spaces and let $G,H$ be countable discrete groups acting properly and isometrically on $X$ and $Y$, respectively. Furthermore we assume that $Y$ has bounded geometry. The cases where $Y$ is required to even have continuously bounded geometry will be pointed out explicitly.

We have already seen in \cref{subsec_crossProducts_equiv} what the equivariant Higson--Roe sequence
\[\dots\to \K_{*+1}(\Roe[G] X)\to \Strg_*^G(X)\to \K_*^G(X)\xrightarrow{\Ind} \K_*(\Roe[G] X)\to\dots\]
is and how the equivariant external products with elements of $\K_*^H(Y)$ and $\K_*(\Roe[H] Y)$ are constructed. This theory is already well established.

On the cohomological side we use crossed products to define equivariant analogues of $\K^*(Y)$ and $\K_*(\sHigCorRed Y)$. First of all, we define the equivariant $\K$-theory of $Y$ as
\[\K^*_H(Y)\coloneqq \K_{-*}(\Cz(Y)\rtimes H)\,.\]
This definition is justified by \cite[Theorem 6.8]{BaumHigsonSchickEquivariantKhomology} which says that $\K_H^0(Y)$ is naturally isomorphic to the Grothendieck group of $H$-equivariant vector bundles on $Y$ if the action of $H$ on $Y$ is cocompact.
Moreover, if the action of $H$ on \(Y\) is free, then it is well known that \(\Cz(Y) \rtimes H\) is Morita equivalent to \(\Cz(H \backslash Y)\) (see also \cref{subsec_compatibility_equiv_nonequiv} below) and therefore \(\K^*_H(Y) \cong \K^*(H \backslash Y)\).

Note that since we assume $H$ to act properly on $Y$, it follows from \cite[Remark 3.4.16]{echterhoff_KK_BC_overview} that the maximal crossed product of $\Cz(Y)$ by $H$ coincides with the reduced one, and consequently also with any other. Hence we only write \(\Cz(Y) \rtimes H\).

At first sight it might look tempting to define $\K^*_H(Y)$ differently\footnote{Not only the definitions are different, but there is apparently not even an isomorphism between them.}
as the $E$-theory group $E_{-*}^H(\C,\Cz(Y))$ in order to define a slant product
\begin{equation}\label{eq_KhomKtheslant}
\K_p^{G\times H}(X\times Y)\otimes \K^q_H(Y)\to\K_{p-q}^G(X)
\end{equation}
as in \cref{defn_Eslant} via $E$-theoretic products, but there is one big problem with this attempt: There simply is no $E$-theoretic product which gets rid of the $H$-equivariance as is needed for \eqref{eq_KhomKtheslant}.
The problem seems to be that both entries of equivariant $E$-theory have to be \textCstar\nobreakdash-algebras which are being acted on by the same group, although we would prefer to consider $\C$ without any action in this case.
Perhaps it is possible to fix this issue by generalizing the notion of equivariant $E$-theory groups to allow for different equivariances in the two entries.
\begin{question}
Let $A$ be a $G$-\textCstar\nobreakdash-algebra, $B$ be a $H$-\textCstar\nobreakdash-algebra and $\alpha$ be a homomorphism between $G$ and $H$.
Is there a meaningful notion of $E$-theory groups $E^\alpha_*(A,B)$ which specialize to $E^G_*(A,B)$ if $G=H$ and $\alpha$ is the identity?
\end{question}

Instead of going the $E$-theory path, we circumvent this problem by defining the equivariant slant product \eqref{eq_KhomKtheslant} in \cref{def_topolequislant} below for $Y$ of continuously bounded geometry ad hoc by turning an equivariant version of \cref{lem_alternativeKhomslant} into a definition.

In order to construct the equivariant analogues of the other slant products, we use $\K_*(\sHigCorRed Y\rtimes_\mu H)$ as the equivariant analogue of $\K_*(\sHigCorRed Y)$, where $\rtimes_\mu$ denotes any exact crossed product functor in the sense of~\cite[Definition~3.1]{BaumGuentWillExactCrossed}. Natural choices for exact crossed product functors are the maximal crossed product, the minimal exact crossed product \cite{BEW_minimal_exact}, the minimal exact and Morita compatible crossed product used in the recent reformulation of the Baum--Connes conjecture due to Baum, Guentner and Willett~\cite[Definition~4.1]{BaumGuentWillExactCrossed},\footnote{Note that Buss, Echterhoff and Willett \cite{BEW_minimal_exact} claimed that the minimal exact and Morita compatible crossed product functor coincides with the minimal exact one. But a gap was found in their proof invalidating this claim, see the erratum in the appendix of arXiv version~3 of \cite{BEW_minimal_exact}.} or the reduced crossed product in the case of $H$ being an exact group.

Before we construct these slant products in \cref{subsec_EquivSlantProd} and prove their properties in \cref{subsec_EquivSlantProp}, let us first explain in the next two sections how the group $\K_*(\sHigCorRed Y\rtimes_\mu H)$ appears in the theory of equivariant co-assembly and show that it contains sufficiently many elements for our purposes.

\subsection{Equivariant co-assembly maps}
\label{label_sec_factorization_assembly_coassembly}
We have already mentioned that the maximal crossed product $\Cz(Y)\rtimes_{\max}H$ coincides with the reduced crossed product $\Cz(Y)\rtimes_{\red}H$ and hence also with any other crossed product. As both the maximal and the reduced crossed product are Morita compatible, the crossed products $\Cz(Y,\Kom)\rtimes_{\max}H$ and $\Cz(Y,\Kom)\rtimes_{\red}H$ are both isomorphic to $(\Cz(Y)\rtimes H)\otimes\Kom$ and hence the same is true for any other crossed product $\Cz(Y,\Kom)\rtimes_\mu H$. It follows that $\K_*(\Cz(Y,\Kom)\rtimes_\mu H)\cong \K_H^{-*}(Y)$ for all crossed product functors $\rtimes_\mu$.

Now, for any exact crossed product $\rtimes_\mu$ we have a short exact sequence
\begin{equation}
\label{eq_ses_crossed_product}
0\to \Cz(Y, \Kom) \rtimes_{\mu} H \to \sHigComRed Y \rtimes_{\mu} H \to \sHigCorRed Y \rtimes_{\mu} H \to 0\,,
\end{equation}
whose connecting homomorphism
\[\mu^*_H\colon \K_*(\sHigCorRed Y \rtimes_{\mu} H)\to \K_H^{1-*}(Y)\]
is an equivariant version of the coarse co-assembly map. Note that if $\rtimes_{\mu^\prime}$ is another exact crossed product and if we have a natural transformation of crossed product functors $\rtimes_\mu \to \rtimes_{\mu^\prime}$, then we get a transformation between the corresponding short exact sequences \eqref{eq_ses_crossed_product} and consequently a commuting triangle
\begin{equation}
\label{eq_transformation_different_crossed_products}
\xymatrix{
\K_*(\sHigCorRed Y \rtimes_{\mu} H) \ar[r] \ar[d] & \K_H^{1-*}(Y)\\
\K_*(\sHigCorRed Y \rtimes_{\mu^\prime} H) \ar[ur]
}
\end{equation}
relating the two equivariant coarse co-assembly maps.

Emerson and Meyer constructed in \cite{EM_descent}, see also \cite[Section~2.3]{EM_coass}, the co-assembly map
\[\mu^\ast_\EM\colon \Ktop_*(H,\sHigCorRed Y) \to \K_H^{1-*}(Y)\,,\]
where the left hand side is defined as in the Baum--Connes conjecture for the coefficient \Cstar-algebra $\sHigCorRed Y$. The Baum--Connes assembly map is, with these coefficients and using $\rtimes_\mu$, a map $\mu_\ast^\BC\colon \Ktop_*(H,\sHigCorRed Y) \to \K_\ast(\sHigCorRed Y \rtimes_{\mu} H)$.\footnote{Note the slight technicality here that $\sHigCorRed Y$ is in general not separable. We are workig here with the convention that in the non-separable case we just take everywhere directed limits over the separable sub-\textCstar-algebras.} It is, by definition, the composition of the maximal version $\Ktop_*(H,\sHigCorRed Y) \to \K_\ast(\sHigCorRed Y \rtimes_{\max} H)$ with the natural quotient map $\K_\ast(\sHigCorRed Y \rtimes_{\max} H) \to \K_\ast(\sHigCorRed Y \rtimes_{\mu} H)$,~\cite[Display~(2.2)]{BaumGuentWillExactCrossed}. We can form the following diagram:
\begin{equation}
\label{eq_diag_alle_assemblies}
\xymatrix{
\Ktop_*(H,\sHigCorRed Y) \ar[rr]^-{\mu^\ast_\EM} \ar[dr]_-{\mu_\ast^\BC} && \K_H^{1-*}(Y)\\
& \K_\ast(\sHigCorRed Y \rtimes_{\mu} H) \ar[ur]_-{\mu^*_H} &
}
\end{equation}

\begin{lem}\label{lem_diag_commutes_assembly_coassembly_EM_BC}
If $\rtimes_\mu$ is an exact crossed product functor,\footnote{Note that it must be defined for non-separable \textCstar-algebras. One way to get such a functor is to use again directed limits over separable sub-\textCstar-algebras \cite[Lemma~8.11]{BEW_exotic_crossed_BC}.} then the Diagram~\eqref{eq_diag_alle_assemblies} commutes.
\end{lem}

\begin{proof}
Note that because of Diagram~\eqref{eq_transformation_different_crossed_products} it is sufficient to consider the case of the maximal crossed product $\rtimes_\mu=\rtimes_{\max}$.
Then this lemma is basically true more or less directly by definition of the map $\mu^*_\EM$ as given in \cite{EM_descent}. For the convenience of the reader let us recall some of the details.

By \cite[Definition~12]{EM_descent} we have the following commutative diagram
\[
\xymatrix{
\K_*((\sHigCorRed Y \otimes_{\max} \rmP) \rtimes_{\max} H) \ar[r]^-{\partial} \ar@{<->}[d]_-{\cong} & \K_{*-1}((\Cz(Y) \otimes_{\max}\rmP)\rtimes_{\max} H) \ar[d]_-{\cong}^-{\rmD_\ast}\\
\Ktop_*(H,\sHigCorRed Y) \ar[r]^-{\mu^\ast_\EM} & \K_H^{1-*}(Y)
}
\]
which is used to define the map $\mu^*_\EM$. Here $\partial$ is the boundary map induced by a certain short exact sequence of \Cstar-algebras, and $\rmP$ is an $H$-\Cstar-algebra which supports the so-called Dirac morphism $\rmD \in \KK^H(\rmP,\IC)$.
It is known that the (maximal version of the) Baum--Connes assembly map with coefficients in a \Cstar-algebra $A$ is equivalent to the map $\rmD_\ast\colon \K_\ast((A \otimes_{\max} \rmP) \rtimes_{\max} H) \to \K_\ast(A \rtimes_{\max} H)$, and the left vertical isomorphism in the previous diagram is the one which identifies the corresponding domains~\cite[Theorem~5.2]{meyer_nest}. That the right vertical map in the above diagram is an isomorphism is explained in \cite[Beginning of Section~2.7]{EM_descent}.

Applying $\rmD_\ast$ to the boundary map in the above diagram, we get
\[
\xymatrix{
\K_*(\sHigCorRed Y \rtimes_{\max} H) \ar[r]^-{\partial} & \K_H^{1-*}(Y)\\
\K_*((\sHigCorRed Y \otimes_{\max} \rmP) \rtimes_{\max} H) \ar[r]^-{\partial} \ar[u]_-{\rmD_\ast} & \K_{*-1}((\Cz(Y) \otimes_{\max}\rmP)\rtimes_{\max} H) \ar[u]_-{\rmD_\ast}^\cong
}
\]
which commutes because Kasparov products are compatible with boundary maps induced from short exact sequences.

The top horizontal map in the last diagram coincides with $\mu^\ast_H$, and the composition
\[\Ktop_*(H,\sHigCorRed Y) \stackrel{\cong}\longleftrightarrow \K_*((\sHigCorRed Y \otimes_{\max} \rmP) \rtimes_{\max} H) \stackrel{\rmD_\ast}\longrightarrow \K_*(\sHigCorRed Y \rtimes_{\max} H)\]
identifies with the Baum--Connes assembly map $\mu_\ast^\BC$, which is a property of the Dirac morphism $\rmD_\ast$; see \cite[Sections~4--6]{meyer_nest} or \cite[Theorem~22]{EM_descent}.
\end{proof}

Commutativity of Diagram~\eqref{eq_diag_alle_assemblies} implies the next corollary stating that in many cases we have enough elements to take equivariant slant products.

Before we go into the corollary, let us recall that $\Eub H$ denotes the classifying space for proper $H$-actions. We say that it is $H$-finite, if it consists only of finitely many $H$-orbits of cells.
Models for $\Eub H$ are unique up to equivariant homotopy, which implies that if we have two $H$-finite models $\Eub H$ and $\Eub H^\prime$, then
\[\K_H^*(\Eub H) \cong \K_H^*(\Eub H^\prime)\,.\]

Furthermore, if we put any equivariant length metric on the $H$-finite model $\Eub H$ and pick any point $p\in \Eub H$, then the map $H\to\Eub H,h\mapsto hp$ is a coarse equivalence which is $H$-equivariant. Note that in general a coarse inverse $\Eub H\to H$ cannot be equivariant, because the action of $H$ on $\Eub H$ is not always free, but the above properties nevertheless ensure that the induced map $\sHigCorRed\Eub H \to\sHigCorRed H$ is an isomorphism of \textCstar\nobreakdash-algebras which is also $H$-equivariant, i.\,e.\ an isomorphism of $H$-\textCstar\nobreakdash-algebras. This isomorphism is even canonical, because the coarse equivalences $H\to\Eub H$ are pairwise close for different choices of the point $p\in\Eub H$. Consequently, we obtain a canonical isomorphism
\begin{equation}\label{eq_GEGiso}
\K_*(\sHigCorRed H \rtimes_{\mu} H) \cong \K_*(\sHigCorRed \Eub H \rtimes_{\mu} H)
\end{equation}
for each $H$-finite model for $\Eub H$.

In addition, any $H$-equivariant homotopy equivalence between two $H$-finite models $\Eub H$ and $\Eub H^\prime$ is automatically a quasi-isometry. Hence it induces a canonical ismomorphisms $\K_*(\sHigCorRed \Eub H \rtimes_{\mu} H) \cong \K_*(\sHigCorRed \Eub H^\prime \rtimes_{\mu} H)$ which is compatible with \eqref{eq_GEGiso}.

Everything that has been said in the above paragraphs tells us that the co-assembly map in the following corollaries is canonical.

\begin{cor}\label{cor_equiv_coarse_coass_surj}
Let $H$ be a countable discrete group that admits a $H$-finite classifying space for proper $H$-actions $\Eub H$.

If $H$ has a $\gamma$-element, then the equivariant coarse co-assembly map
\[\mu^*_H\colon \K_*(\sHigCorRed H \rtimes_{\mu} H) \to \K_H^{1-*}(\Eub H)\]
is surjective for any exact crossed product functor $\rtimes_\mu$.
\end{cor}

\begin{proof}
We consider Diagram~\eqref{eq_diag_alle_assemblies} for the group $H$ and the space $\Eub H$. It follows from \cite[Proposition~13]{EM_coass} that $\mu^*_\EM$ is an isomorphism since $H$ has a $\gamma$-element. Hence by Diagram~\eqref{eq_diag_alle_assemblies} the map $\mu^*_H$ is surjective.
\end{proof}

Let $\rtimes_\mu$ be a correspondence crossed product functor. Without defining what this is, we just note that such crossed product functors admit the descent homomorphism $\KK^H_*(A,B) \to \KK_*(A \rtimes_\mu H, B \rtimes_\mu H)$ which is further compatible with Kasparov products \cite[Proposition~6.1]{BEW_exotic_crossed_BC}. We assume that~$H$ admits a $\gamma$-element. Then the assembly map $\mu_\ast^\BC\colon \Ktop_*(H,A) \to \K_\ast(A \rtimes_{\mu} H)$ is an isomorphism onto the summand $\gamma \cdot \K_\ast(A \rtimes_{\mu} H)$ for every $H$-\Cstar-algebra $A$, where $\gamma$ acts as a projection via the Kasparov product. Hence \cref{cor_equiv_coarse_coass_surj} refines to the statement that
\[\mu^*_H\colon \gamma\cdot\K_*(\sHigCorRed H \rtimes_{\mu} H) \to \K_H^{1-*}(\Eub H)\]
is an isomorphism. Hence, if $\gamma$ acts as the identity, then $\mu^*_H$ is an isomorphism.
\begin{cor}\label{cor_equiv_coass_iso}
Let $H$ be a countable discrete group that admits a $H$-finite classifying space for proper $H$-actions $\Eub H$. Assume further that $H$ is exact and that it satisfies \parensup{the reduced version of} the Baum--Connes conjecture with coefficients.

Then the equivariant coarse co-assembly map
\[\mu^*_H\colon \K_*(\sHigCorRed H \rtimes_\red H) \to \K_H^{1-*}(\Eub H)\]
is an isomorphism, where $\rtimes_\red$ is the reduced crossed product functor.
\end{cor}
\begin{proof}
If $H$ is exact, then it has a $\gamma$-element and the reduced crossed product is exact (and it is always a correspondence crossed product). Therefore the claim follows from the discussion prior to the corollary, since satisfying (the reduced version of) the Baum--Connes conjecture with coefficients gives that $\gamma$ acts as the identity on $\K_*(\sHigCorRed H \rtimes_\red H)$.
\end{proof}

\begin{example}
Gromov hyperbolic groups satisfy the assumptions of \cref{cor_equiv_coass_iso}.
For them $\Eub H$ can be taken as the Rips complex $P_R(H)$ of $H$ for a large enough $R \gg 1$ (\cite[Section~2]{baum_connes_higson}, \cite{meintrup_schick}), the Baum--Connes conjecture with coefficients was proven by Lafforgue \cite{lafforgue_BC,puschnigg_BC}, and exactness follows from them having finite asymptotic dimension (\cite[Page~23]{gromov_asdim}\cite{roe_asdim_hyperbolic}).
\end{example}

\subsection{Exactness of groups and the stable Higson corona}
\label{sec_weak_containment}

In the previous sections we saw that the choice of crossed product $\sHigCorRed Y \rtimes_{\mu} H$ matters, and that we have a connection to exactness of the group $H$.
In the present section we will investigate this connection further, and relate it to the so-called \emph{weak containment} property of $\sHigCorRed Y$, i.e., the question in which cases we have $\sHigCorRed Y \rtimes_{\max} H \cong \sHigCorRed Y \rtimes_{\red} H$.
Most of the results here were developed in discussions with Rufus Willett.\footnote{One can prove similar results for $\sHigCor Y$ instead of $\sHigCorRed Y$. But since we are mainly using $\sHigCorRed Y$ in this paper, we have restricted our attention to it in this section.}

\subsubsection{The case of the (stable) Higson compactification}

Before we can prove the main result of this section (\cref{prop_amenable_Higson_compactification} below), we first need a technical result about the double dual of the stable Higson compactification $\sHigComRed Y$ of a metric space $Y$ and its relation to the (usual) Higson compactification of $Y$.
Recall that the \emph{Higson compactification} is the compact space corresponding to \(\HigCom(Y)\), the unital commutative \textCstar\nobreakdash-algebra of all bounded, continuous functions \(Y \to \C\) of vanishing variation.
The \emph{Higson corona} \(\HigCorSpace Y\) is the compact space defined via \(\Ct(\HigCorSpace Y) \coloneqq \HigCom(Y)/\Cz(Y)\).

\begin{lem}\label{lem_embedding_double_commutant}
Let $Y$ be a proper metric space.

Then there is an embedding of $\HigCom(Y)$ into the center of the double dual of $\sHigComRed Y$, i.\,e., an injective and unital $^\ast$-homomorphism
\begin{equation}
\label{eq_embedding_center_double_commutant}
\HigCom(Y) \to Z((\sHigComRed Y)^\ast{}^\ast)\,.
\end{equation}
If $Y$ is further equipped with an isometric action\footnote{Note that we do not need the action to be proper.} of a countable discrete group~$H$, then the map \eqref{eq_embedding_center_double_commutant} is equivariant for the induced actions of $H$ on the corresponding \textCstar\nobreakdash-algebras.
\end{lem}

\begin{proof}
Recall that the Banach space dual of the compact operators $\Kom(H)$ are the trace class operators $\mathcal{S}^1(H)$, and that the Banach space dual of $\mathcal{S}^1(H)$ is $\Lin(H)$, and hence $\Kom(H)^\ast{}^\ast \cong \Lin(H)$ \cite[Theorem~I.8.6.1]{blackadar_operator_algebras}.
Further, $\Lin(H)$ can be equipped with the ultra-weak topology\footnote{This is just the weak-$^\ast$ topology on $\Lin(H)$ if we consider it as the dual of $\mathcal{S}^1(H)$, i.e., a net $(T_\lambda)_{\lambda \in \Lambda}$ in $\Lin(H)$ converges ultra-weakly to $T$ if and only if $(|\tr(ST_\lambda)|)_{\lambda \in \Lambda}$ converges to $|\tr(ST)|$ for every $S \in \mathcal{S}^1(H)$.} whose restriction to the closed unit ball of $\Lin(H)$ coincides with the corresponding restriction of the weak operator topology \cite[Theorem~4.2.4]{murphy}.
Hence, choosing an orthonormal basis $(e_i)_{i \in \IN}$ of $H$ and setting $p_n \in \Kom(H)$ to be the orthogonal projection onto the linear span of $e_1, \ldots, e_n$, we get a sequence $(p_n)_{n \in \IN}$ of compact operators of norm $1$ converging in the weak operator topology to the identity operator $\id_H$ on $H$. Consequently, $(p_n)_{n \in \IN}$ converges ultra-weakly to $\id_H$.

Now choose any sequence of compact operators $(k_n)_{n \in \IN}$ on $\elltwo$ converging ultra-weakly to the identity $\id_{\elltwo}$.
The sequence $(f \otimes k_n)_{n \in \IN}$ for $f \in \HigCom(Y)$, which is a sequence in $\sHigComRed Y \subset (\sHigComRed Y)^\ast{}^\ast$, converges ultra-weakly in $(\sHigComRed Y)^\ast{}^\ast$.
Its limit is, by definition, the image of $f$ under the sought map \eqref{eq_embedding_center_double_commutant}.

It is clear that \eqref{eq_embedding_center_double_commutant} is an injective and unital $^\ast$-homomorphism.
It remains to show that its image is contained in the center of $(\sHigComRed Y)^\ast{}^\ast$. Now in general the extension of the product on $A$ to the correct product on $A^\ast{}^\ast$ was achieved by Arens \cite{arens_1,arens_2}. Looking at the formulas, we see that indeed the map \eqref{eq_embedding_center_double_commutant} ends up in the center of the double dual.

That in the situation of $Y$ being equipped with the action of a group $H$ the map \eqref{eq_embedding_center_double_commutant} will be equivariant, is quickly seen.
\end{proof}

For a discrete group $H$ we recall now the different notions of amenability of $H$-\textCstar-algebra from \citeauthor{BEW_amenable}~\cite[Definitions~2.1~and~4.13]{BEW_amenable}. Note that there are also variants of some of these notions occuring in, e.g., \cite{anantharaman,BrownOzawaCstarFiniteDim}. How these variants are related to each other is explained in \cite[Remark~2.2]{BEW_amenable}.

\begin{defn}\label{defn_amenabilities}
Let $H$ be a discrete group and $A$ be an $H$-\textCstar-algebra.
\begin{myenuma}
\item\label{item_strong_amenability} $A$ is called strongly amenable if there is a net $(\theta_i\colon H \to Z(\mathcal{M}(A)))_{i \in I}$ of positive type functions\footnote{In general, a function $\vartheta\colon H \to B$ is of positive type if for any finite subset $\{h_1, \ldots, h_n\}$ of $H$ the matrix $(\alpha_{h_i}(\vartheta(h_i^{-1} h_j)))_{i,j} \in M_n(B)$ is positive, where $\alpha$ is the action of $H$ on $B$ \cite[Definition~2.1]{anantharaman_annalen}.} such that
\begin{itemize}
\item each $\theta_i$ is finitely supported,
\item for each $i$ we have $\theta_i(e) \le 1$, and
\item for each $h \in H$ we have $\theta_i(h) \to 1$ strictly as $i \to \infty$.
\end{itemize}
\item $A$ is called amenable if there is a net $(\theta_i\colon H \to Z(A^\ast{}^\ast))_{i \in I}$ of positive type functions such that
\begin{itemize}
\item each $\theta_i$ is finitely supported,
\item for each $i$ we have $\theta_i(e) \le 1$, and
\item for each $h \in H$ we have $\theta_i(h) \to 1$ ultra-weakly\footnote{Recall that a net $(T_\lambda)_{\lambda \in \Lambda}$ in $A^\ast{}^\ast$ converges ultra-weakly to $T$ if and only if $(T_\lambda(\varphi))_{\lambda \in \Lambda}$ converges to $T(\varphi)$ for every $\varphi \in A^\ast$.} as $i \to \infty$.\qedhere
\end{itemize}
\end{myenuma}
\end{defn}

\noindent
Note that strong amenability implies amenability \cite[Remark~2.2]{BEW_amenable}.

\begin{prop}\label{prop_amenable_Higson_compactification}
Let $Y$ be a proper metric space equipped with an isometric action of a discrete group $H$.\footnote{Note that we do not need here the action of $H$ on $Y$ to be proper.}

Consider the following three statements:
\begin{myenuma}
\item\label{item_equiv_amenable_one} The group $H$ acts amenably on the Higson compactification of $Y$.
\item\label{item_equiv_amenable_two} $\sHigComRed Y$ is an amenable $H$-\textCstar-algebra.
\item\label{item_equiv_amenable_three} We have $\sHigComRed Y \rtimes_{\max} H \cong \sHigComRed Y \rtimes_{\red} H$.
\end{myenuma}
Then we have \ref{item_equiv_amenable_one} $\Rightarrow$ \ref{item_equiv_amenable_two} $\Rightarrow$ \ref{item_equiv_amenable_three}~and \ref{item_equiv_amenable_one}~additionally implies that $H$ is exact.
\end{prop}

\begin{proof}
We start with the implication \ref{item_equiv_amenable_one} $\Rightarrow$ \ref{item_equiv_amenable_two} and while doing this we also establish that \ref{item_equiv_amenable_one}~implies exactness of $H$.

\begin{itemize}
\item We show that \ref{item_equiv_amenable_one}~implies \ref{item_equiv_amenable_two}
If $H$ acts amenably on the Higson compactification of $Y$, then by \cite[Proposition~6.3]{anantharaman} the $H$-\textCstar-algebra $\HigCom(Y)$ is strongly amenable.
Since $\HigCom(Y)$ is commutative, by \cite[Lemma~2.5]{BEW_amenable} strong amenability of $\HigCom(Y)$ in the sense of \cite{anantharaman} is equivalent to strong amenability of it in the sense of \cite[Definition~2.1]{BEW_amenable}.
Therefore there exists a net $(\theta_i\colon H \to Z(\mathcal{M}(\HigCom(Y))))_{i \in I}$ of positive type functions satisfying the corresponding conditions listed in \cref{defn_amenabilities}. But since $\HigCom(Y)$ is unital and commutative, we have $Z(\mathcal{M}(\HigCom(Y))) = \HigCom(Y)$.
Composing with \eqref{eq_embedding_center_double_commutant}, we get a net $(\theta_i\colon H \to Z((\sHigComRed Y)^\ast{}^\ast))_{i \in I}$ showing that $\sHigComRed Y$ is amenable.\footnote{Here one has to know the fact that if $A$ is a unital \textCstar-algebra, then the strict topology on $\mathcal{M}(A) = A$ coincides with the norm topology.}

\item Since $\HigCom(Y)$ is unital, commutative and strongly amenable, this implies that $H$ is exact by \cite[Theorem~5.3]{BEW_amenable} (see also \cite{higson_roe_exact} for the fact that amenable actions on compact Hausdorff spaces imply exactness).

\item The implication \ref{item_equiv_amenable_two} $\Rightarrow$ \ref{item_equiv_amenable_three} is a completely general fact: amenability of an $H$-\textCstar-algebra $A$ implies $A \rtimes_{\max} H \cong A \rtimes_{\red} H$ by \cite[Proposition~4.8]{anantharaman_annalen} (see also \cite[Section~4]{BEW_amenable}).\qedhere
\end{itemize}
\end{proof}

\begin{rem}
Note that $\sHigComRed Y$ is a strongly amenable $H$-\textCstar-algebra if and only if $H$ is amenable.

First, because $\sHigComRed Y$ is unital, the notions of strong amenability as introduced by Anantharaman-Delaroche \cite{anantharaman} and by \citeauthor{BEW_amenable}~\cite{BEW_amenable} coincide for it \cite[Lemma~2.5]{BEW_amenable}.

If the group $H$ is amenable, then every $H$-\textCstar-algebra is strongly amenable.
Assume now that $\sHigComRed Y$ is strongly amenable, that is, we have a net $(\theta_i\colon H \to Z(\mathcal{M}(\sHigComRed Y)))_{i \in I}$ of positive type functions satisfying the corresponding conditions listed in \cref{defn_amenabilities}.
Since $\sHigComRed Y$ is unital we have $\mathcal{M}(\sHigComRed Y) = \sHigComRed Y$, and we further have that $Z(\sHigComRed Y) \cong \IC$. Hence the net $(\theta_i)_{i \in I}$ maps actually into $\IC$. But this means that $H$ is amenable.
\end{rem}

\begin{question}
The above results in combination with the results of \textup{\cite{BEW_amenable}} suggest that for a proper metric space $Y$ equipped with an isometric action of a discrete group $H$ the following conditions could be equivalent to each other:
\begin{myenuma}
\item\label{item_question_amenable_one} The group $H$ acts amenably on the Higson compactification of $Y$.
\item\label{item_question_amenable_two} $\sHigComRed Y$ is an amenable $H$-\textCstar-algebra.
\item\label{item_question_amenable_three} The group $H$ is exact and we have $\sHigComRed Y \rtimes_{\max} H \cong \sHigComRed Y \rtimes_{\red} H$.
\end{myenuma}
\end{question}

The corresponding version of the above question for the (stable) Higson corona should also be true and is stated in the introduction as \cref{conj_weak_containment}.

\subsubsection{The case of the (stable) Higson corona}

In this section we will adapt the results of the previous one to the (stable) Higson corona. We apply these results to the Gromov boundary of a hyperbolic group in \cref{ex_hyperbolic_amenable_action} below.

To treat the case of the (stable) Higson corona, we use the following fact about double duals of quotient \textCstar-algebras: if $0 \to I \to A \to A/I \to 0$ is any short exact sequence of \textCstar-algebras, then we have canonical isomorphisms $(A/I)^\ast{}^\ast \cong A^\ast{}^\ast / I^\ast{}^\ast$ and $A^\ast{}^\ast \cong I^\ast{}^\ast \oplus (A/I)^\ast{}^\ast$ \cite[Section III.5.2.11]{blackadar_operator_algebras}.

The next lemma is an adaption of \cref{lem_embedding_double_commutant}:

\begin{lem}\label{lem_embedding_double_commutant_corona}
Let $Y$ be a proper metric space.

Then there is an embedding of $\Ct(\HigCorSpace Y)$ into the center of the double dual of $\sHigCorRed Y$, i.e., an injective and unital $^\ast$-homomorphism
\begin{equation}
\label{eq_embedding_center_double_commutant_corona}
\Ct(\HigCorSpace Y) \to Z((\sHigCorRed Y)^\ast{}^\ast)\,.
\end{equation}
If $Y$ is further equipped with an isometric action of a discrete group~$H$, then the map \eqref{eq_embedding_center_double_commutant_corona} is equivariant for the induced actions of $H$ on the corresponding \textCstar-algebras.
\end{lem}

\begin{proof}
We have $\Ct(\HigCorSpace Y) \cong \HigCom(Y) / \Cz(Y)$ and $\sHigCorRed Y \cong \sHigComRed Y / \Cz(Y, \Kom)$.

\begin{sloppypar}
The *-homomorphism $f \mapsto f \otimes \id_{\elltwo}$ inducing \eqref{eq_embedding_center_double_commutant} maps the \textCstar-algebra $\Cz(Y)$ to $Z(\Cz(Y, \Kom)^\ast{}^\ast)$. It extends by \cite[Section III.5.2.10]{blackadar_operator_algebras} to a normal *-homomorphism\footnote{A normal *-homomorphism is one which is continuous for the respective ultra-weak topologies \cite[Section III.2.2]{blackadar_operator_algebras}.} $\Cz(Y)^\ast{}^\ast \to Z(\Cz(Y, \Kom)^\ast{}^\ast)$. Considering also the analogous normal extension of \eqref{eq_embedding_center_double_commutant} to $\HigCom(Y)^\ast{}^\ast$ we conclude that \eqref{eq_embedding_center_double_commutant} induces normal *-homomorphisms mapping the short exact sequence
\[
0 \to \Cz(Y)^\ast{}^\ast \to \HigCom(Y)^\ast{}^\ast \to \Ct(\HigCorSpace Y)^\ast{}^\ast \to 0
\]
to the short exact sequence
\[
0 \to \Cz(Y, \Kom)^\ast{}^\ast \to (\sHigComRed Y)^\ast{}^\ast \to (\sHigCorRed Y)^\ast{}^\ast \to 0
\]
and the image of $\Ct(\HigCorSpace Y)^\ast{}^\ast$ will be contained in $Z((\sHigCorRed Y)^\ast{}^\ast)$. The restriction of the map on $\Ct(\HigCorSpace Y)^\ast{}^\ast$ to its \textCstar-subalgebra $\Ct(\HigCorSpace Y)$ is the sought map \eqref{eq_embedding_center_double_commutant_corona}.
\end{sloppypar}

Let us show injectivity of \eqref{eq_embedding_center_double_commutant_corona}.
General theory \cite[Section~III.5.2.11]{blackadar_operator_algebras} tells us that there is a central projection $p \in (\sHigComRed Y)^\ast{}^\ast$ such that $p\cdot (\sHigComRed Y)^\ast{}^\ast$ is $\Cz(Y, \Kom)^\ast{}^\ast$ and such that $(1-p) \cdot (\sHigComRed Y)^\ast{}^\ast$ is $(\sHigCorRed Y)^\ast{}^\ast$.
This central projection is given as the supremum of an approximate unit $(u_\lambda)_{\lambda \in \Lambda}$ in $\Cz(Y, \Kom)$.
Let $f \in \HigCom(Y)$ be non-zero in $\Ct(\HigCorSpace Y)$. We have to show that it is still non-zero in $(\sHigCorRed Y)^\ast{}^\ast$, i.e., that it does not lie in $\Cz(Y, \Kom)^\ast{}^\ast$.
Because $f$ is non-zero in $\Ct(\HigCorSpace Y)$, there exists a point $x \in \HigCorSpace Y$ such that $\operatorname{ev}_x(f) \not= 0$. We choose any unit vector $v$ in the Hilbert space $H$ (the auxiliary Hilbert space used in the definition of $\sHigComRed Y$) and define a linear functional $\eta$ on $\sHigComRed Y$ by
\[
\eta(g) \coloneqq \operatorname{ev}_x(g_v)\,,
\]
where $g_v \in \HigCom(Y)$ is the function given by $g_v(y) \coloneqq \varphi_v(g(y))$ using the vector state $\varphi_v(T) \coloneqq \langle Tv,v\rangle$ on $H$.
Because $\eta\colon \sHigComRed Y \to \IC$ is a positive linear map, it extends to a normal positive linear map $\eta^\ast{}^\ast\colon (\sHigComRed Y)^\ast{}^\ast \to \IC^\ast{}^\ast \cong \IC$ and we have $\eta^\ast{}^\ast(f \otimes \id_{\elltwo}) = \operatorname{ev}_x(f) \not= 0$.
For $\theta \in \Cz(Y, \Kom)^\ast{}^\ast$ we have $\eta^\ast{}^\ast(\theta) = 0$ since (choosing $(\theta_\mu)_{\mu \in \Lambda^\prime} \in \Cz(Y, \Kom)$ approximating $\theta$ ultra-weakly)
\begin{align*}
\eta^\ast{}^\ast(\theta) & = \eta^\ast{}^\ast(p\cdot\theta)\\
& = \eta^\ast{}^\ast(\lim_{\lambda \to \infty} u_\lambda\cdot\theta)\\
& = \lim_{\lambda \to \infty} \eta^\ast{}^\ast(u_\lambda\cdot\theta)\\
& = \lim_{\lambda \to \infty} \eta^\ast{}^\ast(u_\lambda\cdot \lim_{\mu \to \infty} \theta_\mu)\\
& = \lim_{\lambda \to \infty} \lim_{\mu \to \infty} \eta^\ast{}^\ast(u_\lambda\cdot \theta_\mu)
\end{align*}
(where we used that $\eta^\ast{}^\ast$ is ultra-weakly continuous and multiplication separately ultra-weakly continuous) and $\eta^\ast{}^\ast(u_\lambda\cdot\theta_\mu)=0$ since $u_\lambda\cdot\theta_\mu \in \Cz(Y, \Kom)$. Because $\eta^\ast{}^\ast(f \otimes \id_{\elltwo}) \not= 0$, this means that $f \notin \Cz(Y, \Kom)^\ast{}^\ast$ finishing the proof that \eqref{eq_embedding_center_double_commutant_corona} is injective.

The other statements about the map \eqref{eq_embedding_center_double_commutant_corona} are straight-forward to prove, which finishes this proof.
\end{proof}

\begin{prop}\label{prop_amenable_Higson_corona}
Let $Y$ be a proper metric space equipped with an isometric action of a discrete group $H$.

Consider the following three statements:
\begin{myenuma}
\item\label{item_equiv_corona_amenable_one} The group $H$ acts amenably on the Higson corona of $Y$.
\item\label{item_equiv_corona_amenable_two} $\sHigCorRed Y$ is an amenable $H$-\textCstar-algebra.
\item\label{item_equiv_corona_amenable_three} We have $\sHigCorRed Y \rtimes_{\max} H \cong \sHigCorRed Y \rtimes_{\red} H$.
\end{myenuma}
Then we have \ref{item_equiv_corona_amenable_one} $\Rightarrow$ \ref{item_equiv_corona_amenable_two} $\Rightarrow$ \ref{item_equiv_corona_amenable_three}~and \ref{item_equiv_corona_amenable_one}~additionally implies that $H$ is exact.
\end{prop}

\begin{proof}
Completely analogous to the proof of \cref{prop_amenable_Higson_compactification}.
\end{proof}

\begin{example}\label{ex_hyperbolic_amenable_action}
Let $H$ be a Gromov hyperbolic group.

It is known that in this case $H$ acts amenably on its Gromov boundary \cite[Example~2.7.4]{anantharaman}. Since the Gromov boundary is dominated by the Higson corona, i.e., there is a natural $H$-map $\HigCorSpace H \to \partial_{\text{Gromov}} H$, it follows that~$H$ also acts amenably on its Higson corona.

\cref{prop_amenable_Higson_corona} then implies that $\sHigCorRed H$ is an amenable $H$-\textCstar-algebra and $\sHigCorRed H \rtimes_{\max} H \cong \sHigCorRed H \rtimes_{\red} H$.
\end{example}

\subsection{Construction of the equivariant slant products}\label{subsec_EquivSlantProd}

Just like in Section \ref{subsec_crossProducts_equiv} let us fix an ample $X$-$G$-module $(H_X,\rho_X,u_G)$ and an ample $Y$-$H$-module $(H_Y,\rho_Y,u_H)$.

\subsubsection{The slant product on the equivariant Roe algebra}\label{sec_slant_equiv_Roe_algebra}

We denote by $\rho_Y^\prime\colon \Cz(Y, \Kom) \to \Lin(H_Y \otimes \elltwo)$ the tensor product of the given representation $\rho_Y$ and the canonical representation of $\Kom$ on $\elltwo$. The \(H\)-action on \(Y\) induces an \(H\)-action on $\Cz(Y, \Kom)$ and \((\rho_Y^\prime, u_H\otimes\id_{\elltwo})\) becomes a covariant pair.
Because every automorphism of a \Cstar-algebra extends uniquely to an automorphism of its multiplier algebra, we get a covariant pair \((\bar\rho_Y, u_H \otimes \id_{\elltwo})\) for \((\Mult(\Cz(Y, \Kom), H)\).
Here \(\bar\rho_Y\) is precisely the same as in \eqref{eq_rep_of_Cb}.
We amplify the latter covariant pair via the left-regular representation \(\lambda_H \colon H \to \U(\elltwo(H))\) to obtain a covariant pair \((\hat\rho_Y, \hat{u}_H)\), where \(\hat\rho_Y \coloneqq \id_{\elltwo(H)} \otimes \bar\rho_Y\) and \(\hat{u}_H \coloneqq \lambda_H \otimes u_H \otimes \id_{\elltwo}\).

We consider \(\sHigComRed Y \subset \Lin(\elltwo(H) \otimes H_Y \otimes \elltwo)\) via the representation $\hat\rho_Y$. It follows from the above that the \(H\)-\Cstar-algebra \(\sHigComRed Y\) is covariantly represented on \( \elltwo(H) \otimes H_Y \otimes \elltwo\).
Moreover, by Fell's absorption principle \cite[Proposition~4.1.7]{BrownOzawaCstarFiniteDim} this yields an embedding of the reduced crossed product \(\hat\rho_Y \rtimes \hat{u}_H \colon \sHigComRed Y \rtimes_\red H \hookrightarrow \cB(\elltwo(H) \otimes H_Y \otimes \elltwo)\).

Let us redefine \(H_{X \times Y} \coloneqq H_X \otimes \elltwo(H) \otimes H_Y\).
The Hilbert space \(H_{X \times Y}\) has a unitary representation \(u_{G \times H}\) of \(G \times H\) via \(g,h \mapsto u_G(g) \otimes \lambda_H(h) \otimes u_H(h) \) and a representation \(\rho_{X \times Y}\) of \(\Cz(X \times Y)\) via \(f \tens f^\prime \mapsto  \rho_{X}(f) \otimes \id_{\elltwo(H)} \otimes \rho_Y(f^\prime)\).
This turns \({H}_{X \times Y}\) into an ample \((X \times Y)\)-\((G \times H)\)-module.
Similarly as in \labelcref{eq_varRepOfCzX}, we define \(\tilde{H}_X\coloneqq H_{X \times Y} \otimes \elltwo =H_X\otimes  \elltwo(H) \otimes H_Y\otimes \elltwo\) and
\[
  \tilde\rho_X\coloneqq \rho_X\otimes\id_{\elltwo(H) \otimes H_Y\otimes\elltwo}\colon \Cz(X)\to \Lin(\tilde H_X)\,.
\]
In addition, let \(\tilde\rho_Y \coloneqq \id_{H_X} \otimes \hat\rho_Y\) and \(\tilde{u}_{H} \coloneqq \id_{H_X} \otimes \hat{u}_H\).
Then we obtain  \(\tilde\rho_Y \rtimes \tilde{u}_H = \id_{H_X} \otimes (\hat\rho_Y \rtimes \hat{u}_H) \colon \sHigComRed Y \rtimes_\red H \hookrightarrow \cB(\tilde H_X)\).
Note that the definitions of \(\tilde\rho_Y\) and \(\tilde\rho_X\) are slightly different than in \cref{subsubsec_slant_Roe} because here we use \(\elltwo(H)\) as an additional tensor factor in \(\tilde{H}_X\).
This allows us to use the reduced crossed product.

Now the following equivariant versions of the lemmas from \cref{subsubsec_slant_Roe} hold.
Here we let \(\Fipro[G](\tilde\rho_X) \subset \cB(\tilde H_X)\) denote the \Cstar-algebra generated by all the \(G\)-equivariant operators of finite propagation.
Then by the same argument as for \cref{lem_RoeidealinFipro}, the \Cstar-algebra \(\Roe[G](\tilde\rho_X)\) is an ideal in \(\Fipro[G](\tilde\rho_X)\).

\begin{lem}\label{lem_equi_slant_Roe_cycles_idealizer}
The images of the two representations
\begin{align*}
\tau\colon\Roe[G \times H](\rho_{X\times Y})&\to \Lin(\tilde H_X) \text{ given by } S\mapsto S\otimes\id_{\elltwo}\,,
\\\tilde\rho_Y \rtimes \tilde{u}_H \colon \sHigComRed Y \rtimes_{\red} H &\to \cB(\tilde H_X) \text{ defined above}\,,
\end{align*}
are contained in \(\Fipro[G](\tilde\rho_X)\).
\end{lem}
\begin{proof}
  The operators considered here are clearly \(G\)-equivariant.
  Thus for the first part the argument is the same as for \cref{lem_slant_Roe_cycles_idealizer}.
  For the second part, we use in addition that the operators \(\tilde{u}_H(h)\) for \(h \in H\) commute with the image of \(\tilde\rho_X\) and hence have propagation zero.
\end{proof}

\begin{lem}\label{lem_equi_slant_Roe_cycles_commute}
The images of the $*$-homomorphisms \(\Roe[G \times H](\rho_{X\times Y})\to \Fipro[G](\tilde\rho_X)\) and $\sHigComRed Y \rtimes_{\red} H \to \Fipro[G](\tilde\rho_X) $ obtained from the previous lemma commute up to $\Roe[G](\tilde\rho_X)$.
\end{lem}
\begin{proof}
 The image of \(\tilde\rho_Y \rtimes \tilde{u}_H \colon \sHigComRed Y \rtimes_{\red} H \to \Fipro[G](\tilde\rho_X)\) has a dense subset consisting of linear combinations of products of operators of the form \(\tilde\rho_Y(f) = \id_{H_X} \otimes \id_{\elltwo(H)} \otimes \bar\rho_{Y}(f) \), where \(f \in \sHigComRed(f)\), with operators of the form \(\tilde{u}_H(h) = \id_{H_X} \otimes \lambda_H(h) \otimes u_H(h) \otimes \id_{\elltwo} \), where \(h \in H\). Thus it suffices to show that the commutators
 \begin{gather}
   [S \otimes \id_{\elltwo}, \id_{H_X} \otimes \id_{\elltwo(H)} \otimes \bar\rho_{Y}(f)], \label{eq_sHigCom_commutator}\\
   [S \otimes \id_{\elltwo}, \tilde{u}_H(h)]
   \label{eq_groupH_commutator}
 \end{gather}
 are contained in \(\Roe[G](\tilde\rho_X)\) for \(S \in \Roe[G \times H](\rho_{X \times Y})\), \(f \in \sHigComRed(Y)\) and \(h \in H\).
 We may also assume that \(S\) has finite propagation.
 Then the proof of \cref{lem_slant_Roe_cycles_commute} (applied to the \(Y\)-module \(\elltwo(H) \otimes H_Y\) with the representation \(\id_{\elltwo(H)} \otimes \rho_Y\)) shows that \labelcref{eq_sHigCom_commutator} is an element of \(\Roe(\tilde\rho_X)\) with finite propagation.
 It is also \(G\)-equivariant because \(S\) is.
 Hence \labelcref{eq_sHigCom_commutator} is an element of \(\Roe[G](\tilde\rho_X)\).
Finally, \(S \otimes \id_{\elltwo}\) is \(H\)-equivariant because \(S\) is.
That is, it commutes with \(\tilde{u}_H(h)\) and thus \labelcref{eq_groupH_commutator} is zero.
\end{proof}

Hence as an analogue of \labelcref{eq_starhom_Phi} we get an induced $*$-homomorphism
\[\Phi\colon\Roe[G \times H](\rho_{X\times Y})\otimes_{\max} (\sHigComRed Y \rtimes_{\red} H) \to \Fipro[G](\tilde\rho_X)/\Roe[G](\tilde\rho_X)\]
given by $S\otimes f \delta_h \mapsto [\tau(S) \circ \tilde\rho_Y(f) \circ \tilde{u}_H(h)]$.

A slight elaboration of the proof of \cref{lem_slant_Roe_cycles_factors} shows the following.

\begin{lem}
The above $*$-homomorphism $\Phi$ factors through the \Cstar-algebra 
\[\Roe[G \times H](\rho_{X \times Y}) \otimes_{\max} \frac{\sHigComRed Y \rtimes_{\red} H}{\Cz(Y, \Kom) \rtimes_{\red} H}.\]
That is, it defines a $*$-homomorphism
\begin{equation}
\Roe[G \times H](\rho_{X \times Y}) \otimes_{\max} \frac{\sHigComRed Y \rtimes_{\red} H}{\Cz(Y, \Kom) \rtimes_{\red} H} \to  \Fipro[G](\tilde\rho_X)/\Roe[G](\tilde\rho_X).
\label{eq_equiv_slant_product_red_quotient}
\end{equation}
\end{lem}

Unfortunately, in general we cannot use \(\sHigCorRed Y \rtimes_{\red} H\) in the above due to a potential failure of exactness in the sequence \(\Cz(Y, \Kom) \rtimes_\red H \to \sHigComRed(Y) \rtimes_\red H \to \sHigCorRed(Y) \rtimes_\red H\).
If \(H\) is an exact group, then this presents no issues.
More generally, we can remedy this by using an exact crossed functor \(\rtimes_\mu\) in the sense of \cite[Definition~3.1]{BaumGuentWillExactCrossed} instead of \(\rtimes_\red\).
Indeed, we deduce the next proposition immediately from \labelcref{eq_equiv_slant_product_red_quotient}.
\begin{prop}
  Let \(\rtimes_\mu\) be an exact crossed product functor.
  Then there is a homomorphism
  \begin{equation}
  \Psi_\mu \colon \Roe[G \times H](\rho_{X \times Y}) \otimes_{\max} (\sHigCorRed Y \rtimes_{\mu} H) \to  \Fipro[G](\tilde\rho_X)/\Roe[G](\tilde\rho_X)
  \end{equation}
  induced by $S\otimes [f] \delta_h \mapsto [\tau(S) \circ \tilde\rho_Y(f) \circ \tilde{u}_H(h)]$.
\end{prop}

\begin{defn}
  By the same construction as in \cref{defn_coarseslant}, we obtain the slant product
\begin{equation}
\label{equiv_equiv_slant_roe}
	\K_{p}(\Roe[G \times H](X \times Y)) \otimes \K_{1-q}(\sHigCorRed Y \rtimes_{\mu} H) \to \K_{p-q}(\Roe[G] X)
\end{equation}
for any exact crossed product functor \(\rtimes_\mu\). 
If \(H\) is an exact group, we also obtain \labelcref{equiv_equiv_slant_roe} for \(\mu = \red\).
\end{defn}

\subsubsection{The slant product on the equivariant localization algebras}
\label{sec_equiv_localization_algebras}
We define $\FiproLoc[G](\tilde{\rho}_X)$ as the $\Cstar$-subalgebra of $\Cb([1,\infty),\Fipro[G](\tilde{\rho}_X))$ generated by the bounded and uniformly continuous functions $S\colon [1,\infty) \to \Fipro[G](\tilde{\rho}_X)$ such that the propagation of $S(t)$ is finite for all $t \ge 1$ and tends to zero as $t \to \infty$.
Similarly we define $\FiproLocz[G](\tilde{\rho}_X)$ as the ideal in $\FiproLoc[G](\tilde{\rho}_X)$ consisting of all maps that vanish at~$1$.
Note that $\Loc[G](\tilde{\rho}_X)$ is an ideal in $\FiproLoc[G](\tilde{\rho}_X)$ and $\Locz[G](\tilde{\rho}_X)$ is even an ideal in all of the three $\FiproLoc[G](\tilde{\rho}_X)$, $\FiproLocz[G](\tilde{\rho}_X)$ and of course $\Loc[G](\tilde{\rho}_X)$.

\begin{lem}\label{lem_equiv_slant_Loc_technical}
The following analogue of \cref{lem_slant_Loc_technical} holds.
\begin{enumerate}
\item
The images of the two isometric $*$-homomorphisms
\[\tau_\LSym\colon\Loc[G \times H](\rho_{X\times Y})\to \Cb([1,\infty),\Lin(\tilde H_X))\,,\]
which is obtained by applying the functor $\Cb([1,\infty),-)$ to $\tau$, and
\[
\tilde\rho_{Y,H,\LSym}\colon \sHigComRed Y \rtimes_\red H \xrightarrow{\tilde\rho_Y \rtimes_\red \tilde{u}_H}\Lin(\tilde H_X)\xrightarrow[\text{as constant functions}]{\text{inclusion}} \Cb([1,\infty),\Lin(\tilde H_X))\,,
\]
are contained in $\FiproLoc[G](\tilde\rho_X)$.

\item The image of $\tau_\LSym$ commutes up to $\Loc[G](\tilde\rho_X)$ with the image of $\sHigComRed Y \rtimes_\red H$ under $\tilde\rho_{Y,H,\LSym}$
and the image of $\Locz(\rho_{X\times Y})$ under $\tau_\LSym$ commutes up to $\Locz(\tilde\rho_X)$ with the image of $\sHigComRed Y$ under $\tilde\rho_{Y,\LSym}$.
Hence they induce $*$-homomorphisms
\begin{align*}
\Phi_\LSym\colon\Loc[G \times H](\rho_{X\times Y})\otimes_{\max} \sHigComRed Y \rtimes_\red H &\to \FiproLoc[G](\tilde\rho_X)/\Loc[G](\tilde\rho_X) \\
\Phi_{\LSym,0}\colon\Locz[G \times H](\rho_{X\times Y})\otimes_{\max} \sHigComRed Y \rtimes_\red H &\to \FiproLoc[G](\tilde\rho_X)/\Locz[G](\tilde\rho_X)
\end{align*}
given by $S\otimes f\mapsto [\tau_\LSym(S)\circ\tilde\rho_{Y,H,\LSym}(f)]$ and the image of $\Phi_{\LSym,0}$ is even contained in $ \FiproLocz[G](\tilde\rho_X)/\Locz[G](\tilde\rho_X)$.

\item Let \(\mu\) be an exact crossed product functor.
Then the $*$-homomorphisms $\Phi_\LSym$ and $\Phi_{\LSym,0}$ factor through $\Loc(\rho_{X\times Y})\otimes_{\max}\sHigCorRed Y \rtimes_\mu H$ and $\Locz(\rho_{X\times Y})\otimes_{\max}\sHigCorRed Y \rtimes_\mu H$, respectively.
In other words, they define $*$-homomorphisms
\begin{align*}
\Psi_\LSym\colon\Loc[G \times H](\rho_{X \times Y}) \otimes_{\max} \sHigCorRed Y \rtimes_\mu H &\to  \FiproLoc[G](\tilde\rho_X)/\Loc[G](\tilde\rho_X)\,,
\\\Psi_{\LSym,0}\colon\Locz[G \times H](\rho_{X \times Y}) \otimes_{\max} \sHigCorRed Y \rtimes_\mu H &\to  \FiproLocz[G](\tilde\rho_X)/\Locz[G](\tilde\rho_X)\,.
\end{align*}
\end{enumerate}
\end{lem}

\begin{defn}\label{defn_equiv_localizationslant}
Let \(\mu\) be an exact crossed product functor.
By \cref{lem_equiv_slant_Loc_technical} we obtain the following slant products analogously as in \cref{defn_localizationslant}.
\begin{align*}
\K_p^{G \times H}(X \times Y) \otimes \K_{1-q}(\sHigCorRed Y \rtimes_\mu H) & \to \K_{p-q}^{G}(X),\\
\Strg_p^{G \times H}(X \times Y) \otimes \K_{1-q}(\sHigCorRed Y \rtimes_\mu H) & \to \Strg_{p-q}^G (X).
\end{align*}
If the group \(H\) is exact, we also obtain the above for \(\mu = \red\).
\end{defn}

\subsection{Properties of the equivariant slant products}\label{subsec_EquivSlantProp}
In this section we state the properties of the equivariant slant products which are analogous to those of the non-equivariant one.

The following compatibility of the slant products with the equivariant Higson--Roe sequence is proven in exactly the same way as its non-equivariant counterpart \cref{thm_commutativity_diagram} by decorating everything with the groups $G$ and~$H$.
\begin{thm}\label{equithm_commutativity_diagram}
The diagram
\begin{align*}
\xymatrix{
\Strg_p^{G\times H}(X \times Y) \ar[r] \ar[d]^{/ \theta} & \K_p^{G \times H}(X \times Y) \ar[r] \ar[d]^{/ \theta} & \K_p(\Roe[G \times H](X \times Y)) \ar[d]^{/ \theta} \ar[r]^-{\partial} & \Strg_{p-1}^{G \times H}(X \times Y) \ar[d]^{/ \theta}\\
\Strg_{p-q}^G(X) \ar[r] & \K_{p-q}^G(X) \ar[r] & \K_{p-q}(\Roe[G] X) \ar[r]^-{\partial} & \K_{p-1-q}^G(X)
}
\end{align*}
commutes for every $\theta\in \K_{1-q}(\sHigCorRed Y\rtimes_\mu H)$.\qed
\end{thm}

In \cref{sec_recovering_usual_slant_prod} we had shown that our non-equivariant slant product for the localization algebra factors through co-assembly and the usual slant product between $\K$-homology and $\K$-theory.
Now we shall show that our slant product for the equivariant localization algebra also factors through our version of co-assembly from \cref{label_sec_factorization_assembly_coassembly}.
The first part of \cref{sec_recovering_usual_slant_prod} goes through equivariantly.
That is, one obtains an isomorphism
\[
  \K_\ast(\Loc[G](X)) \cong E_\ast^G(\Cz(X), \C)
\]
and the latter group is, by definition, the equivariant \(\K\)-homology group \(\K^G_\ast(X)\) in the E-theory picture.
However, there is no already well known slant product between equivariant $\K$-homology and equivariant $\K$-theory defined via equivariant $E$-theory, as explained earlier.
Instead we turn an equivariant version of \cref{lem_alternativeKhomslant} into a definition.

Consider the map
\begin{align*}
\Upsilon_{\LSym} \colon \Loc[G\times H](\rho_{X\times Y}) \otimes \Cz(Y, \Kom)\rtimes H & \to \Loc[G](\tilde\rho_X) / \Cz([1,\infty),\Roe(\tilde\rho_X)),\\
(T_t)_{t \in [1,\infty)} \otimes f & \mapsto \left[ (\tau(T_t) \circ (\tilde\rho_Y\rtimes\tilde u_H)(f))_{t \in [1,\infty)} \right].
\end{align*}
By enriching the reasoning in \cref{sec_recovering_usual_slant_prod} with the ideas of the proofs of Lemmas \ref{lem_equi_slant_Roe_cycles_idealizer} and \ref{lem_equi_slant_Roe_cycles_commute} we get that $\Upsilon_{\LSym}$ is a well-defined $\ast$-homomorphism.

\begin{defn}\label{def_topolequislant}
The slant product between equivariant $\K$-homology and equivariant $\K$-theory is defined as the composition
\begin{align*}
	\K_p^{G\times H}(X \times Y) \otimes \K^{q}_H(Y)
    &\cong \K_p(\Loc[G\times H](\rho_{X \times Y})) \otimes \K_{-q}(\Cz(Y, \Kom)\rtimes H) \\
  &\xrightarrow{\boxtimes} \K_{p-q}(\Loc[G\times H](\rho_{X\times Y}) \otimes \Cz(Y, \Kom)\rtimes H) \\
	&\xrightarrow{({\Upsilon_{\LSym}})_\ast} \K_{p-q}( \Loc[G](\tilde\rho_X) / \Cz([1, \infty), \Roe[G](\tilde\rho_X))) \\
	&\xrightarrow{\cong} \K_{p-q}(\Loc[G](\tilde\rho_X)) \cong \K_{p-q}^G(X),
\end{align*}
where the third map is the inverse of the isomorphism induced on $\K$-theory by the canonical projection $\Loc[G](\tilde\rho_X) \to \Loc[G](\tilde\rho_X) / \Cz([1,\infty), \Roe[G](\tilde\rho_X))$.
\end{defn}

The proofs of \cref{prop:asymptoticCommuteBndGmtry}, \cref{cor:asymptoticCommutatorHigCom} and \cref{thm_KHomSlantCompatible} in \cref{sec_recovering_usual_slant_prod} now generalize to the equivariant case to yield the following equivariant version of \cref{thm_KHomSlantCompatible}.
\begin{thm}\label{thm_equiKHomSlantCompatible}
Let $Y$ have continuously bounded geometry.
The slant product on the localization algebra from \cref{defn_equiv_localizationslant} and the slant product from \cref{def_topolequislant} are related to each other via the co-assembly map
\[\mu^*_H\colon \K_*(\sHigCorRed Y \rtimes_{\mu} H)\to \K_H^{1-*}(Y)\]
in the sense that $x/\theta=x/\mu_H^*\theta$ for all $x\in\K^{G\times H}_*(X\times Y)$ and $\theta\in \K_*(\sHigCorRed Y \rtimes_{\mu} H)$.
\end{thm}

Next, \cref{thm_slantcrosscomp} and \cref{cor_slantcrosscomp} can also be generalized to the equivariant cases.
\begin{thm}
Let $G,H,K$ be countable discrete groups acting properly and isometrically on proper metric spaces $X,Y,Z$ and assume that $Z$ has bounded geometry.
Then the compositions
\[
\xymatrix@R-2pc{
\K^G_m(X)\ar[r]^-{\times z}&\K^{G\times H\times K}_{m+p}(X\times Y\times Z) \ar[r]^-{/\theta}&\K^{G\times H}_{m+p-q}(X\times Y)
\\\Strg^G_m(X)\ar[r]^-{\times z}&\Strg^{G\times H\times K}_{m+p}(X\times Y\times Z) \ar[r]^-{/\theta}&\Strg^{G\times H}_{m+p-q}(X\times Y)
\\\K_m(\Roe[G] X)\ar[r]^-{\times z}&\K_{m+p}(\Roe[G\times H\times K](X\times Y\times Z)) \ar[r]^-{/\theta}&\K_{m+p-q}(\Roe[G\times H](X\times Y))
}\]are equal to the external product with the appropriate slant product $z/\theta$ for all $m,p,q\in \Z$ and all $z,\theta$ as follows:
\begin{itemize}
\item In the first two compositions $z\in \K^{H\times K}_p(Y\times Z)$ and in the third one
$z\in \K_p(\Roe[H\times K] (Y\times Z))$. 
\item In the first composition either $\theta\in \K_K^q(Z)$ or $\theta\in \K_{1-q}(\sHigCorRed Z\rtimes_\mu K)$ and in the other two $\theta\in \K_{1-q}(\sHigCorRed Z\rtimes_\mu K)$. \qed
\end{itemize}
\end{thm}
\begin{proof}[Outline of proof]
The proof works almost word by word the same, just that one has to replace the Hilbert space $H_Z$ by $\ell^2(K)\otimes H_Z$ at several places within the constructions. The only difference is the case $\theta\in \K_K^q(Z)$, because we didn't define the equivariant version of this slant product via $E$-theory and hence it does not follow from abstract properties of $E$-theory. Instead one has to prove it along the lines of the other cases, but using the map $\Upsilon_{\LSym}$ instead of $\Psi,\Psi_{\LSym}$ or $\Psi_{\LSym,0}$.
\end{proof}

\begin{defn}\label{defn_equivariant_pairing}
The pairing
\[\langle-,-\rangle\colon \K_p(\Roe[H] Y) \otimes \K_{1-q}(\sHigCorRed Y\rtimes_\mu H)\to \K_{p-q}(\Roe\{*\})\cong \begin{cases}\Z&p-q\text{ even}\\0&p-q\text{ odd}\end{cases}\]
is defined as the special case of the equivariant slant product where $X=\{*\}$ is a single point equipped with the action of the trivial group.
The same construction applied to the localization algebra instead of the Roe algebra also yields a pairing
\[\langle-,-\rangle\colon \K^H_p(Y) \otimes \K_{1-q}(\sHigCorRed Y\rtimes_\mu H)\to \K_{p-q}(\{*\})\cong \begin{cases}\Z&p-q\text{ even}\\0&p-q\text{ odd}\end{cases}\]
and similarily we also have a pairing
\[
\langle-,-\rangle\colon \K^H_p(Y) \otimes \K_H^{q}(Y)\to \K_{p-q}(\{*\})\cong \begin{cases}\Z&p-q\text{ even}\\0&p-q\text{ odd.}\end{cases}\qedhere
\]
\end{defn}

\begin{cor}\label{cor_equivariant_composition_is_pairing}
The compositions
\[
\xymatrix@R-2pc{
\K^G_m(X) \ar[r]^-{\times z} & \K^{G\times H}_{m+p}(X\times Y) \ar[r]^-{/\theta} & \K^G_{m+p-q}(X)\\
\Strg^G_m(X) \ar[r]^-{\times z} & \Strg^{G\times H}_{m+p}(X\times Y) \ar[r]^-{/\theta} & \Strg^G_{m+p-q}(X)\\
\K_m(\Roe[G] X) \ar[r]^-{\times z} & \K_{m+p}(\Roe[G\times H](X\times Y)) \ar[r]^-{/\theta} & \K_{m+p-q}(\Roe[G] X)
}
\]
are equal to the multiplication with $\langle z,\theta\rangle$, which is either an integer, if $p-q$ is even, or zero by construction, if $p-q$ odd, for all $m,p,q\in \Z$ and all $z,\theta$ as follows:
\begin{itemize}
\item In the first two compositions $z\in \K^H_p(Y)$ and in the third 
$z\in \K_p(\Roe[H] Y)$. 
\item In the first composition either $\theta\in \K_H^q(Y)$ or $\theta\in \K_{1-q}(\sHigCorRed Y\rtimes_\mu H)$ and in the other two $\theta\in \K_{1-q}(\sHigCorRed Y\rtimes_\mu H)$.\qed
\end{itemize}
\end{cor}

Next on the list is naturality.

\begin{defn}
Let $X,X'$ be proper metric spaces equipped with proper isometric actions by the same countable discrete group $G$ and fix an ample $X$-$G$-module $(H_X,\rho_X,u_G)$ and an ample $X'$-$G$-module $(H_{X'},\rho_{X'},u'_G)$. An isometry $V\colon H_X\to H_{X'}$ is said to \emph{equivariantly cover} an equivariant coarse map $\alpha\colon X\to X'$ if it covers $\alpha$ in the sense of \cref{defn:coarseCoveringIsometry} and is in addition equivariant with respect to $u_G$ and $u'_G$. Similarily, a uniformly continuous family of isometries $V\colon[1,\infty)\to\Lin(H_X,H_{X'})$ is said to \emph{equivariantly cover} an equivariant uniformly continuous coarse map $\alpha\colon X\to X'$ if it covers $\alpha$ in the sense of \cref{def:equivcoveringiso} and is in addition equivariant.
\end{defn}

\begin{prop}[{see~\cite[Proposition 4.5.12, Theorems 5.2.6 and 6.6.3]{WillettYuHigherIndexTheory}}]
Equivariantly covering \parensup{uniformly continuous families of} isometries as in the previous definition always exist. Conjugation with an isometry $V$ which equivariantly covers a coarse map $\alpha$ yields a $*$-homomorphism
\[\Ad_V\colon \Roe[G](\rho_X)\to \Roe[G](\rho_{X'})\]
and conjugation with a uniformly continuous family of isometries $V$ which equivariantly covers a unifomly continuous equivariant coarse map $\alpha$ yields $*$-homomorphisms
\[\Ad_V\colon \Loc[G](\rho_X)\to \Loc[G](\rho_{X'})\,,\quad \Ad_V\colon \Locz[G](\rho_X)\to \Locz[G](\rho_{X'})\,.\]
The induced maps on $\K$-theory
\[\K_*(\Roe[G]X)\to \K_*(\Roe[G]X')\,,\quad\K_*^G(X)\to \K_*^G(X')\,,\quad\Strg_*^G(X)\to \Strg_*^G(X')\]
are independent of all choices, depend functorial on $\alpha$ and will all be denoted by $\alpha_*$.
Furthermore, they make the diagram
\[\mathclap{
\xymatrix{
\K_{\ast+1}(\Roe[G] X) \ar[r]^-{\partial} \ar[d]^-{\alpha_*} & \Strg_\ast^G(X) \ar[r] \ar[d]^-{\alpha_*} & \K_\ast^G(X) \ar[r]^-{\Ind} \ar[d]^-{\alpha_*} & \K_\ast(\Roe[G] X) \ar[d]^-{\alpha_*}\\
\K_{\ast+1}(\Roe[G] X') \ar[r]^-{\partial} & \Strg_{\ast}^G(X') \ar[r] & \K_{\ast}^G(X') \ar[r]^-{\Ind} & \K_{\ast}(\Roe[G] X')
}}\]
commute.
\end{prop}

Again, functoriality of $\K_*^G(-)$ can be extended to $G$-equivariant proper continuous maps and functoriality of $\Strg_*^G(-)$ can be extended to $G$-equivariant continuous coarse maps by using different pictures of these groups. This functoriality can be dealt with in exactly the same way as in the non-equivariant case as described in \cref{rem_better_functoriality_structgroup}.

On the other side, contravariant functoriality of $\K_*(\sHigCorRed(-)\rtimes H)$ under $H$-equivariant coarse maps and of $\K^*_H(-)$ under $H$-equivariant proper continuous maps is clear.

The equivariance can be implemented very easily into the proofs of \cref{thm_coarseslantfunctorial} and \cref{thm_LocSlantNatural} and the proof can also be adapted to the slant product between equivariant $\K$-homology and equivariant $\K$-theory defined in \cref{def_topolequislant} by using the map $\Upsilon_L$ instead of $\Psi_L$. One obtains the following.

\begin{thm}
All the equivariant slant products are natural in the sense that the formula
\[\alpha_*(x/\beta^*(\theta))=(\alpha\times \beta)_*(x)/\theta\]
holds in each of the following cases:
\begin{itemize}
\item $\alpha\colon X\to X'$ and $\beta\colon Y\to Y'$ are equivariant coarse maps, $x\in \K_*(\Roe[G\times H](X\times Y))$ and $\theta\in \K_*(\sHigCorRed Y'\rtimes_\mu H)$.
\item $\alpha\colon X\to X'$ is an equivariant proper continuous map and $\beta\colon Y\to Y'$ is a equivariant continuous coarse maps, $x\in \K_*^{G\times H}(X\times Y)$ and $\theta\in \K_*(\sHigCorRed Y'\rtimes_\mu H)$.
\item $\alpha\colon X\to X'$ and $\beta\colon Y\to Y'$ are equivariant continuous coarse maps, $x\in \Strg_*^{G\times H}(X\times Y)$ and $\theta\in \K_*(\sHigCorRed Y'\rtimes_\mu H)$.
\item $\alpha\colon X\to X'$ and $\beta\colon Y\to Y'$ are equivariant proper continuous maps, $x\in \K_*^{G\times H}(X\times Y)$ and $\theta\in \K_H^*(Y')$.\qed
\end{itemize}
\end{thm}

Finally, using equivariant functoriality we get the same coarsification results as in \cref{sec_compatibility_Rips}. More precisely, given any proper metric space $X$ equipped with a proper isometric action by a countable discrete group $G$, it is possible to choose the discrete coarsely equivalent subspace $X'\subset X$ to be $G$-invariant. Then each Rips complex $P_RX'$ inherits a canonical proper isometric $G$-action such that all the inclusion maps $P_RX'\to P_SX'$ for $R\leq S$ are equivariant continuous coarse equivalences.

Then we can define the equivariant coarse $\K$-homology and the equivariant coarse structure group as
\begin{align*}
\KX_*^G(X)&\coloneqq \varinjlim_{R\geq 0}\K_*^G(P_RX')
&\text{and}&
&\SX_*^G(X)&\coloneqq \varinjlim_{R\geq 0}\Strg_*^G(P_RX')
\end{align*}
and similarily there is also a equivariant coarse $\K$-theory which fits into a Milnor-$\varprojlim^1$-sequence
\[0\to{\varprojlim_{R\geq 0}}^1\K^{*+1}_G(P_RX')\to \KX^*_G(X)\to \varprojlim_{R\geq 0}\K^*_G(P_RX')\to 0\,.\]
All of the above groups are functorial under equivariant coarse maps in the obvious way and the Milnor-$\varprojlim^1$-sequence is natural.

Again, the construction does not yield anything new for the $\K_*(\Roe[G]X)$ and $\K_*(\sHigCorRed X\rtimes_\mu G)$. There is a natural equivariant coarsified version of the Higson--Roe sequence
\begin{equation*}
\cdots\to\K_{*+1}(\Roe[G] X)\to \SX_*^G(X)\to \KX_*^G(X)\xrightarrow{\mu_G} \K_*(\Roe[G] X)\to\cdots
\end{equation*}
and also a natural equivariant coarse co-assembly map
\[\mu^*_G\colon\K_{1-*}(\sHigCorRed X\rtimes_\mu G)\to \KX^*_G(X)\,.\]

Furthermore, the map $X\to P_{2R}X'$ defined in \eqref{eq_into_Rips} can also be assumed to be equivariant by using an equivariant partition of unity $\{\varphi_{x'}\}_{x'\in X}$, i.\,e.\ one for which $\varphi_{gx'}(gx)=\varphi_{x'}(x)$ for all $x\in X$, $x'\in X'$ and $g\in G$, and it is unique up to equivariant homotopy equivalence which stays close to the identity. Therefore we have canonical equivariant coarsification maps
\begin{align*}
\coarsify_G\colon \K_*^G(X)&\to \KX_*^G(X)
&\text{and}&
&\coarsify_G\colon\Strg_*^G(X)&\to\SX_*^G(X)
\end{align*}
and a canonical equivariant co-coarsification map
\[\coarsify^*_G\colon\KX_G^*(X)\to\K_G^*(X)\]
which are natural under equivariant continuous coarse maps and make the diagrams
\[\xymatrix@C=0ex{
\K_*^G(X)\ar[rr]^-{\mu_G}\ar[dr]_-{\coarsify_G}&&\K_*(\Roe[G] X)&\qquad&\K_{1-*}(\sHigCorRed X\rtimes_\mu G)\ar[rr]^-{\mu^*_G}\ar[dr]_-{\mu^*_G}&&\K^*_G(X)
\\&\KX_*^G(X)\ar[ur]_-{\mu_G}&&&&\KX^*_G(X)\ar[ur]_-{\coarsify^*_G}&
}\]
commute.

Inserting the equivariance into the constructions of \ref{subsec_equivcoarseifiedcrossslant} now gives rise to equivariant coarsified external and slant products
\begin{alignat*}{3}
\times\colon&& \KX_m^G(X)&\otimes \KX_n^H(Y)&&\to \KX_{m+n}^{G\times H}(X\times Y)
\\\times\colon&& \SX_m^G(X)&\otimes \KX_n^H(Y)&&\to \SX_{m+n}^{G\times H}(X\times Y)
\\/\colon&& \KX_p^{G\times H}(X\times Y)&\otimes \KX^q_H(Y)&&\to \KX_{p-q}^G(X)
\\/\colon&& \KX_p^{G\times H}(X\times Y)&\otimes \K_{1-q}(\sHigCorRed Y\rtimes_\mu H)&&\to \KX_{p-q}^G(X)
\\/\colon&& \SX_p^{G\times H}(X\times Y)&\otimes \K_{1-q}(\sHigCorRed Y\rtimes_\mu H)&&\to \SX_{p-q}^G(X)
\end{alignat*}
and pairings
\begin{alignat*}{3}
\langle-,-\rangle\colon&&\KX_p^H(Y)&\otimes\KX^q_H(Y)&\to&\K_{p-q}(\C)
\\\langle-,-\rangle\colon&&\KX_p^H(Y)&\otimes\K_{1-q}(\sHigCorRed Y\rtimes_\mu H)&\to&\K_{p-q}(\C)
\end{alignat*}
which are natural under pairs of equivariant coarse maps, compatible with the maps in the equivariant coarsified Higson--Roe sequence, the equivariant coarsification and co-coarsification maps and co-assembly.

Furthermore, if $Z$ is a third proper metric space of bounded geometry with a proper isometric action by a countable discrete group $K$, then taking first the external product with an element $z\in \KX_m^{H\times K}(Y\times Z)$ and then the slant product with an element $\theta\in K_{1-n}(\sHigCorRed Z\rtimes_\mu K)$ or $\theta\in \KX^n_K(Z)$ is equal to the external product with $z/\theta\in \KX_{m-n}^H(Y)$, and in particular if $Y=\{*\}$ is a one-point space trivially acted on by $H=1$, then this composition is equal to multiplication with $\langle z,\theta\rangle = \langle z,\mu^*(\theta)\rangle=\langle \mu(z),\theta\rangle$, if \(Z\) has continuously bounded geometry.

\subsection{Compatibility with the non-equivariant version}
\label{subsec_compatibility_equiv_nonequiv}
In this section, we describe two ways in which the equivariant slant products from \cref{sec_equiv_version} are compatible with the construction in \cref{sec_slant_products}.

To formulate the first type of compatibility concisely, we use the symbol \(\RoePlaceholder\) as a placeholder to denote either \(\Roe\), \(\Loc\) or \(\Locz\).
We consider the canonical forgetful map 
\[ \ForgetEquiv \colon \RoePlaceholder[G] X \to \RoePlaceholder X \]
which is the inclusion map if \(\RoePlaceholder[G] X\) is defined on a given \((X, G)\)-module and \(\RoePlaceholder X\) is defined on the underlying \(X\)-module.
On the level of \(\K\)-theory, this induces a map
\[
  \ForgetEquiv_\ast \colon \HRPlaceholder_\ast^G(X) \to \HRPlaceholder_\ast(X),
\]
where we use the notation from \cref{rem_HRPlaceholder}.

Recall that in \cref{sec_slant_products} we have constructed slant products
\[
\\/ \colon \HRPlaceholder_p(X \times Y) \otimes \K_{1-q}(\sHigCorRed Y) \to \HRPlaceholder_{p-q}(X)
\]
and from \cref{sec_equiv_version} we have
\[
	\\/ \colon \HRPlaceholder_p^{G \times H}(X \times Y) \otimes \K_{1-q}(\sHigCorRed Y \rtimes_{\mu} H) \to \HRPlaceholder_{p-q}^G(X).
\]
Moreover, there is the \(\ast\)-homomorphism
\[
\iota \colon \sHigCorRed Y \to \sHigCorRed Y \rtimes_{\mu} H\,,\quad f \mapsto f \delta_e\,,
\]
which is well-defined because $H$ is discrete.
Then we have the following observation:
\begin{prop}\label{lem_compatibility_equiv_nonequiv}
Let $x \in \HRPlaceholder_{p}^{G \times H}(X \times Y)$ and $y \in \K_{1-q}(\sHigCorRed Y)$.
Then
\[
\ForgetEquiv_\ast(x / \iota_\ast(y)) = \ForgetEquiv_\ast(x) / y \in \HRPlaceholder_{p-q}(X).
\]
\end{prop}
\begin{proof}
Follows directly by comparing definitions.
\end{proof}

The second and more intricate type of compatibility deals with the case of a free action.
Indeed, if  $G$ acts freely on $X$, then there are canonical induction isomorphisms for \(\K\)-theory
\begin{equation}
  \mathsf{Z}_X \colon \K^\ast_G(X) \xrightarrow{\cong} \K^\ast(G\backslash X),
  \label{eq:KTheoryIndIso}
\end{equation}
and \(\K\)-homology
\begin{equation}
  \mathsf{I}_X \colon \K_\ast^G(X) \xrightarrow{\cong} \K_{\ast}(G \backslash X).
  \label{eq:LocIndIso}
\end{equation}

The isomorphism \labelcref{eq:KTheoryIndIso} follows from a canonical Morita equivalence given by Green's imprimitivity theorem, see~\cite[Corollary~4.11]{WilliamsCrossedProducts} for a textbook reference.
For \labelcref{eq:LocIndIso} see \cite[Theorem~6.5.15]{WillettYuHigherIndexTheory}.
The main result of this subsection is the following comparison theorem.
\begin{thm}\label{thm_equivariantComparison}
    Let \(X\) and \(Y\) be proper metric spaces which are endowed with proper and free actions of countable discrete groups \(G\) and \(H\), respectively.
  Then the equivariant slant product on \(\K\)-homology agrees with the usual slant product on the corresponding quotient spaces up to the isomorphisms \labelcref{eq:KTheoryIndIso,eq:LocIndIso}.
  In other words, the following diagram commutes.
  \[
  \begin{tikzcd}
      \K_p^{G \times H}(X \times Y) \otimes \K^q_H(Y) \ar[r, "\\/"] \ar[d, "\mathsf{I}_{X \times Y} \otimes \mathsf{Z}_Y"]
        & \K_{p-q}^G(X) \ar[d, "\mathsf{I}_X"] \\
      \K_p(G \backslash X \times H \backslash Y) \otimes \K^q(H \backslash Y) \ar[r, "\\/"]
        & \K_{p-q}(G \backslash X)
  \end{tikzcd}
  \]
\end{thm}

In view of \cref{lem_alternativeKhomslant} and our definitions of equivariant \(\K\)-homology and \(\K\)-theory, \cref{thm_equivariantComparison} is equivalent to the following technical proposition.

\begin{prop}\label{thm_equivariantComparison_technical}
  Suppose we are in the setup of \cref{thm_equivariantComparison}.
  Then the equivariant slant product from \cref{def_topolequislant} agrees with the description of the slant product in \cref{lem_alternativeKhomslant} up to the isomorphisms \labelcref{eq:KTheoryIndIso,eq:LocIndIso}.
  In other words, the following diagram commutes.
  \[
  \begin{tikzcd}
      \K_p(\Loc[G \times H](X \times Y)) \otimes \K_{-q}(\Cz(Y) \rtimes H) \ar[r, "\text{\labelcref{def_topolequislant}}"] \ar[d, "\mathsf{I}_{X \times Y} \otimes \mathsf{Z}_Y"]
        & \K_{p-q}(\Loc[G]X) \ar[d, "\mathsf{I}_X"] \\
      \K_p(\Loc(G \backslash X \times H \backslash Y)) \otimes \K_{-q}(\Cz(H \backslash Y)) \ar[r, "\text{\labelcref{lem_alternativeKhomslant}}"]
        & \K_{p-q}(\Loc(G \backslash X))
  \end{tikzcd}
  \]
\end{prop}
To prove \cref{thm_equivariantComparison_technical}, we need an explicit description of the induction isomorphisms \labelcref{eq:KTheoryIndIso,eq:LocIndIso}.
Start with an explicit Morita equivalence which implements \labelcref{eq:KTheoryIndIso}.
We exhibit the construction from \cite[Corollary~4.11]{WilliamsCrossedProducts} for convenience of the reader.
The expression
\[
  \langle f \mid f^\prime \rangle_{\Cz(G \backslash X)}(G x) \coloneqq \sum_{g \in G} \overline{f(g^{-1} y)} f^\prime(g^{-1} y),\quad f, f^\prime \in \Cc(X),
\]
yields a (right) inner product on \(\Cc(X)\) with values in \(\Cc(G \backslash X)\).
Moreover, \(\Cc(G \backslash X)\) acts on \(\Cc(X)\) from the right by
\[
  (f \cdot g)(x) = f(x) g(G x),\quad f \in \Cc(X), g \in \Cc(G \backslash X).
\]
Similarly, there is a left inner product on \(\Cc(X)\) with values in \(\Cc(X)\rtimes_{\mathrm{alg}} G\) defined by
\[
  \prescript{}{\Cz(X)\rtimes G}{\langle} f^\prime \mid f \rangle \coloneqq \sum_{g \in G} f^\prime \ \overline{g \cdot f}\ \delta_g,\quad f^\prime, f  \in \Cc(X).
\]
and the crossed product \(\Cc(X)\rtimes_{\mathrm{alg}} G\) acts from the left on \(\Cc(X)\) by
\[
  \left( \sum_{g \in G} f_g\ \delta_g \right) \cdot f = \sum_{g \in G} f_g (g \cdot f). 
\]
The right action of \(\Cc(G \backslash X)\) commutes with the left action of \(\Cc(X) \rtimes_{\mathrm{alg}} G\).
Simultaneously completing \(\Cc(X)\) and the coefficient algebras, we obtain an \(\Cz(X) \rtimes G\)-\(\Cz(G \backslash X)\)-Morita equivalence bimodule which we denote by \(\mathsf{Z}_X\).
The isomorphism \labelcref{eq:KTheoryIndIso} is given by the \(\ast\)-isomorphism
\[
  \phi \colon \Cz(X) \rtimes G \xrightarrow{\cong} \Kom_{\Cz(G \backslash X)}(\mathsf{Z}_X)
\]
that is induced by the left action.

We denote the conjugate module of \(\mathsf{Z}_X\) by \(\bar{\mathsf{Z}}_X\).
Then we have canonical identifications
\begin{equation}
\mathsf{Z}_Y \otimes_{\Cz(G \backslash X)} \bar{\mathsf{Z}}_X \cong \Kom_{\Cz(G \backslash X)}(\mathsf{Z}_X) \cong \Cz(X) \rtimes G, \label{eq:MoritaInverse1}
\end{equation}
see \cite[Corollary~8.2.15]{BlecherLeMerdyOperatorAlgebras} for the first isomorphism.
Similarly, 
\begin{equation}
  \bar{\mathsf{Z}}_X \otimes_{\Cz(X) \rtimes G} \mathsf{Z}_X \cong \Kom_{\Cz(X) \rtimes G}(\bar{\mathsf{Z}}_X) \cong \Cz(G \backslash X). \label{eq:MoritaInverse2}
\end{equation}

Next, we describe the isomorphism \(\labelcref{eq:LocIndIso}\) in a way that is compatible with the above Morita equivalence.
If we start with an ample \(G \backslash X\)-module \(H_{G \backslash X}\), then \(\mathsf{Z}_X \otimes_{\Cz(G \backslash X)} H_{G \backslash X}\) is an ample \((X,G)\)-module.\footnote{The proof that \(H_X\) constructed in this way is ample can be reduced to the trivial product situation \(X = G \times \underbar{X}\), compare the proof of \cref{lem_lifts}.}
Similarly, if \(H_X\) is an ample \((X, G)\)-module, then \(\bar{\mathsf{Z}}_X \otimes_{\Cz(X) \rtimes G} H_X\) is an ample \(G \backslash X\)-module.
Moreover, by \labelcref{eq:MoritaInverse1,eq:MoritaInverse2} these constructions are mutually inverse up to canonical isomorphisms.
Hence we will assume that the representations on \(H_X\) and \(H_{G \backslash X}\) have been chosen in such a way that we can identify \(H_X = \mathsf{Z}_X \otimes_{\Cz(G \backslash X)} H_{G \backslash X}\) and thus \(H_{G \backslash X} = \bar{\mathsf{Z}}_X \otimes_{\Cz(X) \rtimes G} H_X\).
We introduce the following auxilliary concept to describe \labelcref{eq:LocIndIso}.
\begin{defn}\label{defn_LocIndIso}
  Let \((S_t) \in \Loc(G \backslash X)\) and \((T_t) \in \Loc[G](X)\).
  We say that \((T_t)\) is a \emph{lift} of \((S_t)\) \parensup{or, alternatively, \((S_t)\) is \emph{pushdown} of \((T_t)\)} if for each \(z \in \mathsf{Z}_X\) we have 
  \begin{align*}
    (\mathfrak{T}_z \circ S_t - T_t \circ \mathfrak{T}_z)_{t \in [1, \infty)} &\in \Cz([1, \infty), \Lin(H_{G \backslash X}, H_X)), \\
    (S_t \circ \mathfrak{T}_z^\ast - \mathfrak{T}_z^\ast \circ T_t)_{t \in [1, \infty)}  &\in \Cz([1, \infty), \Lin(H_X, H_{G \backslash X})),
  \end{align*}
  where
  \begin{align*}
    \mathfrak{T}_z \colon H_{G \backslash X} &\to H_X = \mathsf{Z}_X \otimes_{\Cz(G \backslash X)} H_{G \backslash X},
    \quad \xi \mapsto z \otimes \xi. \qedhere
  \end{align*}
\end{defn}
\begin{rem}[Notation]
  To work with the defining conditions of lifts and pushdowns more conveniently, we will use the following notation:
  For two families of linear operators \((T_t)_{t \in [1, \infty)}\) and \((\tilde{T}_t)_{t \in [1, \infty)}\) with the same domain and target space, we write \(T_t \sim \tilde{T}_t\) if \(T_t - \tilde{T}_t \to 0\) in operator norm.
  Then the conditions from \cref{defn_LocIndIso} are equivalent to \(\mathfrak{T}_z \circ S_t \sim T_t \circ \mathfrak{T}_z\) and \(S_t \circ \mathfrak{T}_z^\ast \sim \mathfrak{T}_z^\ast \circ T_t\).
\end{rem}
In our constructions below, we will use the following ideal in the localization algebra.
\[\mathclap{
\LocVanish[G](X) \coloneqq \{ (T_t) \in \Loc[G](X) \mid \forall f \in \Cz(X) \colon  \lim_{t \to \infty} T_t \rho_X(f) = \lim_{t \to \infty} \rho_X(f) T_t = 0\}
}\]
An Eilenberg swindle as in \cite[Lemma~6.4.11]{WillettYuHigherIndexTheory} shows that the \(\K\)-theory of \(\LocVanish[G](X)\) vanishes.
Hence the canonical projection \(\Loc[G](X) \to \Loc[G](X)/\LocVanish[G](X)\) induces an isomorphism on \(\K\)-theory.
\begin{lem}\ \label{lem_lifts}
  \begin{myenumi}
    \item Every \((S_t) \in \Loc(G \backslash X)\) admits a lift and every \((T_t) \in \Loc[G](X)\) admits a pushdown.\label{item:LiftExists}
    \item Let \((T_t) \in \Loc[G](X)\) be a lift of \((S_t) \in \Loc(G \backslash X)\).
    Then we have \((T_t) \in \LocVanish[G](X)\) if and only if \((S_t) \in \LocVanish(G \backslash X)\).\label{item:LiftUnique}
  \end{myenumi}
\end{lem}
\begin{proof}
\labelcref{item:LiftExists}
The argument here is essentially the same construction as in \cite[Construction~6.5.14]{WillettYuHigherIndexTheory}. Let \(U \subseteq G \backslash X\) be an open subset over which the canonical projection \(\pi \colon X \to G \backslash X\) is trivialized, that is, \(\pi^{-1}(U) \cong G \times U\).
Let \(\mathsf{Z}_{\pi^{-1}(U)}\) be the closure of \(\Cz(\pi^{-1}(U)) \cdot \mathsf{Z}_X\).
Note that this is at the same time the closure of \(\mathsf{Z}_X \cdot \Cz(U)\) and this module implements the Morita equivalence from \(\Cz(\pi^{-1}(U)) \rtimes G\) to  \(\Cz(U)\).
Since \(\pi^{-1}(U) \cong G \times U\), we have
\begin{equation}
  \mathsf{Z}_{\pi^{-1}(U)} \cong \ell^2(G) \otimes \Cz(U)
  \label{eq:trivialZX}
\end{equation}
and 
\begin{equation} 1_{\pi^{-1}(U)} H_X = \mathsf{Z}_{\pi^{-1}(U)} \otimes_{\Cz(U)} 1_U H_{G \backslash X} \cong \ell^2(G) \otimes 1_U H_{G \backslash X}  \label{eq:trivialHX}
\end{equation}
 
Suppose for the moment that \(S_t\) has support inside \(U\).
That is, there exists a continuous function \(\chi \colon G \backslash X \to [0,1]\) with \(\supp(\chi) \subseteq U\) such that \(\chi S_t = S_t \chi = S_t\).
Then set \(T_t \coloneqq 1_{\pi^{-1}(U)} (\id_{\ell^2(G)} \otimes S_t) 1_{\pi^{-1}(U)} \in \Loc[G](X)\), where we implicitly use \labelcref{eq:trivialHX}.
We claim that \(T_t\) is a lift of \(S_t\).
Let \(z \in \Cc(X) \subseteq \mathsf{Z}_X\).
Then, by construction, \(T_t \circ \mathfrak{T}_{z} = T_t \circ \mathfrak{T}_{(\chi \circ \pi) z}\), \(\mathfrak{T}_{z}^\ast \circ T_t= \mathfrak{T}_{(\chi \circ \pi) z}^\ast \circ T_t \), \(\mathfrak{T}_z \circ S_t = \mathfrak{T}_{(\chi \circ \pi) z} \circ S_t\) and \(S_t \circ \mathfrak{T}_z^\ast = S_t \circ \mathfrak{T}_{(\chi \circ \pi) z}^\ast\).
Hence it suffices to check the defining condition of a lift for \(z \in \Cc(\pi^{-1}(U)) \subseteq \mathsf{Z}_{\pi^{-1}(U)} \subseteq \mathsf{Z}_X\).
In view of \labelcref{eq:trivialZX}, we are then further reduced to checking it for elements of the form \(z = \delta_g \otimes \varphi\), where \(g \in G\), \(\varphi \in \Cc(U)\).
Then 
\[
  \mathfrak{T}_{z} \circ S_t - T_t \circ \mathfrak{T}_z = (\xi \mapsto \delta_g \otimes \underbrace{[\varphi, S_t]}_{\to 0} \xi) \to 0
\]
considered as a map \(H_{G \backslash X} \to \ell^2(G) \otimes 1_U H_{G \backslash X} \cong 1_{\pi^{-1}(U)}H_X \subseteq H_X\).
Similarly,
\[
  S_t \circ \mathfrak{T}_{z}^\ast - \mathfrak{T}_z^\ast \circ T_t = (\delta_\gamma \otimes \xi \mapsto \langle \delta_g|\delta_\gamma \rangle \underbrace{[S_t, \varphi]}_{\to 0} \xi) \circ 1_{\pi^{-1}(U)} \to 0,
\]
where the right-hand side is viewed as a composition 
\[H_X \xrightarrow{1_{\pi^{-1}(U)}} 1_{\pi^{-1}(U)}H_X \cong \ell^2(G) \otimes 1_U H_{G \backslash X} \to H_{G \backslash X}.
\]
This shows that \((T_t)\) is a lift of \((S_t)\).

In general, we take a locally finite covering \(\mathcal{U}\) of \(G \backslash X\) such that \(\pi \colon X \to G \backslash X\) is trivial over each \(U \in \mathcal{U}\).
Then choose an \(\ell^2\)-partition\footnote{That means that \((\phi_U^2)\) is a partition of unity.} of unity \((\phi_U)_{U \in \mathcal{U}}\) subordinate to \(\mathcal{U}\).
Then for each \(U \in \mathcal{U}\), the element \((\phi_U S_t \phi_U)\) is supported inside \(U\).
Take the corresponding lift which was constructed in the previous paragraph and denote it by \((T_t^U) \in \Loc[G](X)\).
Then define \(T_t \coloneqq \sum_{U \in \mathcal{U}} T_t^{U}\).
This is a strongly converging sum since the covering \(\mathcal{U}\) is locally finite and one can easily verify that \((T_t) \in \Loc[G](X)\).
We again claim that \((T_t)\) is the desired lift of \((S_t)\).
To prove this, now take \(z \in \mathsf{Z}_X\), \(\varphi \in \Cc(G \backslash X)\).
Then 
\[
  \mathfrak{T}_{z \varphi} \sum_{U \in \mathcal{U}} \phi_U S_t \phi_U = \mathfrak{T}_{z} \sum_{U \in \mathcal{U}} \varphi \phi_U S_t \phi_U \sim \mathfrak{T}_z  \sum_{U \in \mathcal{U}} \varphi \phi_U^2 S_t = \mathfrak{T}_{z} \varphi S_t = \mathfrak{T}_{z \varphi} S_t,
  \]
  where we use that \([S_t, \phi_U] \to 0\) and the fact that the sums in the middle terms have only finitely many non-zero entries because the covering \(\mathcal{U}\) is locally finite.
  Thus 
  \[
    \mathfrak{T}_{z \varphi} \circ S_t \sim \sum_{U \in \mathcal{U}} \mathfrak{T}_{z \varphi} \circ \phi_U S_t \phi_U \sim \sum_{U \in \mathcal{U}} T^U_{t} \circ \mathfrak{T}_{z \varphi} = T_t \circ \mathfrak{T}_{z \varphi},
  \]
  where we again have used that the sums have only finitely many non-zero entries.
    Since elements of the form \(z \varphi\) are dense in \(\mathsf{Z}_X\), we actually obtain
  \[
    \mathfrak{T}_{z} \circ S_t \sim T_t \circ \mathfrak{T}_{z}
  \]
  for all \(z \in \mathsf{Z}_X\).
  An analogous argument shows that \(S_t \circ \mathfrak{T}_z^\ast \sim \mathfrak{T}^\ast_z \circ T_t\).
  This completes the proof that \((T_t)\) is a lift of \((S_t)\).
  
  The existence of pushdowns is proved analogously by first solving the problem for operators which are supported inside \(\pi^{-1}(U) \cong G \times U\).
  Indeed, if an operator on \(\ell^2(G) \otimes 1_{U} H_{G \backslash X} \subseteq H_X\) is \(G\)-equivariant and has sufficiently small propagation in \(X\), then it is necessarily of the form \(\id \otimes S\), and hence has a pushdown.
  The general case can again be reduced to this via an \(\ell^2\)-partition of unity.
  
  \labelcref{item:LiftUnique}
  Suppose that \((S_t) \in \LocVanish(G \backslash X)\) and \((T_t) \in \Loc[G](X)\) is a lift of \((S_t)\).
  Let \(z_1,z_2 \in \mathsf{Z}_X\) and \(\varphi \in \Cz(G \backslash X)\).
  Then with \(f \coloneqq\prescript{}{\Cz(X)\rtimes G}{\langle}z_1 \varphi \mid z_2 \rangle \in \Cz(X) \rtimes G\), we have
  \[
   f T_t = \mathfrak{T}_{z_1 \varphi} \circ \mathfrak{T}_{z_2}^\ast \circ T_t 
   \sim \mathfrak{T}_{z_1 \varphi} \circ S_t \circ \mathfrak{T}_{z_2}^\ast = \mathfrak{T}_{z_1} \circ \underbrace{\varphi S_t}_{\to 0} \circ \mathfrak{T}_{z_2}^\ast \to 0.
  \]
  and 
  \[
   T_t f =  T_t \circ \mathfrak{T}_{z_1 \varphi} \circ \mathfrak{T}_{z_2}^\ast 
   \sim \mathfrak{T}_{z_1 \varphi} \circ S_t  \circ \mathfrak{T}_{z_2}^\ast 
   = \mathfrak{T}_{z_1} \circ \underbrace{\varphi S_t}_{\to 0}  \circ \mathfrak{T}_{z_2}^\ast 
\to 0.\]
Since the left-module structure on \(\mathsf{Z}_Y\) is full, this is enough to check that \((T_t) \in \LocVanish[G](X)\).
  The converse implication is proved analogously.
  \end{proof}
The following proposition is an immediate consequence of \cref{lem_lifts}.
\begin{prop}
  Choosing pushdowns induces a well-defined \(\ast\)-isomorphism
  \[
      \mathsf{I}_X \colon \frac{\Loc[G](X)}{\LocVanish[G](X)} \xrightarrow{\cong} \frac{\Loc(G \backslash X)}{\LocVanish(G \backslash X)},
  \]
  with inverse given by choosing lifts.
  Together with the canonical isomorphisms \(\K_\ast(\Loc[\blank](\blank) \to \K_\ast(\Loc[\blank](\blank)/\LocVanish[\blank](\blank))\), this yields the isomorphism \labelcref{eq:LocIndIso}.
\end{prop}

Since the construction of lifts and pushdowns in \cref{lem_lifts} is the same as in \cite[Construction~6.5.14]{WillettYuHigherIndexTheory}, the isomorphism \(\mathsf{I}_X\) is indeed the same as the one given by \cite[Theorem~6.5.15]{WillettYuHigherIndexTheory}.

To compose our slant product with the Morita equivalence, we need a slightly more general version of the construction from \cref{def_topolequislant}.
To that end, let us fix a—for the moment arbitrary—countably generated Hilbert \(\Cz(Y)\rtimes H\)-module \(E\).
We will construct a \(\ast\)-homomorphism
\begin{equation}
  \Upsilon_{\LSym} \colon \Loc[G\times H](X \times Y) \otimes_{\max} \Kom_{\Cz(Y)\rtimes H}(E)
    \to \frac{\Loc[G](\tilde{\rho}_X)}{\Cz([1,\infty), \Roe[G](\tilde{\rho}_X))},
    \label{eq:generalTopEquSlantProduct}
\end{equation}
where we redefine \(\tilde{H}_X \coloneqq H_X \otimes (E \otimes_{\Cz(Y) \rtimes H} H_Y)\) and 
\[
  \tilde{\rho}_X \rtimes \tilde{u}_G \coloneqq (\rho_X \rtimes u_G)\otimes \id \otimes \id \colon \Cz(X) \rtimes G \to \Lin(\tilde{H}_X)\,.
\]
Note that if \(E = (\Cz(Y)\rtimes H) \otimes \ell^2\) and we also replace \(H_Y\) by \(\ell^2(H) \otimes H_Y\), then this reduces to our previous definitions.
There is a canonical \(\ast\)-homomorphism
\[
  \kappa_E \colon \Kom_{\Cz(Y) \rtimes H}(E) \to \FiproLoc[G](\tilde{\rho}_X),
  \quad
  K \mapsto (\id_{H_X} \otimes K \otimes_{\Cz(Y) \rtimes H} \id_{H_Y}),
\]
viewed as constant functions.
Unfortunately, in general there is no canonical map \(\Loc[G \times H](X \times Y) \to \FiproLoc[G](\tilde{\rho}_X)\) which would be required to precisely mimic \cref{def_topolequislant}.
Instead, we use a similar scheme as in the construction of the induction isomorphisms.
The next definition, along with its name, is inspired by the notion of \enquote{connection} which appears in the construction of the Kasparov product (see for instance \cite[Section~18.3]{blackadar}).
\begin{defn}\label{defn_connection}
  Let \((T_t) \in \Loc[G\times H](X \times Y)\). 
  We say \((F_t) \in \FiproLoc[G](\tilde{\rho}_X)\) is a \emph{connection} for \((T_t)\), if for each \(K \in \Kom_{\Cz(X) \rtimes H}(E)\), we have 
  \begin{align*}
    (F_t \circ \kappa_E(K))_{t \in [1,\infty)} &\in \Loc[G](\tilde{\rho}_X), \\
    (\kappa_E(K) \circ F_t)_{t \in [1,\infty)} &\in \Loc[G](\tilde{\rho}_X), 
  \end{align*}
  and for each \(e \in E\) we have
  \begin{align*}
    (\mathfrak{T}_e \circ T_t - F_t \circ \mathfrak{T}_e)_{t \in [1,\infty)} &\in \Cz([1,\infty), \Lin(H_X \otimes H_Y, \tilde{H}_X)), \\
    (T_t \circ \mathfrak{T}_e^\ast - \mathfrak{T}_e^\ast \circ F_t)_{t \in [1,\infty)}  &\in \Cz([1,\infty), \Lin(\tilde{H}_X, H_X \otimes H_Y)),
  \end{align*}
  where
  \begin{align*}
    \mathfrak{T}_e \colon H_{X} \otimes H_Y &\to \tilde{H}_X \cong E \otimes_{\id \otimes (\Cz(Y) \rtimes H)} (H_X \otimes H_Y),
    \quad \xi \mapsto e \otimes \xi.
    \qedhere
  \end{align*}
\end{defn}

\begin{lem}\ \label{lem_connections}
  \begin{myenumi}
    \item Every \((T_t) \in \Loc[G\times H](X \times Y)\) admits a connection \((F_t) \in \FiproLoc[G](\tilde{\rho}_X)\).\label{item:connectionExists}
    \item If \((F_t) \in \FiproLoc[G](\tilde{\rho}_X)\) is a connection for \((T_t) \in \Loc[G\times H](X \times Y)\) and \(K \in \Kom_{\Cz(Y)\rtimes H}(E)\), then \([F_t, \kappa_E(K)] \in \Cz([1, \infty), \Roe[G](\tilde{\rho}_X)) \). \label{item:connectionCommute}
    \item \label{item:connectionWellDef} Let \((F_t) \in \FiproLoc[G](\tilde{\rho}_X)\) be a connection for \((T_t) \in \Loc[G\times H](X \times Y)\) and \(K \in \Kom_{\Cz(Y)\rtimes H}(E)\).
    \begin{itemize}
      \item If \((T_t) \in \Cz([1,\infty), \Roe[G\times H](X \times Y))\), then \[(F_t \circ \kappa_E(K)) \in \Cz([1,\infty), \Roe[G](\tilde{\rho}_X)).\]
      \item If \((T_t) \in \LocVanish[G \times H](X \times Y)\), then \((F_t \circ \kappa_E(K)) \in \LocVanish[G](\tilde{\rho}_X)\).
    \end{itemize}
  \end{myenumi}
\end{lem}
\begin{proof}
  \labelcref{item:connectionExists}
  We start with the case that \(E =(\Cz(Y) \rtimes H) \otimes \ell^2\) is the standard Hilbert module over \(\Cz(Y) \rtimes H\).
  In this case, \(\tilde{H}_X = H_X \otimes H_Y \otimes \ell^2\) and a connection for \((T_t) \in \Loc[G \times H](X \times Y)\) is given by \(F_t \coloneqq T_t \otimes \id_{\ell^2}\).
  This is because for \(e = f \otimes v \in E =(\Cz(Y) \rtimes H) \otimes \ell^2\), we have 
  \[
    \mathfrak{T}_{e} \circ T_t - F_t \circ \mathfrak{T}_e = (\xi \mapsto \underbrace{[(\id \otimes f), T_t]}_{\to 0} \xi \otimes v) \to 0
  \]
  as an operator \(H_X \otimes H_Y \to H_X \otimes H_Y \otimes \ell^2\) and, similarly, \(T_t \circ \mathfrak{T}_e^\ast - \mathfrak{T}_e^\ast \circ F_t \to 0\).
  Moreover, if \(K = f \otimes L \in \Kom_{\Cz(Y) \rtimes H}(E) \cong (\Cz(Y) \rtimes H) \otimes \Kom(\ell^2) \), then 
  \[
    F_t \circ \kappa_E(K) = (T_t \circ (\id_{H_X} \otimes f)) \otimes L \in \Loc[G](\tilde{\rho}_X)
  \]
  and similarly for \(\kappa_E(K) \circ F_t\).
  
  In general, we apply Kasparov's stabilization theorem to embed \(E\) into \((\Cz(Y) \rtimes H) \otimes \ell^2\) such that there exists an adjointable projection \(P\) on \((\Cz(Y) \rtimes H) \otimes \ell^2\) whose image is \(E\).
  To complete the argument, observe that if \((\hat{F}_t)\) is a connection for \((T_t)\) with respect to the standard Hilbert module, then \(F_t \coloneqq P \hat{F}_t P\) yields a connection with respect to \(E\).
  
  \labelcref{item:connectionCommute}
  By the definition of the compact operators on a Hilbert module, the image of \(\kappa_E\) is the closed linear span of elements \(\mathfrak{T}_{e_1} \circ \mathfrak{T}_{e_2}^\ast = \kappa_E(| e_1 \rangle \langle e_2|)\), where \(e_1, e_2 \in E\).
  Hence the desired statement follows from the defining conditions of a connection:
  \[
    F_t \circ \mathfrak{T}_{e_1} \circ \mathfrak{T}_{e_2}^\ast \sim \mathfrak{T}_{e_1} \circ T_t  \circ \mathfrak{T}_{e_2}^\ast \sim  \mathfrak{T}_{e_1} \circ \mathfrak{T}_{e_2}^\ast \circ F_t.
  \]
  
  \labelcref{item:connectionWellDef}
  Suppose that \((T_t) \in \Cz([1,\infty), \Roe[G\times H](X \times Y))\).
  As in the previous part, we can assume that \(\kappa_E(K) = \mathfrak{T}_{e_1} \circ \mathfrak{T}_{e_2}^\ast\).
  Then 
  \[
    F_t \circ \kappa_E(K) = F_t \circ \mathfrak{T}_{e_1} \circ \mathfrak{T}_{e_2}^\ast \sim
    \mathfrak{T}_{e_1} \circ \underbrace{T_t}_{\to 0} \circ \mathfrak{T}_{e_2}^\ast \to 0,
  \]
  as required.
  If \((T_t) \in \LocVanish[G \times H](X \times Y)\), 
  then let \(\varphi \in \Cz(Y)\) and assume  \(\kappa_E(K) = \mathfrak{T}_{e_1 \varphi} \circ \mathfrak{T}_{e_2}^\ast\).
  This is justified because \(\Cz(Y)\) contains an approximate identity of \(\Cz(Y) \rtimes H\).
  Then for each \(f \in \Cz(X)\), we have
  \begin{align*}
  \tilde{\rho}_X(f) \circ F_t \circ \kappa_E(K) &= \tilde{\rho}_X(f) \circ F_t \circ \mathfrak{T}_{e_1 \varphi} \circ \mathfrak{T}_{e_2}^\ast \\
  &\sim \tilde{\rho}_X(f) \circ \mathfrak{T}_{e_1 \varphi} \circ T_t \circ \mathfrak{T}_{e_2}^\ast \\
  &= \mathfrak{T}_{e_1} \circ \underbrace{(f \otimes \varphi) T_t}_{\to 0} \circ \mathfrak{T}_{e_2}^\ast \to 0
  \end{align*}
  and hence \(F_t \circ \kappa_E(K) \in \LocVanish[G](\tilde{\rho}_X)\).
\end{proof}
We now can define \labelcref{eq:generalTopEquSlantProduct} by 
\[
  (T_t) \otimes K \mapsto [(F_t \circ \kappa_E(K))_{t \in [1,\infty)}],
\]
where we choose a connection \((F_t)\) for each \((T_t)\).
This is well-defined by \cref{lem_connections}.
In fact, it shows that this yields a well-defined \(\ast\)-homomorphism
\begin{equation}
  \Upsilon_{\LSym} \colon \frac{\Loc[G \times H](X \times Y)}{\LocVanish[G \times H](X \times Y)} \otimes_{\max} \Kom_{\Cz(Y)\rtimes H}(E)
    \to \frac{\Loc[G](\tilde{\rho}_X)}{\LocVanish[G](\tilde{\rho}_X)}.
\end{equation}
Since the \(\K\)-theory of \(\LocVanish\) vanishes, this still defines a map
\begin{equation}
  \K_p(\Loc[G \times H](X \times Y)) \otimes \K_{-q}(\Kom_{\Cz(Y)\rtimes H}(E)) \to \K_{p-q}(\Loc[G](X))
  \label{eq:generalTopEquSlantProductKTheory}
\end{equation}
by a completely analogous recipe as \cref{def_topolequislant}.
If \(E\) is a full Hilbert \(\Cz(Y) \rtimes H\)-module, then \(\K_{\ast}(\Kom_{\Cz(Y) \rtimes H}) \cong \K_\ast(\Cz(Y) \rtimes H)\).
In this case, it can be verified that \labelcref{eq:generalTopEquSlantProductKTheory} agrees with the previous definition \labelcref{def_topolequislant} by embedding \(E\) into the standard module \((\Cz(Y) \rtimes H) \otimes \ell^2\) via the Kasparov stabilization theorem and replacing \(H_Y\) by \(\ell^2(H) \otimes H_Y\).

\begin{proof}[Proof of \cref{thm_equivariantComparison_technical}]
  We start with an arbitrary countably generated Hilbert-\(\Cz(Y)\rtimes H\)-module \(E_Y\) that is full.
  Then let \(E_{H \backslash Y} \coloneqq E_Y \otimes_{\Cz(Y)\rtimes H} \mathsf{Z}_Y\) and \(\phi \colon \Kom_{\Cz(Y) \rtimes H}(E_Y) \to \Kom_{\Cz(H \backslash Y)}(E_{H \backslash Y})\) be the canonical map \(S \mapsto S \otimes \id\).
  The map \(\phi\) implements the isomorphism \labelcref{eq:KTheoryIndIso} in this setup.
  In view of the previous discussion, it suffices to prove that the following diagram commutes.
  \[
    \begin{tikzcd}
      \frac{\Loc[G \times H](X \times Y)}{\LocVanish[G \times H](X \times Y)} \otimes_{\max} \Kom_{\Cz(Y) \rtimes H}(E_Y) \ar[r, "\Upsilon_{\mathrm{L}}"] \ar[d, "\mathsf{I}_{X \times Y} \otimes \phi"]
      & \frac{\Loc[G](\tilde{\rho}_X)}{\LocVanish[G](\tilde{\rho}_X)} \ar[d, "\mathsf{I}_{X}"]\\
      \frac{\Loc(G \backslash X \times H \backslash Y)}{\LocVanish(G \backslash X \times H \backslash Y)} \otimes_{\max} \Kom_{\Cz(H \backslash Y)}(E_{H \backslash Y}) \ar[r, "\Upsilon_{\mathrm{L}}"] 
      & \frac{\Loc(\tilde{\rho}_{G \backslash X})}{\LocVanish(\tilde{\rho}_{G \backslash X})}
    \end{tikzcd}
  \]
  Note that we have a canonical identification
  \begin{align}
    E_{H\backslash Y} \otimes_{\Cz(H \backslash Y)} H_{H \backslash Y} 
    &= (E_Y \otimes_{\Cz(Y)\rtimes H} \mathsf{Z}_Y) \otimes_{\Cz(H \backslash Y)} H_{H \backslash Y} \nonumber\\
    &= E_Y \otimes_{\Cz(Y)\rtimes H} (\mathsf{Z}_Y \otimes_{\Cz(H \backslash Y)} H_{H \backslash Y}) \nonumber\\
    &= E_Y \otimes_{\Cz(Y)\rtimes H} H_Y \label{eq:identifyModulesY}
  \end{align}
  and hence
  \begin{align*}
    \mathsf{Z}_{X} \otimes_{\Cz(G \backslash X)} \tilde{H}_{G \backslash X} 
    &= (\mathsf{Z}_{X} \otimes_{\Cz(G \backslash X)} H_{G \backslash X}) \otimes (E_{H\backslash Y} \otimes_{\Cz(H \backslash Y)} H_{H \backslash Y}) \\
    &= H_X \otimes (E_Y \otimes_{\Cz(Y)\rtimes H} H_Y) = \tilde{H}_X.
  \end{align*}
  Therefore it makes sense to consider the arrows entering the lower right corner to land in the same \textCstar-algebra.

  Now let \((T_t) \in \Loc[G \times H](X \times Y)\) and \(K \in \Kom_{\Cz(Y)\rtimes H}(E_Y)\).
  Choose a pushdown \((S_t) \in \Loc(G \backslash X \times H \backslash Y)\) for \((T_t)\).
  Moreover, choose a connection \((F_t) \in \FiproLoc[G](\tilde{\rho}_X)\) for \((T_t)\) and a connection \((G_t) \in \FiproLoc(\tilde{\rho}_{G \backslash X})\) for \((S_t)\).
  To prove the desired commutativity, we need to verify that \(G_t \circ \kappa_{E_{H \backslash Y}}(\phi(K)) \) is a pushdown of \(F_t \circ \kappa_{E_Y}(K)\).
  First observe that \labelcref{eq:identifyModulesY} identifies \(\phi(K) \otimes_{\Cz(H \backslash Y)} \id\) with \(K \otimes_{\Cz(Y) \rtimes H} \id\).
  Then let \(z \in \mathsf{Z}_X\), \(w \in \mathsf{Z}_Y\), \(e \in E_Y\) and fix the following notation as in \cref{defn_LocIndIso,defn_connection}.
  \begin{align*}
    \mathfrak{T}_{z} \colon H_{G \backslash X} &\to H_X \\
    \mathfrak{T}_{w} \colon H_{H \backslash Y} &\to H_Y \\
    \mathfrak{T}_{e} \colon H_Y &\to E_Y \otimes_{\Cz(Y) \rtimes H} H_Y
  \end{align*}
  We have \(e \otimes w \in E_Y \otimes_{\Cz(Y) \rtimes H} \mathsf{Z}_Y = E_{H \backslash Y}\) and up to the identification \labelcref{eq:identifyModulesY} the following maps are equal.
  \[
    \mathfrak{T}_{e \otimes w} = \mathfrak{T}_{e} \circ \mathfrak{T}_{w} \colon H_{H \backslash Y} \to E_{H \backslash Y} \otimes_{\Cz(H \backslash Y)} H_{H \backslash Y} = E_Y \otimes_{\Cz(Y) \rtimes H} H_Y
  \]
  As before, the image of \(\kappa_{E_Y}\) is the closed linear span of operators of the form \(\mathfrak{T}_{e_1} \circ \mathfrak{T}_{e_2}^\ast\), where \(e_1, e_2 \in E_Y\).
  Thus we can assume without loss of generality that 
  \[K \otimes_{\Cz(Y) \rtimes H} \id_{H_Y} = \phi(K) \otimes_{\Cz(H \backslash Y)} \id_{H_Y} = \mathfrak{T}_{e_1} \circ \mathfrak{T}_{e_2}^\ast.\]
  Hence
  \begin{align*}
    F_t \circ \kappa_{E_Y}(K) \circ (\mathfrak{T}_{z} \otimes \id) 
    &= F_t \circ (\id \otimes \mathfrak{T}_{e_1} \circ \mathfrak{T}_{e_2}^\ast) \circ (\mathfrak{T}_z \otimes \id) \\
    &=F_t \circ  (\id \otimes \mathfrak{T}_{e_1}) \circ (\mathfrak{T}_z \otimes \mathfrak{T}_{e_2}^\ast)\\
    &\sim (\id \otimes \mathfrak{T}_{e_1}) \circ T_t \circ (\mathfrak{T}_z \otimes \mathfrak{T}_{e_2}^\ast)
  \end{align*}
  because \((F_t)\) is a connection for \((T_t)\).
  We can furthermore assume that \(e_2 = e_2^\prime \prescript{}{\Cz(Y)\rtimes H}{\langle} w^\prime \mid w \rangle\), where \(e_2^\prime \in E_Y\), \(w, w^\prime \in \mathsf{Z}_Y\) because such elements are dense.
  Then \(\mathfrak{T}_{e_2} = \mathfrak{T}_{e_2^\prime} \circ \mathfrak{T}_{w^\prime} \circ \mathfrak{T}_{w}^\ast\) and we continue 
  \begin{align*}
    (\id \otimes \mathfrak{T}_{e_1}) \circ T_t & \circ (\mathfrak{T}_z \otimes \mathfrak{T}_{e_2}^\ast) \\
    &= (\id \otimes \mathfrak{T}_{e_1}) \circ T_t \circ (\mathfrak{T}_z \otimes \mathfrak{T}_{w}) \circ (\id \otimes \mathfrak{T}_{w^\prime}^\ast \circ \mathfrak{T}_{e_2^\prime}^\ast) \\
    &\sim (\id \otimes \mathfrak{T}_{e_1}) \circ (\mathfrak{T}_z \otimes \mathfrak{T}_{w}) \circ S_t \circ  (\id \otimes \mathfrak{T}_{w^\prime}^\ast \circ \mathfrak{T}_{e_2^\prime}^\ast) \\
    &= (\mathfrak{T}_z \otimes \id) \circ (\id \otimes \underbrace{\mathfrak{T}_{e_1} \circ \mathfrak{T}_{w}}_{\mathfrak{T}_{e_1 \otimes w}}) \circ S_t \circ (\id \otimes \mathfrak{T}_{w^\prime}^\ast \circ \mathfrak{T}_{e_2^{\prime}}^\ast) \\
    &\sim (\mathfrak{T}_z \otimes \id) \circ G_t \circ (\id \otimes \mathfrak{T}_{e_1} \circ \mathfrak{T}_{w}) \circ (\id \otimes \mathfrak{T}_{w^\prime}^\ast \circ \mathfrak{T}_{e_2^{\prime}}^\ast) \\
    &= (\mathfrak{T}_z \otimes \id) \circ G_t \circ (\id \otimes \mathfrak{T}_{e_1} \circ \mathfrak{T}_{e_2}^\ast) \\
    &= (\mathfrak{T}_z \otimes \id) \circ G_t \circ (\id \otimes \phi(K) \otimes \id) \\
    &= (\mathfrak{T}_z \otimes \id) \circ G_t \circ \kappa_{E_{H \backslash Y}}(\phi(K)).
  \end{align*}
  because \((S_t)\) is a pushdown of \((T_t)\) and \((G_t)\) is a connection for \((S_t)\).
  Hence 
  \[
    F_t \circ \kappa_{E_Y}(K) \circ (\mathfrak{T}_{z} \otimes \id) \sim (\mathfrak{T}_z \otimes \id) \circ G_t \circ \kappa_{E_{H \backslash Y}}(\phi(K)).
  \]
  A completely analogous argument also proves that 
  \[
    (\mathfrak{T}_{z}^\ast \otimes \id) \circ F_t \circ \kappa_{E_Y}(K) \sim  G_t \circ \kappa_{E_{H \backslash Y}}(\phi(K)) \circ (\mathfrak{T}_z^\ast \otimes \id).
  \]
  Thus \((G_t)\) is a pushdown of \((F_t)\), as required.
\end{proof}
This finishes the proof of \cref{thm_equivariantComparison}.

\section{Injectivity of external products}
\label{sec_applications}
In this section, we deduce our injectivity results for external product maps using the machinery developed in \cref{sec_slant_products,sec_equiv_version}.
In the following statements, we fix an exact crossed product functor \(\mu\) or \(\mu = \red\) if \(H\) is exact.
Here and in the following, we will use the notation convention from \cref{rem_HRPlaceholder}.
We start with an injectivity result for the external product with a single element.
This is a direct consequence of our equivariant slant products.

\begin{thm}\label{CrossWithElementInj}
	Let $Y$ be a proper metric space of continuously bounded geometry endowed with a proper action of a countable discrete group \(H\).
	Let \(y \in \K_n^H(Y)\) be such that there exists \(\vartheta \in \K_{1-n}(\sHigCorRed Y \rtimes_{\mu} H)\) with \( \langle y, \mu^\ast_H \vartheta \rangle = 1\) \parensup{or with \(\langle y, \mu^\ast_H \vartheta \rangle \neq 0\)}.

	Then for every proper metric space \(X\) endowed with a proper action of a countable discrete group \(G\), the external product
	\[\HRPlaceholder^G_{\ast}(X) \xrightarrow{\times y} \HRPlaceholder^{G \times H}_{\ast+n}(X \times Y)\]
	is \parensup{rationally} split-injective.
\end{thm}
\begin{proof}
	It follows from \cref{cor_equivariant_composition_is_pairing} that a retraction for \(\times y\) is given by (a rational multiple of) the map
	\(
		\HRPlaceholder_{p+n}^{G \times H}(X \times Y)) \xrightarrow{\\/ \vartheta} \HRPlaceholder_{p}^G(X)
	\).
\end{proof}

If the equivariant coarse co-assembly map is rationally surjective, then we obtain rational injectivity of the entire external product map in the presence of a free action on the second factor.
\begin{thm}\label{thm_coass_surj}
Let $Y$ be a proper metric space of continuously bounded geometry endowed with a proper and free action of a countable discrete group \(H\).
Suppose that the equivariant coarse co-assembly map
\[
\mu^*_H \colon \K_{1-*}(\sHigCorRed Y \rtimes_\mu H)\to \K^{*}_H(Y)
\]
is rationally surjective.

Then for every proper metric space $X$ endowed with a proper action of a countable discrete group \(G\), the external product
\[\HRPlaceholder^G_m (X) \otimes \K_{n}^H(Y) \to \HRPlaceholder_{m+n}^{G \times H} (X \times Y)\]
is rationally injective for every \(m, n \in \Z\).
\end{thm}

\begin{proof}
  First observe that the universal coefficient theorem implies that the pairing between \(\K\)-theory and homology induces an isomorphism
  \[
    \K_n(H \backslash Y) \otimes \Q \cong \Hom(\K^n(H \backslash Y), \Q).
  \]
  Since the action of \(H\) on \(Y\) is free, it follows from \cref{defn_equivariant_pairing,thm_equivariantComparison} that the pairing between \(\K_n^H(Y)\) and \(\K^n_H(Y)\) is equivalent to the pairing between \(\K_n(H \backslash Y)\) and \(\K^n(H \backslash Y)\).
  Thus we also have an isomorphism 
  \begin{equation}
    \K_n^H(Y) \otimes \Q \cong \Hom(\K^n_H(Y), \Q).
    \label{eq:pairingNondegenerate}
  \end{equation}
  Now pick a \(\Q\)-basis $(\theta_i)_{i\in I}$ for $\K^n_H(Y) \otimes \Q$.
  By exploiting the rational surjectivity of the equivariant coarse co-assembly map and possibly multiplying the basis elements with integers, we can assume that there are elements $\vartheta_i \in \K_{1-n}(\sHigCorRed Y \rtimes_\mu H)$ with \(\mu^\ast(\vartheta_i) = \theta_i\) for all $i \in I$.
  We use the equivariant slant products with all \(\vartheta_i\) simultaneously.
  By \cref{cor_equivariant_composition_is_pairing}, this yields a commutative diagram
  \[
    \begin{tikzcd}
      \HRPlaceholder_m^{G}(X) \otimes \K_n^H(Y) \otimes \Q \ar[r, "\times"] \dar[d, "\cong"', "\id \otimes \prod_{i \in I} \langle \blank{,}\theta_i\rangle"] &
      \HRPlaceholder_{m+n}^{G \times H}(X \times Y) \otimes \Q \ar[d, "\prod_{i \in I} /\vartheta_i"]\\
      \HRPlaceholder_m^G(X) \otimes \prod_{i \in I} \Q \ar[r, hook]&
      \prod_{i \in I} (\HRPlaceholder_m^G(X) \otimes \Q).
    \end{tikzcd}
  \]
  The left vertical arrow is an isomorphism by \labelcref{eq:pairingNondegenerate} and  the lower horizontal arrow is injective for abstract reasons.
  Consequently, this proves that the external product map is injective.
\end{proof}

\begin{rem}
At first glance, one might also hope for a coarse version of the previous theorem using the definitions given in \cref{sec_compatibility_Rips}. 
That is, for a proper metric space $Y$ of bounded geometry such that the coarse co-assembly map $\mu^*\colon \K_{1-*}(\sHigCorRed Y)\to \KX^{*}(Y)$ is rationally surjective, the external product maps
\begin{alignat*}{3}
\times\colon&& \KX_m(X)&\otimes \KX_n(Y)&&\to \KX_{m+n}(X\times Y)
\\
\times\colon&& \SX_m(X)&\otimes \KX_n(X)&&\to \SX_{m+n}(X\times Y)
\\
\times\colon&& \K_m(\RoeSymbol X)&\otimes \KX_n(Y)&&\to \K_{m+n}(\RoeSymbol (X\times Y))
\end{alignat*}
should be rationally injective.\footnote{The external product $\times\colon \K_m(\RoeSymbol X)\otimes \KX_n(Y)\to \K_{m+n}(\RoeSymbol(X\times Y))$ is the external product of the Roe algebras composed with the coarsified assembly map $\KX_n(Y)\to \K_n(\RoeSymbol Y)$.}
However, there seems to be little hope.
The proof of \cref{thm_coass_surj} does not work in this situation because the pairing between $\KX_n(Y)$ and $\KX^{n}(Y)$ can be degenerate in general.
\end{rem}

\subsection{Proper metric spaces that are scaleable, combable or coarsely embeddable}

In this section we show that in many situations, where one can prove the coarse Novikov conjecture, one can also prove injectivity of external products on the non-equivariant version of the Higson--Roe sequence.
In the following, the space \(Y\) will come without a group action.
In effect, we will apply \cref{thm_coass_surj} for the case that \(H\) is the trivial group.

We refrain from recalling here the notions occuring in the following corollary (like \emph{scaleable} or \emph{combing}) because they do not appear anywhere else in this paper. The interested reader can find the relevant definitionsin the references provided in the proof below.
The corollary is stated as \cref{cor_conditionsCoAssIso_intro} in the introduction.

\begin{cor}\label{cor:conditionsCoAssIso}
Let $Y$ be either
\begin{myenuma}
\item a uniformly contractible, proper metric space of continuously bounded geometry which is scaleable,
\item a uniformly contractible, proper metric space of continuously bounded geometry which admits an expanding and coherent combing, or
\item the universal cover $\Efree H$ of the classifying space $\Bfree H$ of a group $H$, if $\Bfree H$ is a finite complex and $H$ is coarsely embeddable into a Hilbert space. \label{item:CoarseEmbedd}
\end{myenuma}

Then for every proper metric space $X$ endowed with a proper action of a countable discrete group \(G\), the external product
\[
  \HRPlaceholder_m^G(X) \otimes \K_{n}(Y) \to \HRPlaceholder^G_{m+n}(X \times Y)
\]
is rationally injective for each \(m,n \in \Z\).
\end{cor}

\begin{proof}
In either case we will reduce to \cref{thm_coass_surj}, where we take \(H\) to be the trivial group.
Then we need to show that the ordinary coarse co-assembly map \(\mu^\ast \colon  \K_{1-\ast}(\sHigCorRed Y) \to \K^\ast(Y)\) is rationally surjective.
\begin{myenuma}
\item \citeauthor{EmeMey} \cite[Corollary~8.10]{EmeMey} proved that under these assumptions on $Y$ its coarse co-assembly map is an isomorphism.
\item \citeauthor{engel_wulff} proved that under these assumptions on the space $Y$ the coarse co-assembly map $\mu^\ast\colon \K_{1-*}(\sHigCorRed Y)\to \KX^*(Y)$ is surjective \cite[Theorem~5.10]{engel_wulff}. Because uniform contractibility of $Y$ implies $\KX^*(Y) \cong \K^*(Y)$, see \cref{prop_uniformly_contractible}, the claim follows.
\item \citeauthor{EmeMey} \cite[Section~9]{EmeMey} proved that under these assumptions on $H$ its coarse co-assembly map $\mu^\ast\colon \K_{1-*}(\sHigCorRed \Efree H)\to \KX^*(\Efree H)$ is an isomorphism.
In this case, $\Efree H$ is uniformly contractible and so we get $\KX^*(H) \cong \K^*(\Efree H)$ by \cref{prop_uniformly_contractible}.
Moreover, $\Efree H$ has continuously bounded geometry.\qedhere
\end{myenuma}
\end{proof}

\subsection{Groups with a \texorpdfstring{$\gamma$}{gamma}-element}
\label{sec_rational_kuenneth}

The following result is also a consequence of \cref{thm_coass_surj} and was stated as \cref{intro_thm_struct_injective_gamma} in the introduction.

\begin{cor}\label{thm_struct_injective_gamma}
  Let $N$ be a finite aspherical complex, and assume that $H = \pi_1 N$ has a $\gamma$-element.

  Then for every proper metric space \(X\) endowed with a proper action of a countable discrete group \(G\), the external product map
  \[
  \HRPlaceholder_m^{G}(X) \otimes \K_{n}(N) \to \HRPlaceholder_{m+n}^{G \times H}(X \times \ucov{N})
  \]
   is rationally injective for each \(m,n \in \Z\).
\end{cor}

\begin{proof}
Since $N$ is assumed to be a finite aspherical complex, the group \(H\) is torsion-free and $\ucov N$ is an $H$-finite model for $\Efree H = \Eub H$.
 Hence by \cref{cor_equiv_coarse_coass_surj} we conclude that the equivariant coarse co-assembly map
\[
\mu^\ast_H\colon \K_{1-\ast}(\sHigCorRed \ucov N \rtimes_{\max} H) \to \K^{\ast}_H(\ucov{N}) \cong \K^{\ast}(N)
\]
is surjective. 
Thus we can apply \cref{thm_coass_surj} together with the isomorphism \(\K_\ast(N) \cong \K^H_\ast(\ucov{N})\).
\end{proof}

In the case of $\K$-homology the above injectivity statement follows from the Künneth formula and is therefore valid in full generality, that is, without assuming the existence of a $\gamma$-element.
But for the $\K$-theory of the reduced group \textCstar-algebras, the Künneth formula is only known to hold if $\pi_1 M$ satisfies the Baum--Connes conjecture with coefficients in any \textCstar-algebras with trivial $\pi_1 M$-action \cite{Kuenneth_BC}, which is a considerably stronger assumption than the existence of a $\gamma$-element.
We explore this in \cref{sec_full_kuenneth} below.

\subsection{Higson-essentialness and hypereuclidean manifolds}
\label{sec_hypereuc}
We restate the definition of a Higson-essential manifold which was given in \cref{defn_Higson_essential_intro}.
\begin{defn}\label{defn_Higson_essential}
  We say that a complete Riemannian \spinC-manifold \(X\) of dimension \(m\) is \emph{Higson-essential} if there exists \(\vartheta \in \K_{1-m}(\sHigCorRed X)\) such that \(\langle [\Dirac_X], \mu^\ast(\vartheta)  \rangle = 1\), where \([\Dirac_X] \in \K_m(X)\) denotes the \(\K\)-homological fundamental class of the \spinC-structure.
  If the condition is relaxed to merely \(\langle [\Dirac_X], \mu^\ast(\vartheta)  \rangle \neq 0\), then \(X\) is called \emph{rationally Higson-essential}.
\end{defn}
Start with an immediate observation.
\begin{prop}
  Let \(X\) be a complete Riemannian \spinC-manifold \(X\) of dimension \(m\) such that the coarse co-assembly map \(\mu^\ast \colon \K_{1-m}(\sHigCorRed X) \to \K^m(X)\) is \parensup{rationally} surjective.
  Then \(X\) is \parensup{rationally} Higson-essential.
\end{prop}
\begin{proof}
  This follows from the definition because there always exists a class \(\beta_X \in \K^m(X)\) such that \(\langle [\Dirac_X], \beta_X \rangle = 1\).
  Indeed, \(\beta_X\) can be taken to be the Bott generator inside a coordinate patch \(\R^m \subset X\).
\end{proof}

\begin{rem}\label{rem_surj_implies_Higson-essential}
  In particular, a complete \spinC-manifold which is also a space of a type considered in \cref{cor:conditionsCoAssIso} is Higson-essential.
\end{rem}

Higson-essentialness is an analytic condition.
We will contrast it with the following coarse geometric property which is a generalization of hypereuclidan manifolds from \cref{defn_stably_hypereuclidean_noncompact}.

\begin{defn}[{compare \cite[Definition~2.11]{BrunnbauerHanke}}]
We say a complete oriented Riemannian manifold $X$ of dimension \(m\) is \emph{coarsely hypereuclidean} if there exists a coarse map \(\varphi \colon X \to \R^m\) such that \(\varphi_\ast(\coarsify[X]) \in \HZX_m(\R^m)\) is a generator, where \(\HZX_\ast\) denotes coarse homology and \(\coarsify \colon \HZ^\lf_\ast(X) \to \HZX_\ast(X)\) is the coarsification map.
If the condition is relaxed to merely \(\varphi_\ast(\coarsify[X]) \neq 0 \in \HZX_m(\R^m)\), then \(X\) is called rationally coarsely hypereuclidean.

Moreover, we say that \(X\) is (rationally) \emph{stably coarsely hypereuclidean} if there is \(k \in \N\) such that \(X \times \R^k\) is (rationally) coarsely hypereuclidean.
\end{defn}

Since any proper Lipschitz map is coarse and the coarsification map is an isomorphism for \(\R^m\), it is immediate that (rationally, stably) hypereuclidean in the sense of \cref{defn_stably_hypereuclidean_noncompact} implies (rationally, stably) coarsely hypereuclidean.

\subsubsection{Coarsely hypereuclidean implies Higson-essential}
In this subsection, we prove that stably coarsely hypereuclidan \spinC-manifolds are Higson-essential.
To deal with the \emph{stable} aspect, we need to work with the suspension isomorphism for the stable Higson corona.
We discuss this in the following remark.
\begin{rem}[Suspension for the stable Higson corona] \label{rem:sHigCorSuspension}

	Let \(\R_- \coloneqq (-\infty, 0]\) and \(\R_+ \coloneqq [0, \infty)\).
	Then \(X \times \R = X \times \R_- \cup X \times \R_+\) is a coarsely excisive cover by closed subsets.
	We then have three pull-back diagrams of \Cstar-algebras of the form
	\[
		\xymatrix{
			\mathfrak{C}(X \times \R) \ar[r] \ar[d] & \mathfrak{C}(X \times \R_+) \ar[d]\\
			\mathfrak{C}(X \times \R_-)  \ar[r]& \mathfrak{C}(X),
		}
	\]
	where \(\mathfrak{C} \in \{ \Cz(\blank, \Kom),  \sHigComRed, \sHigCorRed\}\).
	Moreover, all maps in these diagrams are surjections.
	These properties follow from \cite[Lemmas~3.3 and 3.4]{WilletHomological}.
	Associated to each of these pullback diagrams we have a long exact Mayer--Vietoris sequence in \(\K\)-theory~\cite[Section~21.2]{blackadar}.
		As \(X \times \R_\pm\) is flasque, \(\K_\ast(\mathfrak{C}(X \times \R_\pm)) = 0\).
Thus the boundary maps in the Mayer--Vietoris sequences yield isomorphisms \(\Sigma \colon \K_\ast(\mathfrak{C}(X)) \xrightarrow{\cong} \K_{\ast-1}(\mathfrak{C}(X \times \R))\).
Applying the functorial exact sequence \(0 \to \Cz(\blank, \Kom) \to \sHigComRed(\blank) \to \sHigCorRed(\blank) \to 0\) yields an exact sequence of pullback diagrams.
Thus the co-assembly map, which is the boundary map associated to \(0 \to \Cz(\blank, \Kom) \to \sHigComRed(\blank) \to \sHigCorRed(\blank) \to 0\), commutes with the Mayer--Vietoris boundary map up to a sign.
In other words, we have the diagram
\[
	\xymatrix{
	\K^\ast(X) \ar[r]^-{\Sigma}&
	 \K^{\ast+1}(X \times \R) \\
	\K_{1-\ast}(\sHigCorRed X) \ar[r]^-{\Sigma} \ar[u]^-{\mu^\ast} & \K_{1-(\ast+1)}(\sHigCorRed (X \times \R)) \ar[u]^-{\mu^\ast}
	}
\]
which commutes up to multiplication with \(-1\).
\end{rem}
We are now ready to to prove the main theorem of this subsection.
\begin{thm}\label{hypereuclideanImpliesCoassemblyLift}
	Let \(X\) be an \(m\)-dimensional \parensup{rationally} stably coarsely hypereuclidean \spinC-manifold.
	Then \(X\) is \parensup{rationally} Higson-essential.
\end{thm}
\begin{proof}
	By assumption, we have a coarse map $\varphi\colon X \times \R^k \to \R^{m+k}$ such that \(\varphi_\ast \coarsify [X \times \R^k] = d\coarsify [\R^{m+k}]\), where \(d = 1\) \parensup{or \(d \neq 0 \in \Z\), respectively}.
  We first observe that then the same applies to the \(\K\)-homological fundamental clases, that is, 
  \[
  \varphi_\ast \coarsify [\Dirac_{X \times \R^{k}}] = d \coarsify [\Dirac_{\R^{m+k}}] \in \KX_{m+k}(\R^{m+k}).
  \]
  This follows from the Chern character.
  Indeed, 
  \[
  \ch_*([\Dirac_{X \times \R^k}]) \in \bigoplus_{i\in \N} \HZ^\lf_{m+k-2i}(X \times \R^k;\Q)
  \] 
  is the Poincaré dual of the Todd class \(\Td(X \times \IR^k)\) of \(X \times \R^k\).
  Since \(\R^{m+k}\) has coarse and locally finite homology only in degree \(m+k\), only the top degree component of  \(\ch_*([\Dirac_{X \times \R^k}])\) contributes to \(\ch_*(\varphi_\ast \coarsify [\Dirac_{X \times \R^k}])\).
  As the degree zero part of \(\Td(X \times \R^k)\) is \(1\), this component is \([X \times \R^k]\) and we conclude that
  \[\ch_*(\varphi_\ast \coarsify [\Dirac_{X \times \R^k}]) = \varphi_\ast \coarsify [X \times \R^k] = d \coarsify [\R^{m+k}] \in \HZX_{m+k}(\R^{m+k}; \Q).\]
  This also implies the desired integral equality because the transformation \(\Z \cong \HZX_\ast(\R^{m+k}) \to \HZX_\ast(\R^{m+k}; \Q) \cong \Q\) is injective.
  
  \begin{sloppypar}
	Next, we consider the following commutative diagram (up to signs).
	\[
		\xymatrix{
      \K^m(X) \ar[r]^-{\Sigma^k}_-{\cong}
      & \K^{m+k}(X \times \R^k)
      & \K^{m+k}(\R^{m+k})  \\
			\KX^m(X) \ar[r]^-{\Sigma^k}_-{\cong}  \ar[u]^-{\coarsify^\ast}
				& \KX^{m+k}(X \times \R^k) \ar[u]^-{\coarsify^\ast}
				& \KX^{m+k}(\R^{m+k}) \ar[l]_-{\varphi^\ast} \ar[u]^-{\coarsify^\ast}_-{\cong} \\
			\K_{1-m}(\sHigCorRed X) \ar[u]^-{\mu^\ast} \ar[r]^-{\Sigma^k}_-{\cong}
				& \K_{1-m-k}(\sHigCorRed (X\times\R^k)) \ar[u]^-{\mu^\ast}
				& \ar[l]_-{\varphi^\ast} \K_{1-m-k}(\sHigCorRed \R^{m+k}) \ar[u]^-{\mu^\ast}_-{\cong}
		}
	\]
  Here \(\Sigma^k\) signifies the \(k\)-fold application of the suspension ismorphism from \cref{rem:sHigCorSuspension}.
	Let \( \beta_{\R^{m+k}} \in \K^{m+k}(\R^{m+k})\) be the Bott generator and \(\vartheta_{\R^{m+k}} \in \K_{1-m-k}(\sHigCorRed \R^{m+k})\) such that \(\coarsify^\ast \mu^\ast \vartheta_{\R^{m+k}} = \beta_{\R^{m+k}}\).
	Further, we set \(\vartheta \in \K_{1-m}(\sHigCorRed X) \) to be the unique element such that \(\Sigma^k \vartheta = \varphi^\ast \vartheta_{\R^{m+k}} \).
  To prove that \(X\) is (rationally) Higson-essential, we need to verify that \(\langle [\Dirac_X], \coarsify^\ast \mu^\ast \vartheta\rangle = \pm d\).
  Indeed, \(\Sigma^k \mu^\ast \vartheta = \pm \mu^\ast \Sigma^k \vartheta = \pm \mu^\ast \varphi^\ast \vartheta_{\R^{m+k}} = \pm \varphi^\ast \mu^\ast \vartheta_{\R^{m+k}}\) and therefore
	\begin{align*}
		\langle [\Dirac_X], \coarsify^\ast \mu^\ast \vartheta\rangle
			&= \langle [\Dirac_X] \times [\Dirac_{\R^k}], \Sigma^k \coarsify^\ast \mu^\ast \vartheta \rangle \\
      &= \langle [\Dirac_X] \times [\Dirac_{\R^k}], \coarsify^\ast \Sigma^k  \mu^\ast \vartheta \rangle \\
			&= \pm \langle [\Dirac_{X \times \R^{k}}], \coarsify^\ast \varphi^\ast \mu^\ast \vartheta_{\R^{m+k}} \rangle \\
      &= \pm \langle \varphi_\ast \coarsify [\Dirac_{X \times \R^{k}}], \mu^\ast \vartheta_{\R^{m+k}} \rangle \\
      &= \pm \langle d \coarsify [\Dirac_{\R^{m+k}}], \mu^\ast \vartheta_{\R^{m+k}} \rangle \\
      &= \pm \langle d [\Dirac_{\R^{m+k}}], \coarsify^\ast \mu^\ast \vartheta_{\R^{m+k}} \rangle \\
      &= \pm \langle d [\Dirac_{\R^{m+k}}], \beta_{\R^{m+k}} \rangle = \pm d \,. \qedhere 
  \end{align*}
\end{sloppypar}
\end{proof}

\subsubsection{Injectivity of external products}
The following result explains why the notion of Higson-essentialness is useful for our purposes.
It was stated as \Cref{intro_theorem_injectivity_higson_essential} in the introduction.

\begin{cor}\label{cor_application_of_higson_essential}
	Let \(Y\) be an \(n\)-dimensional \spinC-manifold of continuously bounded geometry.
  Suppose that \(Y\) is \parensup{rationally} Higson-essential \parensup{in particular, this is satisfied if \(Y\) is \parensup{rationally} stably coarsely hypereuclidean}.
  Assume furthermore that \(Y\) is endowed with a proper action of a countable discrete group \(H\) which preserves the \spinC structure.

	Then for every proper metric space \(X\) which is endowed with a proper action of a countable discrete group \(G\), the external product
\[\HRPlaceholder_\ast^{G}(X) \xrightarrow{\times [\Dirac_Y]} \HRPlaceholder_{\ast+n}^{G \times H}(X \times Y)\]
is \parensup{rationally} split-injective.
\end{cor}
\begin{proof}
  If  \([\Dirac_Y] \in \K_n^H(Y)\) denotes the equivariant fundamental class of \(Y\), then \(\ForgetEquiv_\ast [\Dirac_Y] \in \K_n(Y)\) is the fundamental class of the underlying non-equivariant \spinC-manifold, where we used the notation from \cref{subsec_compatibility_equiv_nonequiv}.
  By assumption, there exists \(\vartheta \in \K_{1-n}(\sHigCorRed Y)\) with \(\langle \ForgetEquiv_\ast [\Dirac_Y], \mu^\ast \vartheta\rangle = d\), where \(d = 1\) (or \(d \neq 0\), respectively).
  Let \(\iota \colon \sHigCorRed Y \to \sHigCorRed Y \rtimes_\mu H\) be the map which was also considered in \cref{subsec_compatibility_equiv_nonequiv}.
  Then \cref{defn_equivariant_pairing,thm_equiKHomSlantCompatible,lem_compatibility_equiv_nonequiv} imply that
  \[
    \langle [\Dirac_Y], \mu^\ast_H \iota_\ast \vartheta \rangle = \langle \ForgetEquiv_\ast [\Dirac_Y], \mu^\ast \vartheta\rangle = d.
  \]
  Thus we can apply \cref{CrossWithElementInj}.
\end{proof}

\subsubsection{Contractible manifolds}

Assuming that the non-compact manifold is contractible (for instance, consider the universal cover of an aspherical manifold), we get the following stronger results.

\begin{prop}\label{AsphericalHypereuclideanImpliesSurj}
	Let \(Y\) be a contractible \(n\)-dimensional \spinC-manifold.
  Suppose that \(Y\) is \parensup{rationally} Higson-essential \parensup{in particular, this is satisfied if \(Y\) is \parensup{rationally} stably coarsely hypereuclidean}.

	Then the coarse co-assembly map \(\mu^\ast \colon \K_{1-\ast}(\sHigCorRed Y) \to \K^\ast(Y)\) is \parensup{rationally} surjective.
\end{prop}

\begin{proof}
	Because \(Y\) is contractible, Poincaré duality for \spinC-manifolds (see for instance~\cite[Exercise~11.8.11]{higson_roe}) implies
	\[\K^p(Y) \cong \RK_{n-p}(Y) \cong \RK_{n-p}(\ast) \cong
	\begin{cases}
		\Z & p \equiv n \mod 2, \\
		0 & \text{otherwise}.
	\end{cases}
	\]
	If \(Y\) is Higson-essential, by definition there exists a class in the image of the co-assembly map which pairs to \(1\) with \([\Dirac_Y]\).
	Such a class must generate \(\K^n(Y) \cong \Z\).
	If \(Y\) is only rationally Higson-essential, the co-assembly map is still rationally non-trivial and hence rationally surjective as the target is one-dimensional.
\end{proof}

\begin{cor}
  Let \(Y\) be a contractible \(n\)-dimensional \spinC-manifold of continuously bounded geometry.
  Suppose that \(Y\) is rationally Higson-essential \parensup{in particular, this is satisfied if \(Y\) is rationally stably coarsely hypereuclidean}.

	Then for every proper metric space $X$ endowed with a proper action of a countable discrete group \(G\), the external product
	\[\HRPlaceholder_m^G(X) \otimes \K_{p}(Y) \to \HRPlaceholder_{m+p}^G(X \times Y)\]
	is rationally injective for each \(m, p \in \Z\).
\end{cor}
\begin{proof}
	Combine \cref{AsphericalHypereuclideanImpliesSurj,thm_coass_surj}.
\end{proof}

\section{K\"{u}nneth theorems for the structure group}
\label{sec_full_kuenneth}
In this final section---which does not use the methods of the rest of the paper---we deduce a full Künneth-like theorem for the analytic structure group in the case that the Baum--Connes conjecture is satisfied.
\begin{thm}\label{thm_kuenneth}
  Let $H$ be a countable discrete group.
  Assume that $H$ is torsion-free and satisfies the Baum--Connes conjecture for all coefficient \textCstar-algebras with trivial $H$-action.\footnote{For example, $H$ could be a-T-menable \cite{higson_kasparov} or it could be hyperbolic \cite{lafforgue_BC,puschnigg_BC}.}
  
  Then for any simplicial complex $M$, the external product map
  \[ \RStrg_*^{G}(\ucov{M}) \otimes \RK_*(\Bfree H) \to \RStrg_*^{G \times H}(\ucov{M} \times \Efree H),\]
   where \(G = \pi_1 M\), is rationally an isomorphism. 
   If $\RK_*(\Bfree H)$ is torsion-free, then it is integrally an isomorphism.
\end{thm}

\begin{proof}
We consider the following commutative diagram which is part of a map between rationally exact sequences.
\begin{align*}
\mathclap{\xymatrix{
\RStrg^G_*(\ucov{M}) \otimes \RK_*(\Bfree H)   \ar[r] \ar[d] & \RK_*(M) \otimes \RK_*(\Bfree H) \ar[r] \ar[d] &  \K_*(\CstarRed G) \otimes \RK_*(\Bfree H) \ar[d]\\
\RStrg_*^{G \times H}(\ucov{M} \times \Efree H) \ar[r] & \RK_*(M \times \Bfree H) \ar[r] & \K_\ast(\CstarRed (G \times H))
}}
\end{align*}
Indeed, the top sequence is the analytic exact sequence~\labelcref{eq:HigsonRoeRepresentable} for the space $\ucov{M}$ tensored with $\RK_*(\Bfree H)$.
The functor $\blank \otimes \RK_*(\Bfree H)$ is rationally exact and hence the top sequence is rationally an exact sequence.
The lower sequence is the analytic exact sequence~\labelcref{eq:HigsonRoeRepresentable} for the space $\ucov{M} \times \Efree H$.

The middle vertical arrow is rationally an isomorphism due to the Künneth formula for $\K$-homology. Note that in general for any ring spectrum we have a Künneth spectral sequence.
But in the case of the complex $\K$-theory spectrum one can show that the spectral sequence degenerates suitably to give rise to a short exact sequence. This follows by similar arguments as presented in the remark on top of page~62 in \cite{land_diss}.

Because we assume $H$ to be torsion-free and to satisfy the Baum--Connes conjecture, the assembly map $\RK_\ast(\Bfree H) \to \K_*(\CstarRed H)$ is an isomorphism.
Also, the reduced group \textCstar-algebra $\CstarRed H$ satisfies the Künneth formula, because we assume that $H$ satisfies (the reduced version of) the Baum--Connes conjecture for all coefficient \textCstar-algebras with trivial $H$-action.\footnote{Tu \cite{tu_UCT} proved that if $H$ is amenable, then $\CstarRed H$ lies in the bootstrap class and hence satisfies the Künneth formula. That $\CstarRed H$ satisfies the Künneth formula if $H$ satisfies the Baum--Connes conjecture was proven by \citeauthor{Kuenneth_BC}~\cite{Kuenneth_BC}.}
Hence the right vertical arrow in the diagram is rationally an isomorphism.

It follows from the five lemma that the left vertical arrow must also be rationally an isomorphism.

If $\RK_*(\Bfree H)$ is torsion-free, then the functor $\blank \otimes \RK_*(\Bfree H)$ is integrally exact, that is, the top sequence in the above diagram is exact.
Furthermore, the $\operatorname{Tor}$-terms in the Künneth formulas for the middle and right vertical arrows vanish in this case and therefore these arrows are isomorphisms. Therefore the left vertical arrow is also an isomorphism by the five-lemma.
\end{proof}

\printbibliography[heading=bibintoc]

\end{document}